\documentclass{calabi}

\begin{document}

\title{Quantitative Heegaard Floer cohomology and the Calabi invariant}
\date{\today}
\author{Daniel Cristofaro-Gardiner, Vincent Humili\`ere, Cheuk Yu Mak, Sobhan Seyfaddini, Ivan Smith}
\maketitle

\begin{abstract} 
We define a new family of spectral invariants associated to certain Lagrangian links in compact and connected surfaces of any genus.
We show that our invariants recover the Calabi invariant of Hamiltonians in their limit.
As applications, we resolve several open questions  from topological surface dynamics and continuous symplectic topology: we show that the group of Hamiltonian homeomorphisms of any compact surface with (possibly empty) boundary is not simple; we extend the Calabi homomorphism to the group of Hameomorphisms constructed by Oh-M\"uller; and, we construct an infinite dimensional family of quasimorphisms on the group of area and orientation preserving homeomorphisms of the two-sphere.  

Our invariants are inspired by recent work of Polterovich and Shelukhin defining and applying spectral invariants, via orbifold Floer homology, for links composed of parallel circles in the two-sphere.  A particular feature of our work is that it avoids the orbifold setting and relies instead on `classical'  Floer homology.  This not only substantially simplifies the technical background, but seems essential for some aspects (such as the application to constructing quasimorphisms).
\end{abstract}

\tableofcontents


\section{Introduction}

\subsection{Recovering the Calabi invariant}
Let $(\Sigma, \omega)$ denote a compact and connected surface, possibly with boundary, equipped with an area-form.   When the boundary is non-empty, the group of Hamiltonian diffeomorphisms  admits  a homomorphism $$\Cal: \Ham(\Sigma, \omega) \rightarrow \R,$$ called the {\bf Calabi invariant}, defined as follows.   
Let $\theta \in \Ham(\Sigma, \omega)$.  Pick a Hamiltonian $H: [0,1]\times \Sigma \rightarrow \R$, supported in the interior of $\Sigma$,  such that $\theta = \phi^1_H$;  see Section \ref{sec:prelim_symp} for our conventions in the definition of $\Ham$. Then, 
$$\Cal(\theta) := \int_{0}^1\int_{\Sigma} H  \, \omega \; dt.$$ 
The above integral does not depend on the choice of $H$ and so $\Cal(\theta)$ is well-defined.    Moreover, 
it defines  a non-trivial group homomorphism.
For further details on the Calabi homomorphism see \cite{calabi,mcduff-salamon}.

The first goal of the present work is to recover the Calabi invariant from more modern invariants, called {\bf spectral invariants}. 
In fact, we prove a more general result for closed surfaces.
 Spectral invariants have by now a long history of applications in symplectic topology, see for example \cite{viterbo, Schwarz, Oh99, Entov-Polterovich, Oh05, Sey-displaced,  LZ18, volume, airie, twoinfinity, CGHS20}.  For our work here, what is critical is that  the techniques of continuous symplectic topology allow us to define spectral invariants for area-preserving homeomorphisms, and we will see several applications below.

To state our result about recovering Calabi, define a {\bf Lagrangian link}  
$\ul{L} \subset \Sigma$ to be a smooth embedding of finitely many pairwise disjoint circles (see Figure \ref{fig:lagrangian-link} below).  We emphasize, because it contrasts the setup for many other works about Floer theory on surfaces, that the individual components of the link are not required to be Floer theoretically non-trivial; for example, they can be small contractible curves.  Whenever $\ul{L}$ satisfies a certain monotonicity assumption, see Definition \ref{def:monotone_link}, we define a {\bf link spectral invariant} $c_{\ul{L}}: C^{\infty}([0,1] \times \Sigma, \omega) \to \R$.  The properties of the invariants $c_{\ul{L}}$ are summarized in Theorem~\ref{t:spectral} below.  We have the following result for suitable sequences of Lagrangian links which always exist and which we refer to as {\bf equidistributed} links, see Section \ref{sec:Calabi_property} for the precise definition.    
A sequence of links being equidistributed in particular implies that the number of contractible components diverges to infinity,  whilst their diameters in a fixed metric tend to zero; we therefore think of such links as `probing the small-scale geometry' 
of the surface. 

\begin{theo}[Calabi property]\label{theo:recovering_Calabi}
Let $\ul{L}^m$ be a sequence of equidistributed Lagrangian links  in a closed symplectic surface $(\Sigma, \omega)$.  Then, for any $H\in C^\infty([0,1]\times \Sigma)$ we have
$$\lim_{m \to \infty} c_{\ul{L}^m}(H) = \int_0^1 \int_\Sigma H_t \; \omega \;dt.$$
\end{theo}

\begin{remark}\label{rem:volume_ECH}
The Calabi property is reminiscent of a property conjectured by Hutchings for spectral invariants defined using Periodic Floer homology, see \cite[Rmk.\ 1.12]{CGHS20}, which was verified in \cite{CGHS20} for monotone twist maps.  We were partly inspired to think about it because of this conjecture.  Hutchings' conjecture was in turn inspired by a Volume Property for spectral invariants defined using Embedded Contact Homology proved in \cite{volume} that has had various applications, see for example \cite{airie, twoinfinity}.   On the other hand, the above Calabi property is different from a property with the same name appearing in the works of Entov-Polterovich \cite{Entov-Polterovich} on Calabi quasimorphisms or the recent paper of Polterovich-Shelukhin \cite{Pol-Shel21}.  What these papers refer to as the Calabi property is equivalent to the Support control property of our Theorem \ref{t:spectral}. 
\end{remark}

We can think of a result like Theorem~\ref{theo:recovering_Calabi} as asserting that we have ``enough" spectral invariants to recover classical data.  We now explain several applications.

\subsection{The algebraic structure of the group of area-preserving homeomorphisms}
\label{sec:app}

Our first applications resolve two old questions from topological  surface dynamics that have been key motivating problems in the development of continuous symplectic topology.  The ability to recover Calabi is central for both proofs.

\subsubsection*{Hamiltonian homeomorphisms }
Let $\Homeo_0(\Sigma,\omega)$ denote the identity component in the group of homeomorphisms of $\Sigma$ which preserve the measure induced by $\omega$ and coincide with the identity near the boundary of $\Sigma$, if the boundary is non-empty. We say $\varphi \in  \Homeo_0(\Sigma, \omega)$  is a {\bf Hamiltonian homeomorphism} if it can be written as a uniform limit of Hamiltonian diffeomorphisms.  The set of all  such homeomorphisms is denoted by $\overline{\Ham}(\Sigma, \omega)$;  this is a normal subgroup of $\Homeo_0(\Sigma,\omega)$.  Hamiltonian homeomorphisms have been studied extensively in the surface dynamics community; see, for example, \cite{Matsumoto, LeCalvez05, LeCalvez06}.\footnote{We remark that when $\Sigma = S^2$, $\overline{\Ham}$ is the group of area and orientation preserving homeomorphisms, and when $\Sigma = D^2$, it  is the group of area preserving homeomorphisms that are the identity near the boundary. } 

\sloppy There exists a homomorphism out of $\Homeo_0(\Sigma,\omega)$, called the {\bf mass-flow} homomorphism, introduced by Fathi \cite{fathi}, whose kernel is $\overline{\Ham}(\Sigma, \omega)$.
The normal subgroup  $\overline{\Ham}(\Sigma, \omega)$ is proper when $\Sigma$ is different from the disc or the sphere.  In the 1970s, Fathi asked in \cite[Section 7]{fathi} if $\overline{\Ham}(\Sigma, \omega)$ is a simple group; in higher dimensions, one can still define mass-flow and Fathi showed \cite[Thm. 7.6]{fathi} that its kernel is always simple, under a technical assumption on the manifold which always holds when the manifold is smooth.  When $\Sigma$ is a surface with genus $0$, Fathi's question was answered in \cite{CGHS20,CGHS21}; however, the higher genus case has remained open. 
Although the details of mass-flow are not needed for our work, we recall some facts about it in Section \ref{sec:mass-flow}.

By using our new spectral invariants, we can answer Fathi's question in full generality:

\begin{theo}\label{theo:non_simplicity}
$\overline{\Ham}(\Sigma, \omega)$ is not simple.
\end{theo}

Theorem~\ref{theo:non_simplicity} generalizes the aforementioned results of \cite{CGHS20, CGHS21} proving this result in the genus zero case.  Our proof is logically independent of these works.
To prove the theorem, following \cite{CGHS20, CGHS21} we construct a normal subgroup $\FHomeo(\Sigma, \omega)$, called the group of {\bf finite energy homeomorphisms}, and we prove that it is proper, see Section \ref{sec:thm1}.  The group FHomeo is inspired by Hofer geometry, and one can define Hofer's metric on it, see \cite[Sec. 5.3]{CGHS21}.   For another proof in the genus $0$ case, see \cite{Pol-Shel21}.

The group $\FHomeo(\Sigma, \omega)$ contains the commutator subgroup of $\overline{\Ham}(\Sigma, \omega)$, see Proposition \ref{prop:Hameo_subgp}, hence we learn from our main result that $\overline{\Ham}(\Sigma, \omega)$ is not perfect, either.

\subsubsection*{Extending the Calabi invariant }

One would like to understand more about the algebraic structure of $\overline{\Ham}(\Sigma, \omega)$ beyond the simplicity question.  Recall that $\Ham(\Sigma,\omega)$ denotes the subgroup of Hamiltonian diffeomorphisms and suppose now that the boundary of $\Sigma$ is non-empty. 

A question of Fathi from the 1970s \cite[Section 7]{fathi} asks if $\Cal$ admits an extension to $\overline{\Ham}(\D, \omega)$.
An illuminating discussion  by Ghys of this question appears in \cite[Section 2]{Ghys_ICM};
it follows from results of Gambaudo-Ghys \cite{Gambaudo-Ghys} 
and Fathi \cite{Fathi-Calabi} 
that Calabi is a topological invariant of Hamiltonian diffeomorphisms, i.e.\ if $f,g \in \Ham(\Sigma, \omega)$ are conjugate by some $h \in \Homeo_0(\Sigma, \omega)$, then $\Cal(f) = \Cal(g)$.  Hence, 
 it seems natural to try and extend Calabi to $\overline{\Ham}(\Sigma, \omega)$, or at least to a proper normal subgroup.\footnote{Fathi proves in \cite{Fathi-Calabi} that $\Cal$ extends to Lipschitz area-preserving homeomorphisms.  These, however, do not form a normal subgroup.} Our proof of Theorem~\ref{theo:non_simplicity} involves constructing an ``infinite twist" Hamiltonian homeomorphism which, heuristically, has infinite Calabi invariant, so our interest in what follows will be extending the Calabi homomorphism to a proper normal subgroup rather than the full group.  

There is a later conjecture of Fathi about what an appropriate normal subgroup for the purpose of extending Calabi
might be.  In the article \cite{muller-oh}, Oh and M\"uller introduced a normal subgroup $\Hameo(\Sigma, \omega),$ called the group of {\bf Hameomorphisms} of $\Sigma,$ and whose definition we review in \ref{sec:hameo}; the idea of the definition is that these are elements of $\overline{\Ham}(\Sigma, \omega)$ that have naturally associated Hamiltonians.  The group $\Hameo(\Sigma, \omega)$ is contained in $\FHomeo(\Sigma, \omega)$, see Proposition \ref{prop:Hameo_subgp}, and so our proof of Theorem~\ref{theo:non_simplicity} shows that it is proper.  The aforementioned conjecture of Fathi is that the Calabi invariant admits an extension to $\Hameo(\Sigma, \omega)$ when $\Sigma$ is the disc; see \cite[Conj.\ 6.1]{Oh10}.  We prove this for any $\Sigma$ with non-empty boundary.

\begin{theo}\label{theo:extending_calabi}
The Calabi homomorphism admits an extension to a homomorphism from the group $\Hameo(\Sigma, \omega)$ to the real line.
\end{theo}

Theorem \ref{theo:extending_calabi} implies that $\Hameo(\Sigma, \omega)$ is 
neither simple nor perfect, 
when $\partial \Sigma \neq \emptyset$; we do not know whether or not the kernel of Calabi on $\Hameo$ is simple.

\begin{remark}
\begin{enumerate}

\item Theorem \ref{theo:extending_calabi} implies that $\FHomeo(\Sigma, \omega)$ is not simple either, when $\partial \Sigma \neq \emptyset$.  This is because by Proposition \ref{prop:Hameo_subgp}, $\Hameo(\Sigma, \omega)$ is a normal subgroup of $\FHomeo(\Sigma, \omega)$: we do not know if $\Hameo(\Sigma, \omega)$ forms a proper subgroup, but if not then we can conclude that the Calabi invariant extends to $\FHomeo(\Sigma, \omega)$ and so it cannot be simple.  By the same reasoning, Theorem~\ref{theo:extending_calabi} implies Theorem~\ref{theo:non_simplicity} in the case where $\partial \Sigma \neq \emptyset.$
\item We also do not know much about the quotient $\overline{\Ham}(\Sigma, \omega) / \Hameo(\Sigma,\omega),$ although we do know that it is abelian, by Proposition \ref{prop:Hameo_subgp},  and that it  contains a copy of $\R$; see Remark \ref{rem:quotient}.  
\end{enumerate}
\end{remark}

\subsection{Quasimorphisms on the sphere}

We now explain one more application of our theory in the case $\Sigma = S^2$.  Strictly speaking, this does not use the Calabi property, although it does use the abundance of our new spectral invariants.     

Recall that a {\bf homogeneous quasimorphism} on a group $G$ is a map $\mu :G \rightarrow \R$ such that
\begin{enumerate}
\item  $\mu(g^n) = n \mu(g)$, for all $g \in G$, $n \in \mathbb{Z}$;
\item  there exists a constant $D(\mu) \ge 0$, called the {\bf defect} of $\mu$, with the property that $\vert \mu(gh) - \mu(g) - \mu(h) \vert \leq D(\mu).$
\end{enumerate}

Returning now to the algebraic structure of $\Homeo_0(S^2,\omega)$, note that the vector space of all homogeneous quasimorphisms of a group is an important algebraic invariant of it; however, it has previously been unknown whether $\Homeo_0(S^2,\omega)$ has any nontrivial, i.e.\ non-zero, homogeneous quasimorphisms at all.  

\begin{theo}\label{corol:quasi_homeo}
The space of homogeneous quasimorphsisms on $\Homeo_0(S^2, \omega)$ is infinite dimensional. 
\end{theo} 

The same statement was very recently proven for $\Homeo_0(\Sigma)$ where $\Sigma$ is a surface of positive genus, see \cite{Bowden}, but in contrast the group $\Homeo_0(S^2)$ has no non-trivial homogeneous quasimorphisms as we review in Example~\ref{ex:comm} below.  We also note that the space of all homogeneous quasimorphisms is infinite dimensional for  $\Homeo_0(\Sigma, \omega)$  when the genus of $\Sigma$ is at least one, see \cite[Thm.\ 1.2]{EPP}.  The existence of our quasimorphisms has various implications, as the following illustrates.

\begin{example}
\label{ex:comm}

Recall that the {\bf commutator length} $cl$ of an element $g$ in the commutator subgroup of a group is the smallest number of commutators required to write $g$ as a product.  The {\bf stable commutator length} is defined\footnote{To use a phrase from \cite{calegari}, we can think of the commutator length as a kind of algebraic analogue of the number of handles, and we refer the reader to \cite{calegari} for further discussion.} by $scl(g) := \text{lim}_{n \to \infty} \frac{cl(g^n)}{n}.$  
It follows immediately from the existence of a non-trivial homogenous quasimorphism that the commutator length and the stable commutator length are both unbounded \cite{Ca09}\footnote{The homogeneous quasimorphisms we construct are not homomorphisms because they restrict to non-trivial quasimorphisms on $\Ham(S^2, \omega)$ which, being a simple group, does not admit any non-trivial homomorphisms. 
}.   In stark contrast to this,  Tsuboi \cite{Tsuboi} has shown that $cl(g) =1$ for any $ g \in \Homeo_0(S^n)\setminus \{\id\}$.\footnote{$cl(g) =1$ for $g \in \Homeo_0(S^1)\setminus \{\id\}$ was established earlier by Eisenbud, Hirsch and Neumann \cite{Eisenbud}.}

Moreover, we prove in Proposition \ref{prop:notcont} that $scl$ is unbounded in any $C^0$ neighborhood of the identity.  This contrasts \cite[Thm.\ 1.5]{Bowden} on $C^0$ continuity of scl in the non-conservative setting; see Section \ref{sec:comm}. 
\end{example}
We also explain an application to fragmentation norms in \ref{sec:comm} below.

In the course of our proof of Theorem~\ref{corol:quasi_homeo}, we answer a question  posed by Entov, Polterovich and Py \cite[Question 5.2]{EPP}, which was partly motivated by the desire to obtain a result like Theorem~\ref{corol:quasi_homeo}, see Remark~\ref{rmk:mot}; the question also appears as Problem 23 in the McDuff-Salamon list of open problems \cite[Ch. 14]{mcduff-salamon}.    The question refers in part to the Hofer metric, defined in Section \ref{sec:hameo}.  

\begin{question}
  \label{que:mcsal}
  Does the group $\Ham(S^2, \omega)$ admit any homogeneous quasimorphism which is continuous with respect to the $C^0$ topology?  If yes, can it be made Lipschitz with respect the Hofer metric?\footnote{The analogue of Question \ref{que:mcsal} for the $2$ and $4$ (complex) dimensional quadrics was recently settled in the affirmative by Kawamoto \cite{kawamoto2}.}
  \end{question}

 \begin{theo}\label{theo:quasimorphisms}
The space consisting of homogeneous quasimorphsisms on $\Ham(S^2, \omega)$  which are continuous with respect to the $C^0$ topology and Lipschitz with respect to the Hofer metric is infinite dimensional. 
\end{theo}


In fact, our quasimorphisms satisfy a simple asymptotic formula which can be used to prove that the Calabi property, from Theorem \ref{theo:recovering_Calabi}, holds for more general links on the sphere; see Proposition~\ref{prop:quasicalabi}.

\begin{remark}
In contrast, $\Ham(S^2, \omega)$ does not admit any non-trivial homomorphisms to $\R$ since it is simple \cite{Banyaga}.  As for $\Homeo_0(S^2,\omega)$, when this paper first appeared it
had been an open question whether it admits any non-trivial homomorphisms to $\R$;  
a straightforward modification of the argument in \cite[Cor. 2.5]{CGHS20} shows that any such homomorphism could not be $C^0$ continuous.  We later showed in \cite{CGHMSS22}, using the quasimorphisms constructed here, that $\Homeo_0(S^2,\omega)$ does admit non-trivial homomorphisms to $\R$.
\end{remark}

\begin{remark}
\label{rmk:mot}
As alluded to above, the motivation for the first part of Question~\ref{que:mcsal} is closely connected to our Theorem~\ref{corol:quasi_homeo}: indeed, a result from Entov, Polterovich and Py \cite[Prop.\ 1.4]{EPP} implies that any continuous homogeneous quasimorphism on $\Ham(S^2,\omega)$ would extend to give such a quasimorphism on $\Homeo_0(S^2,\omega)$.  As for the second part of the question, this is tuned to applications in Hofer geometry and $C^0$ symplectic topology.  For example, it was very recently shown in \cite{CGHS21, Pol-Shel21} that $\Ham(S^2,\omega)$ is not quasi-isometric to $\R$, thereby settling what is known as the Kapovich-Polterovich question \cite[Prob. 21]{mcduff-salamon}; prior to \cite{CGHS21,Pol-Shel21}, it was shown in \cite{EPP} that an affirmative answer to the second question in Question~\ref{que:mcsal} would also settle the Kapovich-Polterovich question.
\end{remark}

\subsection{Quantitative Heegard Floer cohomology and link spectral invariants}\label{sec:quantitative_HF}

We now explain the main tool that we use to prove the aforementioned results, which involves studying Floer theory for Lagrangian links by working in the symmetric product.  The key idea is that if $\ul{L} = \sqcup_{j=1}^k L_j \subset \Sigma$ is a Lagrangian link, it defines an embedded  torus $\Sym(\ul{L}) \subset \Sym^k(\Sigma)$ which is Lagrangian for appropriate symplectic forms.  The Lagrangian $\Sym(\ul{L})$ is the image of $\prod_j L_j \subset \Sigma^k$ under the quotient map $\Sigma^k \to \Sym^k(\Sigma)$, and is embedded since $\prod_j L_j$ lies away from the diagonal. (We recall that the symmetric product of a Riemann surface is naturally smooth; for instance, $\Sym^k(\mathbb{P}^1) \cong \mathbb{P}^k$, with the isomorphism given by taking a collection of points to the coefficients of the homogeneous polynomial with that zero-set, cf. Section \ref{Sec:coords_symmetric_product}.)  The Floer theory of the link $\ul{L}$ on $\Sigma$ splits as a direct sum over the different components, with contractible components having vanishing Floer cohomology. By contrast, under suitable monotonicity hypotheses, the single Lagrangian $\Sym(\ul{L})$ has non-vanishing Floer cohomology which is accordingly `sensitive' to all the components of $\ul{L}$.  A very brief review of Lagrangian Floer theory and the particular non-vanishing criterion we use is given in Remark \ref{rmk:Floer_survey}.
\medskip

Some brief historical remarks: Lagrangian links in four-manifolds were studied using Floer theory in symmetric product orbifolds in previous work \cite{MS19}, via a computation of the low order terms in the disc potential. Such an approach encounters difficulties in this setting because a virtual dimension count is not sufficient to exclude high genus curves with non-generic branched covering data  from obstructing the orbifold-Floer theory.  We proceed instead using `classical' Floer theory on the symmetric product of a surface, but computing the corresponding disc potential completely. This not only substantially simplifies the technical background for our argument, but seems essential for some aspects (such as the application to constructing quasimorphisms).
\medskip

Let $\Sigma$ be a closed genus $g$ surface equipped with a symplectic form $\omega$.  Consider a Lagrangian link (or simply a link)  $\ul{L}=\cup_{i=1}^k L_i$ consisting of $k$ pairwise-disjoint circles on $\Sigma$, with the property that 
$\Sigma \setminus \ul{L}$ consists of planar domains $B_j^{\circ}$, with $1 \leq j \leq s$, whose closures $B_j \subset \Sigma$ are also planar; throughout the rest of the paper we will only consider links satisfying this planarity assumption (See Figure \ref{fig:lagrangian-link}).

\begin{figure}[h!]
\begin{center}
\includegraphics[scale=0.6]{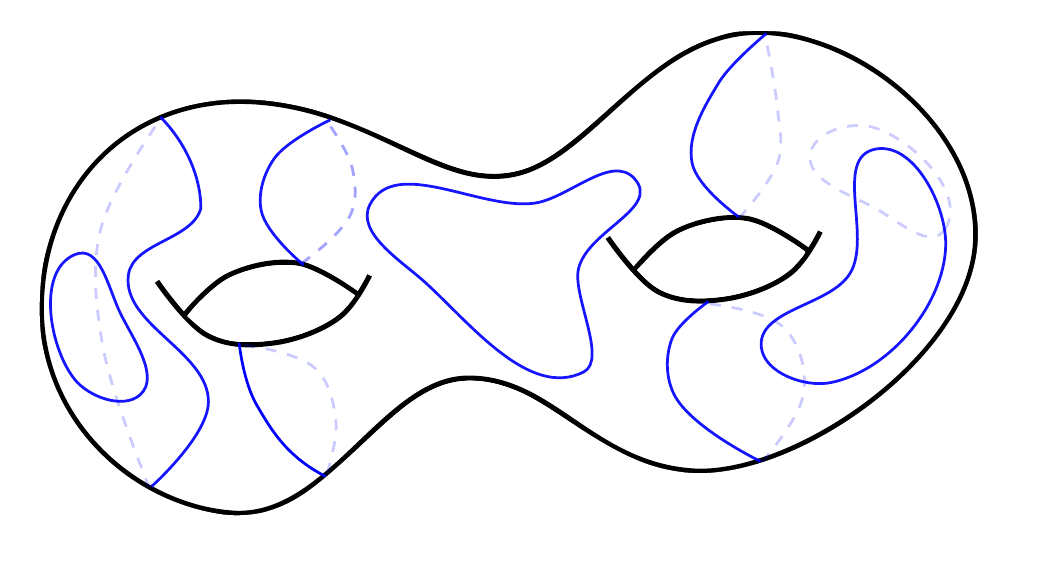}
\caption{In blue, an example of a Lagrangian link with $k=9$ components on a surface $\Sigma$ of genus $2$.}
\label{fig:lagrangian-link}
\end{center}
\end{figure}

Given a link $\ul{L}$, we denote by $\tau_j$ the number of boundary components of $B_j$. Since the Euler characteristic of a planar domain $D$ with $\tau_D$ boundary components is $2-\tau_D$, the Euler characteristic of $\Sigma$ is $2-2g=\sum_{j=1}^s(2-\tau_j)=2s-2k$, and hence $s=k-g+1$. 
Finally, for $1\leq j\leq s$, let $A_j$ denote the $\omega$-area of $B_j$.

\begin{definition}\label{def:monotone_link}
Let $\ul{L}$ be a Lagrangian link satisfying the above planarity assumption. We call  $\ul{L}$ {\bf monotone} if there exists  $\eta \in \R_{\geq 0}$ such that 
\begin{equation} \label{eqn:independent_of_j}
2\eta(\tau_j-1)+A_j
\end{equation}
is independent of $j$, for $j \in \{1,\dots,s\}$. We will use the terminology $\eta$-monotone when we need to specify the value of $\eta$. We refer to the quantity $2\eta(\tau_j-1)+A_j$ as the {\bf monotonicity constant} of $\ul{L}$.\footnote{Our terminology is motivated  by Lemma \ref{l:unobstructed}.}
\end{definition}

\medskip  

We will write $H_t$ for a time-dependent Hamiltonian function $H: [0,1] \times \Sigma \to \R$; it defines a point of the universal cover $\widetilde{\Ham}(\Sigma,\omega)$.  A Hamiltonian $H_t$ is said to be mean-normalized if $\int_\Sigma \, H_t \; \omega = 0$ for all $t \in [0,1]$. Given Hamiltonians $H, H'$ we define $(H\#H')_t(x) = H_t(x) + H'_t((\phi^t_H)^{-1}(x))$, which generates the Hamiltonian flow $\phi^t_H \circ \phi^t_{H'}$.  We refer the reader to Section \ref{sec:prelim_symp} for more details on our notations and conventions.

\begin{theo}\label{t:spectral}
For every monotone Lagrangian link  $\ul{L}=\cup_{i=1}^k L_i$ there exists a  link spectral invariant $$c_{\ul{L}}: C^{\infty}([0,1]\times \Sigma,\omega) \to \mathbb{R}$$ satisfying the following properties.
\begin{itemize}
\item (Spectrality) for any $H$, $c_{\ul{L}}(H)$ lies in the spectrum $\Spec(H:\ul{L})$ (see Definition \ref{d:spectrum} and \eqref{eqn:normalise});

\item (Hofer Lipschitz) for any $H, H'$, 
$$\int_{0}^1 \min  (H_t - H'_t) dt \leq c_{\ul{L}}(H) - c_{\ul{L}}(H') \leq \int_{0}^1 \max\,  (H_t - H'_t) dt;$$
\item (Monotonicity) if $H_t \leq H'_t$ then $c_{\ul{L}}(H) \leq c_{\ul{L}}(H')$;
\item (Lagrangian control) if $H_t|_{L_i} = s_i(t)$ for each $i$, then $$c_{\ul{L}}(H) = \frac{1}{k} \, \sum_{i=1}^k \int s_i(t) dt;$$
moreover for any $H$, $$ \frac{1}{k} \sum_{i=1}^{k} \int_0^1 \min_{L_i} H_t \, dt \leq c_{\ul{L}}(H) \leq \frac{1}{k} \sum_{i=1}^{k} \int_0^1 \max_{L_i} H_t \, dt;$$

\item (Support control) if $\mathrm{supp}(H_t) \subset \Sigma \backslash \cup_j L_j$, then $c_{\ul{L}}(H) =0$;

\item (Subadditivity) $c_{\ul{L}}(H \# H') \leq c_{\ul{L}}(H) + c_{\ul{L}}(H')$;

\item (Homotopy invariance) if $H, H'$ are mean-normalized and determine the same point of the universal cover $\widetilde{\Ham}(\Sigma,\omega)$, then $c_{\ul{L}}(H) = c_{\ul{L}}(H')$;
\item (Shift) $c_{\ul{L}}(H + s(t)) = c_{\ul{L}}(H) + \int_0^1 s(t)\, dt$.
\end{itemize}
\end{theo}

We prove this theorem in Section \ref{sec:proof-spectral}.  The spectral invariant $c_{\ul{L}}$ is defined in Equation \eqref{eqn:normalise}.

\begin{remark}
\label{rmk:polshel}
The idea of looking for spectral invariants suitable for our applications through Lagrangian links was inspired by the recent work of Polterovich and Shelukhin \cite{Pol-Shel21}: they prove a similar result for certain classes of links in $S^2$, consisting of parallel circles, in \cite[Thm. F]{Pol-Shel21} and demonstrate many applications.  
\end{remark}

The above theorem yields spectral invariants for Hamiltonians.  We will explain how to use this result to define spectral invariants for Hamiltonian diffeomorphisms in \ref{sec:specdiffhom}.  To prove our results we will also need spectral invariants for Hamiltonian homeomorphisms.   We will do this in \ref{sec:specdiffhom} as well.

\medskip

In Section \ref{sec:quasimorphisms}, we consider the case $\Sigma = S^2$ and  introduce maps $\mu_{\ul{L}}: \Ham(S^2, \omega) \rightarrow \R$  obtained from homogenization of the link spectral invariant $c_{\ul{L}}$; see Equation \eqref{eqn:homogenization}. The $\mu_{\ul{L}}$ are homogeneous quasimorphisms which inherit some of the properties listed above.  It is with these quasimorphisms that we prove Theorems \ref{corol:quasi_homeo} and \ref{theo:quasimorphisms}.

\subsubsection*{Context for Theorem \ref{t:spectral}}
\label{sec:context}

We briefly discuss the ideas entering into the proof of Theorem \ref{t:spectral}. 
Following an insight from \cite{MS19}, although some of the individual components $L_j$ are Floer-theoretically trivial in $\Sigma$, the link $\ul{L} = \sqcup_j L_j$ defines a Lagrangian submanifold $\Sym(\ul{L})$ of the symmetric product $X = \Sym^k(\Sigma)$ which may be non-trivial.  A Hamiltonian function $H: [0,1] \times \Sigma \to \R$ determines canonically a function $\Sym(H):[0,1]\times \Sym^k(\Sigma) \to \R$. Although this is only Lipschitz continuous across the diagonal, the fact that $\Sym(\ul{L})$ lies away from the diagonal makes it possible to work with (modified versions of) these Hamiltonians unproblematically. 
\medskip

The spectral invariant $c_{\ul{L}}$ is constructed using Lagrangian Floer cohomology of $\Sym(\ul{L})$  in $X$, which can be viewed as a `quantitative' version of the Heegaard Floer cohomology for links from \cite{OS:AGT2008}, cf. Remarks \ref{rmk:link_Floer} and \ref{rmk:link_Floer_2}. This quantitative version counts essentially the same holomorphic discs as in Heegaard Floer theory, but we keep track of holonomy contributions (working with local systems), and of intersection numbers of holomorphic discs with the diagonal.
The parameter $\eta\in \R_{\geq 0}$ of Theorem \ref{t:spectral} plays the role of a bulk deformation; when the assumption \eqref{eqn:independent_of_j} of Definition \ref{def:monotone_link} is satisfied, our variant of Lagrangian Floer cohomology is both Hamiltonian-invariant and non-zero. To prove the non-vanishing of Floer cohomology, we show that for certain links $\ul{L} \subset \PP^1$ the symmetric product Lagrangian $\Sym(\ul{L})$ is smoothly isotopic to a Clifford-type torus contained in a small ball (Corollary \ref{Cor:Heegaard_is_Clifford}), and use that isotopy to control the holomorphic discs with boundary on $\Sym(\ul{L})$ and to compute its disc potential (in the sense of \cite{Cho-Oh, Charest-Woodward}, 
 see Proposition \ref{p:SymL=Clifford}, \ref{p:SymL=Clifford2}).  A combination of the tautological correspondence, relating discs in the symmetric product $\Sym^k(\Sigma)$ with holomorphic maps of branched covers of the disc to $\Sigma$, together with embeddings of the planar domains in $\Sigma \backslash \ul{L}$ into $\PP^1$, allows us to reduce the general computation of the disc potential to this special case (Theorem \ref{t:curvature}).  Once Floer cohomology of $\Sym(\ul{L})$ is defined and non-trivial, the construction and properties of the spectral invariant closely follow the usual arguments \cite{FOOOspectral, LZ18} with only minor modifications. We remark that, in contrast to \cite{MS19}, this paper does not use orbifold Floer cohomology and does not require virtual perturbation techniques.
\medskip

\begin{remark}
When $g=0$ or $\eta=0$, the arguments can be simplified by working with spherically monotone symplectic forms on $X$, with respect to which $\Sym(\ul{L})$ is a monotone Lagrangian. (See Lemma \ref{rmk:monotone_Lagrangian} and Remark  \ref{rmk:floer_is_usual_floer} as well as Section \ref{sec:link=monotone}).
In this case, the spectral invariant we define coincides with the classical monotone Lagrangian spectral invariant associated to $\Sym(\ul{L})$ in $X$ with an appropriate symplectic form (see Lemma \ref{lem:spectral_stabilises}).  

The above allows us to prove Corollary  \ref{cor:requested_estimate}  establishing an inequality between our link spectral invariants and the Hamiltonian Floer spectral invariants of $\Sym(H)$.  With the help of this inequality, we prove that our link spectral invariants yield quasimorphisms in the $g=0$ case.  
\end{remark}

\subsection*{Organization of the paper}  

In Section \ref{sec:prelim}, we set our notation,  introduce our groups of homeomorphisms on surfaces and recall Fathi's mass flow homomorphism. In Section \ref{sec:non_simp_Calabi}, we use the properties of spectral invariants stated in Theorem \ref{t:spectral} to prove the Calabi property (Theorem \ref{theo:recovering_Calabi}), non-simplicity of the group of Hamiltonian homeomorphisms (Theorem \ref{theo:non_simplicity}) and the extension of the Calabi homomorphism to Hameomoprhisms
 (Theorem \ref{theo:extending_calabi}). In Section  \ref{sec:heeg-tori-cliff} we study pseudo-holomorphic discs with boundary on $\Sym(\ul L)$, which allows us to compute the disc potential function of $\Sym(\ul L)$ in Section \ref{s:potential}. This is used in Section \ref{s:QHF} to show that the relevant Floer cohomology is well-defined and non-vanishing. We also define our spectral invariants and prove Theorem \ref{t:spectral} in Section \ref{sec:proof-spectral}. Finally, we prove our results on quasimorphisms in Section \ref{sec:quasimorphisms}, and our results on commutator and fragmentation lengths in Section \ref{sec:comm}.

\subsection*{Acknowledgments} 
C.M. thanks the organizers of  the `Symplectic Zoominar' for the opportunity to speak in the seminar, where this collaboration was initiated. We thank Fr\'ed\'eric Le Roux  for helpful conversations about Section \ref{sec:comm} and Yusuke Kawamoto for helpful discussions about Lemma \ref{rem:upperbound-quasimorphisms} and Section \ref{sec:quasimorphisms}.  The authors are grateful to the anonymous referees for their helpful comments and suggestions.

We thank Ibrahim Trifa for pointing out an error, which has now been corrected, in the statement and proof of Theorem \ref{theo:quasi_independence} (ii).

D.C.G. is partially supported by NSF grant DMS-1711976 and an Institute for Advanced Study von Neumann fellowship.  
V.H. is partially supported by the grant ``Microlocal'' ANR-15-CE40-0007 from Agence Nationale de la Recherche.
C.M. is supported by Simons Collaboration on Homological Mirror Symmetry. 
S.S. is supported by ERC Starting grant number 851701.
I.S. was partially supported by Fellowship EP/N01815X/1 from the Engineering and Physical Sciences Research Council, U.K.

\section{Preliminaries}\label{sec:prelim}
In this section we introduce parts of our notation and review some necessary background. 

\subsection{Recollections}\label{sec:prelim_symp}

Let $(M,\omega)$ be a symplectic manifold.   We denote by $C^{\infty}([0,1]\times M)$ the set of time-varying Hamiltonians that vanish near the boundary when $M$ has non-empty boundary. 
Our convention is such that the (time-varying) Hamiltonian vector field associated to $H$ is defined by $ \omega(X_{H_t},\cdot) = d H_t.$ 
The homotopy class of a Hamiltonian path $\{\phi^t_H: 0\leq t \leq 1\}$ determines an element of the universal cover  $\widetilde{\Ham}(M, \omega)$.  In the case of a surface  $\Sigma \neq S^2$, the fundamental group of $\Ham$ is trivial and so $\widetilde{\Ham} = \Ham$; see \cite[Sec.\ 7.2]{Polterovich2001}.   The fundamental group of $\Ham(S^2, \omega)$ is $\Z/2\Z$ and so $\widetilde{\Ham}(S^2, \omega)$ is a two-fold covering of $\Ham(S^2, \omega)$.  

\medskip

\subsection{Hameomorphisms and finite energy homeomorphisms}\label{sec:hameo}
Denote by $C^{0}([0,1]\times M)$ the set of continuous time-dependent functions on $M$ that vanish near the boundary if $\partial M \neq \emptyset$.   The energy, or the {\bf Hofer norm}, of $H \in C^0([0,1] \times  M )$ is defined by the quantity \[ \| H \|_{(1, \infty)} = \int_0^1 \left( \max_{x \in  M } H_t - \min_{x \in  M } H_t \right) dt.\]  
The {\bf Hofer distance}  between $\varphi, \psi \in \Ham( M , \omega)$ is defined by 
\begin{equation}\label{eqn:def_Hofer_dist}
d_H(\varphi, \psi) : =  \inf \lbrace \| H \|_{(1,\infty)} : \varphi\psi^{-1} = \phi^1_H \rbrace.
\end{equation} 
This is a bi-invariant distance on  $\Ham( M , \omega)$; see \cite{Hofer-metric, Lalonde-McDuff, Polterovich2001}.

\begin{definition}\label{def:hameo}
An element $\phi \in \overline{\Ham}( M , \omega)$ is a {\bf finite energy homeomorphism} if there exists a sequence of smooth Hamiltonians $H_i \in C^{\infty}([0,1]\times  M )$ such that 
 $$\phi^1_{H_i} \xrightarrow{C^0} \phi, \text{ with } \| H_i \|_{(1, \infty)} \le C$$
 for some constant $C$.  An element $\phi \in \overline{\Ham}( M , \omega)$ is called a {\bf Hameomorphism} if there exists a continuous $H \in C^{0}([0,1] \times  M )$ such that 
 $$\phi^1_{H_i} \xrightarrow{C^0} \phi, \text{ and } \| H- H_i \|_{(1, \infty)} \to 0.$$ 
   The set of all finite energy homeomorphisms is denoted by $\FHomeo( M ,\omega)$ and the set of all Hameomorphisms is denoted by $\Hameo( M , \omega)$.\footnote{Oh and M\"uller use the terminology \emph{Hamiltonian homeomorphisms} for the elements of $\Hameo_c( M , \omega)$.  We have chosen to avoid this terminology because in the surface dynamics literature it is commonly used for elements of $\overline{\Ham}( M , \omega)$.}  
\end{definition}

There is an inclusion $\Hameo( M ,\omega) \subset \FHomeo( M ,\omega)$.  
  
  \begin{prop}\label{prop:Hameo_subgp}
  The groups $\Hameo( M , \omega)$ and $\FHomeo( M ,\omega)$ satisfy the following properties.
  \begin{enumerate}[(i)]
  \item They are both normal subgroups of  $\Homeo_0( M , \omega)$;
  \item $\Hameo( M , \omega)$  is a normal subgroup of $\FHomeo( M ,\omega)$;
 \item  If $M$ is a compact surface, they both contain the commutator subgroup of $\Homeo_0( M , \omega)$.
 \end{enumerate}
  \end{prop}
  \begin{proof}  The fact that  $\Hameo( M , \omega)$ is a normal subgroup of $\Homeo_0( M , \omega)$ is proven in \cite{muller-oh}; the same statement for $\FHomeo( M , \omega)$ is proven in \cite[Prop.\ 2.1]{CGHS20}, in the case where $ M $ is the disc;  the same argument generalizes, in a straightforward way, to any $ M $.  This proves the first item. 
  
  The second item  follows from the first and the  inclusion $\Hameo( M ,\omega) \subset \FHomeo( M ,\omega)$.  
  
  The third item follows from a general argument, involving fragmentation techniques \cite{Epstein, Higman, fathi}, which proves that any normal subgroup of $\Homeo_0( M , \omega)$ contains the commutator subgroup $[\Homeo_0( M , \omega), \Homeo_0( M , \omega)]$.  A proof of this in the case where $ M=D^2 $ is presented in \cite[Prop.\ 2.2]{CGHS20}; the argument therein generalizes, in a straightforward way, to any $ M $.
  \end{proof}
  
   We end this section with the observation that $ \phi \in \Homeo_0( M , \omega)$ is a finite energy homeomorphism (resp.\ Hameomorphism)  if it can be written as the $C^0$ limit of a sequence $\phi_i \in \Ham( M , \omega)$ which is bounded (resp.\  Cauchy) in Hofer's distance.
   
   \subsection{The mass-flow and flux homomorphisms}\label{sec:mass-flow}
Let $M$ denote a manifold equipped with a volume form $\omega$ and denote by $\Homeo_0(M, \omega)$ the identity component in the group of volume-preserving homeomorphisms of $M$ that are the identity near $\partial M$.  In \cite{fathi}, Fathi constructs the {\bf mass-flow homomorphism}
$$\mathcal{F}: \Homeo_0(M, \omega) \rightarrow H_1(M)/ \Gamma,$$ 
mentioned above, where $H_1(M)$ denotes the first homology group of $M$ with coefficients in $\R$ and $\Gamma$ is a discrete subgroup of $H_1(M)$ whose definition we will not need here. Clearly, $\Homeo_0(M, \omega)$ is not simple when the mass-flow homomorphism is non-trivial.  This is indeed the case when $M$ is a closed surface other than the sphere.
As we explained in \ref{sec:app}, Fathi proved that $\mathrm{ker}(\mathcal{F})$ is simple if the dimension of $M$ is at least three.     

   \medskip
   
   For the convenience of the reader, we recall here a (symplectic) description of the mass-flow homomorphism in the case of surfaces;  we will be very brief as the precise definition of the mass-flow homomorphism is not needed for our purposes in this article.  
   
   Denote by $\Diff_0(\Sigma, \omega)$ the identity component in the group of area-preserving diffeomorphisms $\Sigma$ that are the identity near the boundary if $\partial \Sigma \neq \emptyset$.   There is a well-known homomorphism, called {\bf flux}, 
   $$\mathrm{Flux}: \Diff_0(\Sigma, \omega) \rightarrow H^1(\Sigma) / \Gamma,$$
 where $H^1(\Sigma)$ denotes the first cohomology group of $\Sigma$ with coefficients in $\R$ and $\Gamma\subset H^1(\Sigma)$ is a discrete subgroup; see \cite{mcduff-salamon} for the precise definition.  The kernel of this homomorphism is $\Ham(\Sigma, \omega)$.   It can be shown that, in the case of surfaces, the flux homomorphism extends continuously with respect to the $C^0$ topology to yield a homomorphism
 $$\mathrm{Flux}: \Homeo_0(\Sigma, \omega) \rightarrow H^1(\Sigma) / \Gamma,$$
which coincides with the  mass-flow homomorphism $\mathcal{F}: \Homeo_0(\Sigma, \omega) \rightarrow H_1(\Sigma)/ \Gamma$, after applying Poincar\'e duality.  As we said above, its kernel, whose non-simplicity we establish in this paper, is exactly the group of Hamiltonian homeomorphisms $\overline{\Ham}(\Sigma, \omega)$.   
   
In dimensions greater than $2$, the mass-flow homorphism can be described similarly in terms of the Poincar\'e dual of the volume flux homomorphism.

\section{Non-simplicity and the extension of Calabi}\label{sec:non_simp_Calabi}

Here we assume Theorem~\ref{t:spectral} and establish our applications to non-simplicity of surface transformation groups and the extension of the Calabi invariant.   Theorem~\ref{t:spectral} will be proven  in the subsequent sections.

\medskip

\subsection{The Calabi property}\label{sec:Calabi_property}
We begin by defining equidistributed sequences of Lagrangian links and  prove Theorem~\ref{theo:recovering_Calabi}.

Throughout this section, we fix a Riemannian metric $d$ on the surface $\Sigma$ and let $\omega$ be the associated area form.  Define the diameter of a Lagrangian link  $\ul{L} = \cup_{i=1}^k L_i$ to be the maximum of the diameters of the contractible components of $\ul{L}$; we will denote it by $\mathrm{diam}(\ul{L})$.

We call a sequence of Lagrangian links $\ul{L}^m$ {\bf equidistributed} if 
\begin{enumerate}[(i)]
\item $\mathrm{diam}(\ul{L}^m) \to 0$;
\item the number of non-contractible components of $\ul{L}^m$ is bounded above by a number $N$ independent of $m$; 
\item the contractible components of each $\ul{L}^m$ are not nested:
more precisely, each such circle bounds a unique disc of diameter no more than $\mathrm{diam}(\ul{L}^m) $ and we require these discs to be disjoint;
\item each $\ul{L}^m$ is monotone, in the sense of Definition \ref{def:monotone_link}, for some $\eta$ which may depend on $m$.
\end{enumerate}

\begin{figure}[h!]
\centering
\includegraphics[scale=0.6]{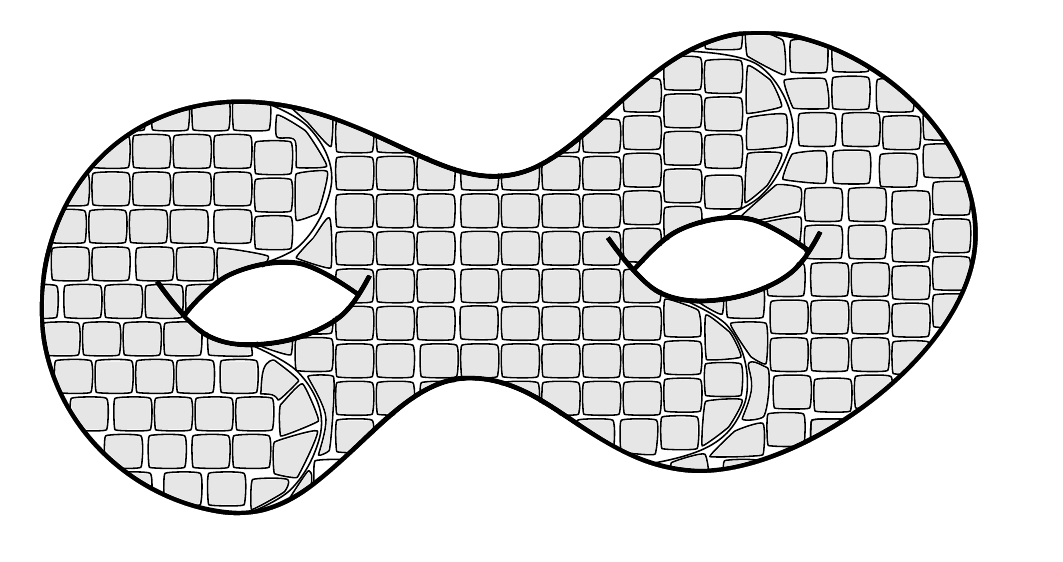}
\caption{A typical example of a link $\ul {L}^m$ for $m$ large in an equidistributed sequence. Here, $\Sigma$ has genus 2, there are 4 non-contractible components in  $\ul {L}^m$ (in blue). The disc components in $\Sigma\setminus \ul {L}^m$ are colored in grey.}
\label{fig:equidistributed}
\end{figure}

Note that any disc associated to a contractible component of $\ul L^m$ as in (iii) must be a connected component of $\Sigma \setminus \ul{L}^m$: indeed, if it contained a component of $\ul L^m$ then this component would have to be contractible and then the disc associated with it would violate the uniqueness property in (iii).  It also follows from (iv) that all these discs have equal area. We denote this common area by $\alpha_m$. Note that the other components of $\Sigma\setminus \ul L^m$ all have area smaller than or equal to $\alpha_m$.

It is straightforward to check that equidistributed sequences of Lagrangian links exist; see Figure \ref{fig:equidistributed}.

\begin{example}\label{ex:equidistrib-on-S2} Let $\eta_m$ be a sequence of real numbers such that
  \begin{equation}
  \eta_m<\frac1{2m(m-1)}\label{eq:equidistrib-on-S2}
  \end{equation}
  for all $m$. Then, there is an equidistributed sequence of $\eta_m$-monotone links $\ul L^m$ on $S^2$.

  Indeed, for each $m$, one can take $\ul L^m$ to be the boundaries of a collection of $m$ pairwise disjoint discs of equal area $\lambda=\frac1{m+1}+2\eta_m\frac{m-1}{m+1}$. The complement of these discs then has area $1-m\lambda$, which is positive by (\ref{eq:equidistrib-on-S2}).
\end{example}

\begin{proof}[Proof of Theorem~\ref{theo:recovering_Calabi}]
We now show that the spectral invariants of an equidistributed sequence of links satisfy the Calabi property.  

 We will suppose throughout the proof that $\int_\Sigma \omega =1$.
 Denote by $L_1, \ldots, L_{k_m}$ the contractible components in $\ul{L}^m$.  These bound closed and pairwise disjoint discs $B_1, \ldots, B_{k_m}$ associated via (iii) above. 
 
 Now fix $\eps > 0$.  Then, since $\mathrm{diam}(\ul L^m)\to 0$, for sufficiently large $m$ we can find a smooth Hamiltonian $G_m$ such that 
 \[ G_m|_{B_i} = s_i(t), \quad \text{max} | H - G_m | \le \eps,\]
 where each $s_i: [0,1] \to \R.$  We have that
 \[\left|\int_0^1 \int_\Sigma  H \; \omega\, dt\,   - c_{\ul{L}^m}(H) \right| \leq  \] 
 \[ \left| \int_0^1 \int_\Sigma  H  - {G_m} \; \omega\, dt\, \right| +  \left| \int_0^1 \int_\Sigma {G_m} \; \omega\, dt\, - c_{\ul{L}^m}({G_m}) \right| +  
\left| c_{\ul{L}^m}({G_m}) - c_{\ul{L}^m}(H) \right|\]
and so we must bound the three terms from the previous line.  

The term $\left| \int_0^1 \int_\Sigma  H  - {G_m} \; \omega\, dt\, \right|$ is bounded by $\eps$ because $\max| H -{G_m} | \leq \eps$ and $\mathrm{Area}(\Sigma) = 1$.  Similarly, we have 
 $\left| c_{\ul{L}^m}({G_m}) - c_{\ul{L}^m}(H) \right| \leq \varepsilon$ by the Hofer Lipschitz property from Theorem \ref{t:spectral}. 
 
 To bound the final term, use the Lagrangian control property of Theorem \ref{t:spectral} to get 
 \[ c_{\ul{L}^m}({G_m})  =   \frac{1}{k_m+\ell_m} \sum_{i=1}^{k_m}\int_0^1 s_i(t) \, dt + E_m ,\]
 where $\ell_m$ is the number of non-contractible components of $\ul L^m$ and $E_m$ satisfies
 \[ 
    \frac{\ell_m}{k_m +\ell_m} (\min {H}-\eps)\leq\frac{\ell_m}{k_m +\ell_m} \min {G_m}  \leq E_m \leq \frac{\ell_m}{k_m +\ell_m} \max {G_m}  \leq \frac{\ell_m}{k_m +\ell_m} (\max H +\eps)
   . \]  
   In particular, since $\ell_m$ is bounded, $E_m$ converges to $0$ as $k_m$ goes to $\infty$.
 
 Now,  noting that $ \int_0^1 s_i(t) dt = \frac1{\alpha_m}\int_0^1\int_{B_i} {G_m} \, \omega\,dt$,  because $\mathrm{Area}(B_i)= \alpha_m$, we can rewrite the above as
 \begin{align*}
  c_{\ul{L}^m}({G_m})  = \frac1{\alpha_m(k_m+\ell_m)}\sum_{i=1}^{k_m}\int_0^1 \int_{B_i} {G_m}\;\omega \, dt +E_m = \frac1{\alpha_m(k_m+\ell_m)}\int_0^1 \int_{\Sigma \setminus C_m} {G_m}\; \omega \; dt +E_m,
 \end{align*}
 where $C_m$ denotes the complement  $C_m := \Sigma \setminus \cup_{i=1}^{k_m} B_i$.  We claim that
 \begin{equation}
 \label{eqn:desired}
  \text{lim}_{m \to \infty} \frac{1}{\alpha_m(k_m + \ell_m)} = 1, \quad \text{lim}_{m \to \infty} \text{area}(C_m) = 0, \quad \text{lim}_{m \to \infty} k_m = \infty;
  \end{equation}
from this, it follows by the third limit that $E_m$ converges to zero in view of the above, and then from the first two limits that \[\left|c_{\ul{L}}({G_m})-\int_0^1 \int_\Sigma {G_m} \; \omega\, dt\right|\leq \eps\]
 for $m$ large enough, as desired.
 
 It remains to show \eqref{eqn:desired}.  

We claim the inequality 
\begin{equation}\label{eqn:area_inequalities}
\frac1{k_m}\geq  \alpha_m\geq \frac{1}{k_m +2N + 1}.
 \end{equation}
Recall that $N$ is the bound on $\ell_m$  which exists since the sequence $\ul L^m$ is equidistributed. The first inequality here is immediate.  To see the second, consider
 the surface $C'_m$ given by removing the noncontractible components of $\ul{L}^m$ from $C_m$.  Then, a coarse bound is that $C'_m$ has at most $2N+1$ components, and so $C_m$ satisfies
\[  \mathrm{area}(C_m)\leq (2N+1)\alpha_m .\]
Using that $\mathrm{area}(C_m)+\sum \mathrm{area}(B_j)=1$, we can now deduce \eqref{eqn:area_inequalities}.

To finish the proof of \eqref{eqn:desired}, since $\mathrm{diam}(\ul L^m)\to 0$, we have $\alpha_m\to 0$ which, in combination with the inequality immediately above, gives the second limit in \eqref{eqn:desired}; it also gives in combination with \eqref{eqn:area_inequalities}, the third limit.  The first limit in \eqref{eqn:desired} now follows from \eqref{eqn:area_inequalities}, since $\ell_m$ is bounded.  
\end{proof}

\subsection{Link spectral invariants for Hamiltonian diffeomorphisms and   homeomorphisms}
\label{sec:specdiffhom}
Theorem \ref{t:spectral} yields link spectral invariants for Hamiltonians.  To prove our results we will also need to define these invariants for Hamiltonian diffeomorphisms and homeomorphisms.

We begin by defining our invariants for Hamiltonian diffeomorphisms.  Suppose that $\Sigma$ is a closed surface and let $\ul{L}$ be a monotone Lagrangian link in $\Sigma$.  Given $\tilde \varphi$, an element in the universal cover $ \widetilde{\Ham}(\Sigma, \omega)$, we pick a mean normalized Hamiltonian $H$ whose flow represents $\tilde \varphi $.  Then, we define 

\begin{equation}\label{eqn:def_diffeos_univcover}
c_{\ul{L}}(\tilde\varphi):= c_{\ul{L}}(H).
\end{equation}
This is well-defined by the homotopy invariance property from Theorem \ref{t:spectral}.  When $\Sigma \neq S^2$, this yields a well-defined map

\begin{equation}\label{eqn:def_diffeos}
c_{\ul{L}} : \Ham(\Sigma, \omega) \rightarrow \R, 
\end{equation}
because $\Ham(\Sigma, \omega)$ is simply connected.

\medskip

For clarity of exposition, we will suppose that $\Sigma$ has positive genus throughout the rest of Section \ref{sec:non_simp_Calabi}; we will see below that this suffices to prove Theorems \ref{theo:non_simplicity} and \ref{theo:extending_calabi}.

The spectral invariant $c_{\ul{L}}: \Ham(\Sigma, \omega) \rightarrow \R$ inherits appropriately reformulated versions of the properties listed in Theorem \ref{t:spectral}.  We list the following properties which will be used below.  For $\phi, \psi \in \Ham(\Sigma, \omega)$ we have 
\begin{enumerate}
\item (Hofer Lipschitz) $\vert c_{\ul{L}}(\varphi) - c_{\ul{L}}(\psi) \vert \leq d_{H}(\varphi, \psi)$, where $d_H$ is the Hofer distance defined in \eqref{eqn:def_Hofer_dist}.

\item (Triangle inequality) $c_{\ul{L}}(\phi \psi) \le c_{\ul{L}}(\phi) + c_{\ul{L}}(\psi)$.
\end{enumerate}

We now turn to defining invariants of homeomorphisms.  An individual $c_{\ul{L}}$ is not in general $C^0$-continuous, as the following example shows.

\begin{example}
\label{ex:notcont}
Let $D$ be a disc that does not meet $\ul{L}$ and let $\varphi$ be supported in $D$.  Then, by the Shift and Support control properties from Theorem~\ref{t:spectral}, we have that 
\[ c_{\ul{L}} = - \Cal(\varphi).\]
Now it is known that $\Cal$ is not $C^0$-continuous.  For example, identify $D$ with a disc of radius one centered at the origin in $\R^2$, equipped with an area form, and take a sequence of Hamiltonians $H_i$ that are compactly supported in discs $D_i$ centered at the origin with radius $1/i$, such that $\Cal(\phi^1_{H_i}) = 1$.  Then the maps $\phi^1_{H_i}$ are converging in $C^0$ to the identity, which has Calabi invariant $0$.   

On the other hand, if we consider a difference of spectral invariants $c_{\ul{L}} - c_{\ul{L'}}$ and $D$ is disjoint from $\ul{L}$ and $\ul{L'}$, then $c_{\ul{L}} - c_{\ul{L'}}$ vanishes on any $\varphi$ supported in $D$.  In fact, we will see in Proposition \ref{prop:C0-cont1} below that this difference is continuous on $\Ham(\Sigma,\omega)$.
\end{example}

We now state the result that allows us to define invariants for homeomorphisms.  The notation $d_{C^0}$ in the proposition stands for the $C^0$ distance which is defined to be  
$$\displaystyle d_{C^0}(\varphi,\psi) = \sup_{x\in \Sigma} d(\varphi(x), \psi(x) ),$$ 
where $d$ is a Riemannian distance on $\Sigma$.

\begin{prop}\label{prop:C0-cont1}
Let $\ul{L}, \ul{L}'$ be monotone Lagrangian links.
The mapping $\Ham(\Sigma, \omega) \rightarrow \R$ defined via 
$$ \varphi \mapsto c_{\ul{L}}(\varphi) - c_{\ul{L'}}(\varphi)$$
is uniformly continuous with respect to $d_{C^0}$.  Consequently, it extends continuously to $\overline{\Ham}(\Sigma, \omega)$.
\end{prop}

\medskip

To treat surfaces with boundary, we will need a variant of Proposition  \ref{prop:C0-cont1}.  Let $\Sigma_0$ be a compact surface with boundary contained in a closed surface $\Sigma$.  Then, by the above discussion, any monotone Lagrangian link $\ul{L}$ in $\Sigma$,  yields a spectral invariant 
$$ c_{\ul{L}} : \Ham(\Sigma_0, \omega) \rightarrow \R$$
obtained from restricting $ c_{\ul{L}}$ to $\Ham(\Sigma_0, \omega)\subset \Ham(\Sigma, \omega)$.  

\begin{prop}\label{prop:C0-cont2}
Let $\ul{L}$ be a monotone Lagrangian link.
The mapping $\Ham(\Sigma_0, \omega) \rightarrow \R$ defined via 
\begin{equation}
\label{eqn:boundarycase}
\varphi \mapsto c_{\ul{L}}(\varphi) + \Cal(\varphi)
\end{equation}
is uniformly continuous with respect to $d_{C^0}$.  Consequently, it extends continuously to $\overline{\Ham}(\Sigma, \omega)$.
\end{prop}

 Note that $c_{\ul{L}}(\varphi) + \Cal(\varphi)$ corresponds to the value of $c_{\ul{L}}(H)$ where $H$ is any Hamiltonian generating $\varphi$ whose support is included in the interior of $\Sigma_0$.
\medskip

  The proofs of the above results follow from standard arguments from $C^0$ symplectic topology; see \cite{Sey13, CGHS20, CGHS21, Pol-Shel21}.  We will now prove these results.

  \begin{proof}[Proof of Proposition \ref{prop:C0-cont1}]
 Define $\zeta: \Ham(\Sigma, \omega) \rightarrow \R$ by
  \[\zeta(\varphi) = c_{\ul{L}}(\varphi) -  c_{\ul{L'}}(\varphi).\] 
  We need to prove that $\zeta$ is uniformly continuous with respect to the $C^0$ distance.

Let $\eps>0$ and fix a closed disc $B\subset \Sigma \setminus (\ul{L} \cup \ul{L'})$. 
By\footnote{The lemma is stated for $\Sigma = S^2$, but the argument works just as well for the case of general $\Sigma$.} \cite[Lemma 3.11]{CGHS21}, 
there exists a real number $\delta>0$ such that  for any  $g\in \Ham(\Sigma, \omega)$ satisfying $d_{C^0}(g,\id)<\delta$, there exists $h\in\Ham(\Sigma, \omega)$ with support in $B$ and
\begin{equation*}
  d_H(g,h)\leq \eps.
\end{equation*}

  Let $\phi_1, \phi_2\in\Ham(\Sigma, \omega)$ be such that $d_{C^0}(\phi_1,\phi_2)<\delta$. We will prove that $|\zeta(\phi_1)-\zeta(\phi_2)|\leq 2\eps$ and this will conclude our proof.

  Since $d_{C^0}(\phi_1\phi_2^{-1}, \id) = d_{C^0}(\phi_1, \phi_2)\leq \delta$, we may pick $h\in\Ham(\Sigma, \omega)$ supported in $B$ and such that
  \begin{equation}
  d_{H}(\phi_1\phi_2^{-1},h)\leq \eps.\label{eq:hofer-g-h}
\end{equation}
We now claim that
\begin{equation}
  \label{eq:c_L(h)=}
   c_{\ul{L}}(h)=-c_{\ul{L}}(h^{-1})=c_{\ul{L'}}(h)=-c_{\ul{L'}}(h^{-1}).
 \end{equation}
Indeed, this follows from the Lagrangian control property of Theorem~\ref{t:spectral}, since we can find a mean normalized Hamiltonian $H$ for $h$ such that $H_t$ is constant
in the complement of $B$, and so $h^{-1}$ has a mean normalized Hamiltonian equal to $-H$
in the complement of $B$.

Now observe that 
\begin{align}
\label{eq:ineq-C0-continuity}
c_{\ul{L}}(\phi_1)= c_{\ul{L}}(\phi_1\phi_2^{-1}\phi_2) \leq  c_{\ul{L}}(\phi_1\phi_2^{-1} h^{-1}) + c_{\ul{L}}(h \phi_2) \nonumber \\
 \leq \eps + c_{\ul{L}}(h \phi_2)  \leq c_{\ul{L}}(h)+c_{\ul{L}}(\phi_2)+\eps.   
\end{align} 
Here, the first inequality holds by the Triangle inequality  property from above; the second holds by the Hofer Lipschitz property combined with \eqref{eq:hofer-g-h}; and the third holds by again applying the Triangle inequality.
 
Similarly,
  \begin{align*}
   c_{\ul{L'}}(\phi_2)= c_{\ul{L'}}((\phi_1\phi_2^{-1})^{-1}\phi_1)
                     \leq c_{\ul{L'}}(h^{-1}\phi_1)+\eps
   \leq c_{\ul{L'}}(h^{-1})+c_{\ul{L'}}(\phi_1)+\eps.   
 \end{align*}
 The above inequalities together with (\ref{eq:c_L(h)=}) give
 \begin{align*}
   \zeta(\phi_1) &= c_{\ul{L}}(\phi_1)-c_{\ul{L'}}(\phi_1)\\
                &\leq c_{\ul{L}}(h)+c_{\ul{L}}(\phi_2)+\eps+c_{\ul{L'}}(h^{-1})-c_{\ul{L'}}(\phi_2)+\eps\\
   &=\zeta(\phi_2) + 2\eps.
 \end{align*}
 Switching the roles of $\phi_1$ and $\phi_2$, we obtain  
$|\zeta(\phi_1)-\zeta(\phi_2)| \leq 2\eps$, which shows that $\zeta$ is uniformly continuous.
\end{proof}

\begin{proof}[Proof of Proposition \ref{prop:C0-cont2}] As in the previous proof, we start by letting $\eps>0$ and fix a closed $B\subset \Sigma_0\setminus(L\cup L')$. We then follow step by step the same argument until we arrive at Inequality \ref{eq:ineq-C0-continuity}:
  \begin{align*}  c_{\ul{L}}(\phi_1)\leq c_{\ul{L}}(h)+c_{\ul{L}}(\phi_2)+\eps.  
  \end{align*}
  Since the Calabi homomorphism is 1-Lipschitz with respect to Hofer's distance, inequality (\ref{eq:hofer-g-h}) yields
  \begin{equation*}
    \Cal(\phi_1)\leq \Cal(\phi_2)+\Cal(h)+\eps.
  \end{equation*}
  Now, by the Shift property from Theorem~\ref{t:spectral}, $c_{\ul{L}}(h)=-\Cal(h)$, as can be seen by choosing a Hamiltonian for $h$ that vanishes outside $B$ and then mean normalizing.  Thus we obtain from the two previous inequalities:
  \[c_{\ul{L}}(\phi_1)+\Cal(\phi_1)\leq c_{\ul{L}}(\phi_2)+\Cal(\phi_2)+2\eps.\]
We conclude by switching the roles of $\phi_1$ and $\phi_2$ as in the proof of Proposition \ref{prop:C0-cont1}.
\end{proof}

\subsection{Infinite twists on positive genus surfaces}
\label{sec:thm1}

We can now prove Theorem \ref{theo:non_simplicity} which states that $\overline{\Ham}$ is not simple.

\begin{proof}[Proof of Theorem \ref{theo:non_simplicity}]
We showed in Proposition~\ref{prop:Hameo_subgp} that $\FHomeo$ is a normal subgroup of $\overline{\Ham}$.  It remains to show that it is proper.  To do this, we adapt the strategy from \cite[Thm. 1.7]{CGHS20}, namely we construct an example of a Hamiltonian homeomorphism that does not have finite energy.

We first consider the case where $\Sigma$ is closed.  Let $\ul{L}^m$ be an  equidistributed sequence of Lagrangian links.  Define  $\zeta_m : \Ham(\Sigma, \omega) \rightarrow \R$ by
$$\zeta_m(\varphi) = c_{\ul{L}^m}(\varphi) - c_{\ul{L}^1}(\varphi).$$   By Proposition \ref{prop:C0-cont1}, $\zeta_m$ admits a continuous extension to $\overline{\Ham}(\Sigma, \omega)$.

We now claim that if $\phi \in \FHomeo$, then $\zeta_m(\phi)$ remains bounded as $m$ varies.  To see this, let $\phi_i = \varphi^1_{H_i}$ be a sequence of diffeomorphisms converging to $\phi$ such that the $H_i$ are mean normalized and have Hofer norm bounded by $C$.  Then by the Hofer Lipschitz property from Theorem \ref{t:spectral}, applied with $H' = 0$, we have that the $\zeta_m(\phi_i)$ are also bounded by $C$.  Hence, by continuity, the $\zeta_m(\phi)$ are bounded as well.  

Next, let  $D  \subset \Sigma \setminus \cup_{i=1}^{k_1}L_i^1$ be a smoothly embedded closed disc, which we identify with the disc of radius $R$ in $\R^2$ centered at the origin with the  area form $\omega = \frac{1}{2\pi} r dr \wedge d\theta$.   We now define an ``infinite twist" homeomorphism $\psi$ supported in $D$ as follows.  Let $(\theta,r)$ denote polar coordinates.  Let $f: (0,R] \to \R$ be a smooth function which vanishes near $R$, is decreasing, and satisfies
\begin{equation}
  \label{eqn:integral2} 
  \int^1_0 r^3f(r)\, dr = \infty.
\end{equation}
We now define $\psi$ by $\psi(0) = 0$ and
\begin{equation}\label{eqn:inf_twist} 
\psi(r,\theta) = (r, \theta + 2 \pi f(r))
\end{equation}
for $r > 0$.\footnote {The area-preserving map $\psi$ can be seen as the time-1 map of the Hamiltonian flow of $F(r, \theta) =  \int_r^R sf(s) ds$.  Indeed, $F$ defines a smooth Hamiltonian flow of $D$ in the complement of the origin which extends continuously to a non-smooth flow on $D$. } The heuristic behind the condition \eqref{eqn:integral2} is that it forces $\psi$ to have ``infinite Calabi invariant".  Indeed, if $f$ was defined on the closed interval $[0,R]$, then $\psi$ would be a Hamiltonian diffeomorphism with Calabi invariant $ \int^1_0 r^3f(r)\, dr$.

We now claim that $\psi$ is a Hamiltonian homeomorphism with the property that $\zeta_m(\psi)$ diverges as $m$ varies.  By \cite[Lem. 1.14]{CGHS20}\footnote{\cite[Lem. 1.14]{CGHS20} uses the condition $\int^1_0\int^1_r s f(s)\, ds\, r\, dr = \infty$, but this is equivalent to \eqref{eqn:integral2} since $\int^1_0\int^1_r s f(s)\, ds\, r\, dr = \frac12\int^1_0 r^3f(r)\, dr$ by integration by parts.}, there are Hamiltonians $F_i$, compactly supported in the interior of $D$, with the following properties:
\begin{enumerate}
\item The sequence $\psi^1_{F_i}$ converges in $C^0$ to $\psi$.
\item $F_i \le F_{i+1}.$
\item $\text{lim}_{i \to \infty} \int^1_0 \int_{\Sigma} F_i \omega = \infty.$
\end{enumerate}
By the first property above, $\psi$ is a Hamiltonian homeomorphism.  We now  apply several properties from Theorem \ref{t:spectral}.  By the Shift property, $\zeta_m(\psi^1_{F_i}) = c_{\ul{L}^m}(F_i) - c_{\ul{L}^1}(F_i)$, and by the Support control property from the same theorem,  $c_{\ul{L}^1}(F_i) = 0.$  It then follows by continuity and the Monotonicity property that
\[ \zeta_m(\psi) \ge  c_{\ul{L}^m}(F_i),\]
hence by the Calabi property from Theorem \ref{theo:recovering_Calabi},
\[ \text{lim}_{m \to \infty} \zeta_m(\psi) \ge  \int^1_0 \int_{\Sigma} F_i \hspace{2 mm} \omega\]
for any $i$.  Hence by the third property above, the $\zeta_m(\psi)$ diverge.

In the case when $\Sigma$ is not closed, we reduce to the above by embedding $\Sigma$ into a closed surface $\Sigma'$.  Now define an infinite twist exactly as above, except in addition the infinite twist is supported in  $\Sigma$:  by the above, this map is not in $\FHomeo(\Sigma',\omega')$, hence can not be in $\FHomeo(\Sigma,\omega).$
\end{proof}

\begin{remark}\label{rem:quotient}
The infinite twist $\psi$, introduced above in \eqref{eqn:inf_twist}, is the time-1 map of the 1-parameter subgroup $\psi^t$ of $\Homeo_0(\Sigma, \omega)$ defined by $\psi^t(0) =0$ and 
$$\psi^t(r,\theta) := (r, \theta + 2 \pi tf(r)).$$
It follows immediately from the above proof that $\psi^t$ is not a finite-energy homeomorphism for $t \neq 0$.   This yields an injective group homomorphism from the real line $\R$ into the quotient $\overline{\Ham}(\Sigma, \omega)/\FHomeo(\Sigma, \omega)$. 

Since $\Hameo(\Sigma, \omega) \subset \FHomeo(\Sigma, \omega)$, we see that $\psi^t$ yields an injective group homomorphism from $\R$ into the quotient $\overline{\Ham}(\Sigma, \omega)/\Hameo(\Sigma, \omega)$, as well.

One can show that the above injections are not surjections; see \cite{Pol-Shel21}.  However, we have not been able to determine whether or not the quotients  are isomorphic to $\R$ as abelian groups.
\end{remark}

\subsection{Calabi on Hameo}

We will now provide a proof of Theorem \ref{theo:extending_calabi} which states that Calabi extends to $\Hameo$. Recall from Remark \ref{rem:volume_ECH} that Hutchings conjectured that one could recover the Calabi invariant from the asymptotics of spectral invariants defined using Periodic Floer homology.  In \cite[Section 7.4]{CGHS20} we explained how to use such a result to deduce Theorem \ref{theo:extending_calabi}. 
When the first version of this paper appeared, it was not known 
whether or not Hutchings' conjecture holds, so this was just a `proof of principle'. 
The argument below adapts the template of \cite[Section 7.4]{CGHS20} to our newly defined Lagrangian spectral invariants, for which we have established the analogous Calabi property in Theorem \ref{theo:recovering_Calabi}. In that sense, one could view Theorem \ref{theo:extending_calabi} as additional circumstantial evidence for Hutchings' conjecture, and indeed after this paper appeared it was later shown in \cite{CPZ21, EH21} that Hutchings' conjecture holds.

\begin{proof}

Let $\phi \in \Hameo(\Sigma,\omega)$, and take an $H \in C^0([0,1]\times \Sigma)$ such that $$\phi^1_{H_i} \xrightarrow{C^0} \phi, \text{ and } \| H- H_i \|_{(1, \infty)} \to 0,$$ 
where the $H_i$ are smooth Hamiltonians as in Definition \ref{def:hameo}. For future use, we record $H$ in the notation by writing $\phi = \phi_H$.

We now define
\begin{equation}
\label{eqn:caldefn}
\Cal(\phi) := \int_{0}^1 \int_\Sigma H \, \omega \, dt.
\end{equation}
We claim this is well-defined.   To show this, it suffices to show that if $\phi = \text{Id}$, then
\begin{equation}
\label{eqn:needed}
\int_{0}^1 \int_\Sigma H \, \omega \, dt = 0,
\end{equation}
since $\text{Cal}$ is a homomorphism on $\Ham(\Sigma,\omega).$  In other words, we will show that if $\phi^1_{H_i} \xrightarrow{C^0} \id$ and $\Vert H-H_i \Vert_{(1, \infty)} \to 0$, then \eqref{eqn:needed} holds.

 As in Proposition~\ref{prop:C0-cont2}, embed $\Sigma$ into a closed surface $\Sigma'$, choose a sequence of equidistributed Lagrangian links $\ul{L}^m$ in $\Sigma'$, and consider $\xi_m: \Ham(\Sigma, \omega) \rightarrow \R$  by
\[ \xi_m(\varphi) =  c_{\ul{L}^m}(\varphi) + \Cal(\varphi).\]
By Proposition \ref{prop:C0-cont2}, $\xi_m$ extends continuously to $\overline{\Ham}(\Sigma, \omega)$. This in particular implies that 
\begin{equation}\label{eqn:limit_zero}
\lim_{j \to \infty}\xi_m(\phi^1_{H_j}) = 0.
\end{equation}
For any fixed $m, i$, we can write
\begin{align*}
  \left\vert \int_{0}^1 \int_\Sigma H \, \omega \,  dt \right\vert   
&\leq  \left\vert \int_0^{1} \int_\Sigma  H \, \omega \, dt  \; - \Cal(\phi^1_{H_i}) \right\vert  \\ &\qquad\qquad+   \left\vert   \Cal(\phi^1_{H_i}) - \xi_m(\phi^1_{H_i}) \right\vert +  \left\vert \xi_m(\phi^1_{H_i}) \right\vert .
\end{align*}
The right hand side of the above inequality is a sum of three terms.  We know that 
$$\left\vert \int_{0}^1 \int_\Sigma H \, \omega - \Cal(\phi^1_{H_i})  \,  dt \right\vert \le || H - H_i ||_{(1,\infty)},$$ since  $H_i $ are smooth and compactly supported Hamiltonians and so $\Cal(\phi^1_{H_i}) = \int_{0}^1 \int_\Sigma H_i \, \omega \, dt.$  We claim that the third term has the same bound.  Indeed, by the Hofer Lipschitz property from Theorem \ref{t:spectral}, we have $| \xi_m(\phi^1_{H_j}) - \xi_m(\phi^1_{H_i}) | \le  || H_j - H_i ||_{(1,\infty)}$ for all $i, j$, and then fixing $i$ and taking the limit as $j \to \infty$ gives $$|\xi_m (\phi^1_{H_i}) | \le  || H - H_i ||_{(1, \infty)}$$ by \eqref{eqn:limit_zero}.   Hence, whatever $m$, the first and third terms of the above inequality can be made arbitrarily small by choosing $i$ sufficiently large.  As for the second term, 
for fixed $i$, this can be made arbitrarily small by choosing $m$ sufficiently large, by the Calabi property proved in Theorem~\ref{theo:recovering_Calabi}.  

Hence, $\text{Cal}$ is well-defined.  It remains to show that it is a homomorphism.  The fact that $\Cal$ is a homomorphism if well-defined was in fact previously shown in \cite{Oh10} so we will be brief.  Let $\psi_1$ and $\psi_2$ be elements of $\Hameo(\Sigma,\omega)$, and choose corresponding $H$, $G$.  By reparametrizing, we can assume that $H$ and $G$ vanish near $0$ and $1$, and we can then form the concatenation   
\[K(t,x) =
     \begin{cases}
       2 H(2t,x), &\quadif  t\in [0, \frac{1}{2}]\\
       2 G(2t-1,x),&\quadif t\in [\frac{1}{2}, 1]
     \end{cases}.
   \]
One now checks that $\phi_K = \phi_G \circ \phi_H$, and it now follows immediately from this formula for $K$ and \eqref{eqn:caldefn} that $\Cal(\phi_G \circ \phi_H) = \Cal(\phi_H) + \Cal(\phi_G).$  The proof that $\Cal((\phi_H)^{-1}) = - \Cal(\phi_H)$ is similar.
\end{proof}

\section{Heegaard tori and Clifford tori}\label{sec:heeg-tori-cliff}
The proof of Theorem \ref{t:spectral} occupies the next three sections of the paper.  Recall from the introduction that this result will be obtained by studying a Floer cohomology for symmetric product Lagrangians $\Sym(\ul{L})$ in the symmetric products of the surface.  This section is mainly devoted to the proof of a monotonicity result (Lemma \ref{l:unobstructed}), which will later on guarantee that we have a well-defined Lagrangian Floer cohomology.
Section \ref{s:potential} computes the potential function of $\Sym(\ul{L})$ and Section \ref{s:QHF} defines the Floer cohomology and spectral invariants.

\subsection{Set-up and outline} \label{s:set-up}

We recall the set-up. 
Fix a  closed genus $g$ surface $\Sigma$, and equip $\Sigma$ with a symplectic form $\omega$.
We can choose a complex structure $J_{\Sigma}$ on $\Sigma$ such that $\omega$ is a K\"ahler form.
Consider a Lagrangian link  $\ul{L}=\cup_{i=1}^k L_i$ consisting of $k$ pairwise-disjoint circles on $\Sigma$, with the property that 
 $\Sigma \setminus \ul{L}$ consists of planar domains $B_j^{\circ}$, with $1 \leq j \leq s$, whose closures $B_j \subset \Sigma$ are also planar. Let $B_j$ have $\tau_j$ boundary components. 
Since the Euler characteristic of a planar domain $D$ with $\tau_D$ boundary components is $2-\tau_D$, the Euler characteristic of $\Sigma$ is $2-2g=\sum_{j=1}^s(2-\tau_j)=2s-2k$, and hence $s=k-g+1$.
We assume throughout that $s \ge 2$.  Finally, for $1\leq j\leq s$, let $A_j$ denote the $\omega$-area of $B_j$.
\medskip

Let $(M,\omega_M)=(\Sigma^k, \omega^{\oplus k})$.
Let $X:=\Sym^k(\Sigma)$ be the $k$-fold symmetric product.
It has a complex structure $J_X$ induced from $J_{\Sigma}$ making $X$ a complex manifold 
and the quotient map $\pi: M \to X$ holomorphic.  
We equip $X$ with the singular K\"ahler current $\omega_X$ which naturally descends from $(M,\omega_M)$ under $\pi$.
Let $\Sym(\ul{L})$ be the Lagrangian submanifold in $X$ given by the image of 
$L_1 \times \dots \times L_k$ under $\pi$.
The spectral invariant $c_{\ul{L}}$ of Theorem \ref{t:spectral} will be constructed using a variant of Lagrangian Floer cohomology of $\Sym(\ul{L})$  in $X$, `bulk deformed' by $\eta$ times the diagonal divisor.

\begin{remark} \label{rmk:Floer_survey}
We briefly recall some points in Lagrangian Floer theory of particular relevance in the sequel, which may help guide the reader.  (This is a necessarily informal and imprecise overview: the main body of the text gives our exact set-up, choice of coefficients, etc. General references for Floer theory which amplify these remarks include \cite{Floer:Morse, FOOOspectral, SeidelBook}.)  Given a  Lagrangian $L$ in a symplectic manifold $X$ (both satisfying suitable monotonicity hypotheses)  and a Hamiltonian translate $\phi_H^1(L)$ which is transversal to $L$, the complex $CF^*(L, \phi_H^1(L))$ is generated by the intersection points $L \cap \phi_H^1(L)$ and has a differential which depends on an auxiliary almost complex structure $J$ and which is determined by  solutions $u: [0,1]\times \mathbb{R} \to X$  to a (possibly perturbed) non-linear Cauchy-Riemann equation $\overline{\partial}_J(u) = 0$.  We say a solution $u$ is \emph{regular} if the linearisation $D(\overline{\partial}_J)|_u$ of the associated differential operator is surjective; for generic data $(J,H)$ spaces of solutions $u$ are smooth at regular points.  A choice of spin structure on $L$ orients the solution spaces; the Floer differential then counts their signed isolated points. There is an important variation in which one equips $L$ with a rank one local system $\mathcal{E}$ (and the differential is weighted by a contribution from the holonomy of this local system), and a further variation -- a special kind of `bulk deformation' --  in which holomorphic strips are further weighted by their intersection number with an almost complex divisor in the complement of $L$.
\medskip

In general, the differential in the Floer complex does not square to zero. 
When it squares to zero,  we say $(L,\mathcal{E})$ is \emph{unobstructed}. In this case, we have the `Oh spectral sequence' $H^*(L) \Rightarrow HF^*(\mathcal{E},\mathcal{E})$. For a Lagrangian torus $L \cong T^n$, there is a distinguished translation-invariant spin structure.  The \emph{disc potential} $W_L$ is a count of isolated holomorphic discs with boundary on $L$ passing through some fixed generic point,  and weighted by the holonomy of a local system.  (Genericity of the point constraint on $L$ and of the choice of  $J$ ensures regularity of the discs. Since $L$ is orientable, it bounds no discs of Maslov index 1, and the isolated discs are those with the minimal Maslov index 2.) Viewed as a $\mathbb{C}$-valued function on the space of rank one local systems $\Hom(\pi_1(L);\mathbb{C}^*)$, $\mathcal{E}$ is a critical point of $W_L$ exactly when the first differential in the Oh spectral sequence vanishes on $H^1(L)$; multiplicativity of the spectral sequence then means the first differential vanishes altogether, the spectral sequence collapses, and $HF(\mathcal{E},\mathcal{E}) \cong H^*(T^n) \neq 0$. Changing the spin structure on $L$ by an element of $H^1(L;\mathbb{Z}/2)$ changes the critical point $\mathcal{E}$ by twisting it by the corresponding $\{\pm 1\}$-valued local system on $L$; in particular, the existence of critical points does not depend on the choice of spin structure. This is why computing the disc potential for a Lagrangian torus $\Sym(\underline{L})$ is a key goal in the sequel. \end{remark}

\begin{remark} \label{rmk:link_Floer_2}
It is crucial for our purposes that our Floer cohomology is invariant under Hamiltonian isotopies (at least those inherited from isotopies of the link $\ul{L}$).  It is well-known that Floer cohomology over $\C$ is Hamiltonian invariant only under monotonicity hypotheses, which is where the hypothesis of Theorem \ref{t:spectral} arises. 
The following illustrative example may be helpful. Consider two circles on $\PP^1$ whose complementary domains have closures (disjoint) discs $B_1$ of area $A_1$, $B_2$ of area $A_2$ and an annulus $B_3$ of area $A_3$. The Maslov index two discs on $\Sym(\ul{L})$ are given by $B_1$, $B_2$ and a double covering of $B_3$. The fact that such branched covers arise makes it natural to keep track of branch points, and hence intersections with the diagonal divisor (see Remark \ref{r:branchdiagonal}); this is the role of our bulk parameter $\eta$.  Hamiltonian invariance for the bulk-deformed version  relies on restricting to values $\eta \geq 0$.  Our analysis of the Floer complex of $\Sym(\ul{L})$  would apply equally well over the Novikov field, cf. Definition \ref{rmk:Novikov_disc_potential}.
\end{remark}

\begin{remark}\label{rmk:link_Floer}
In Heegaard Floer theory for links in 3-manifolds, one begins with a surface $\Sigma$ of genus $g$, two sets of attaching circles $\alpha_1,\ldots,\alpha_k$ and $\beta_1,\ldots, \beta_k$ and two sets of base-points $z_1,\ldots,z_l$ and $w_1,\ldots,w_l$, where $k=g+l-1$, see \cite[Definition 3.1]{OS:AGT2008}.  This data encodes a link in a $3$-manifold; one can take $g=0$ for links in $S^3$.  Link Floer homology is obtained from a version of Lagrangian Floer cohomology of product-like tori associated to $\alpha$ and $\beta$ in $\Sym^k(\Sigma)$.   For link invariants the crucial topological information is contained in the filtrations associated to the intersection numbers with divisors $D_p = p + \Sym^{k-1}(\Sigma)$, for $p \in \{z_i, w_j\}$ one of the base-points, which play no role in this paper. Our `quantitative version' instead keeps track of holonomies of local systems and of intersection number with the diagonal divisor.  We also work with `anchored' or `capped' Floer generators, so that the action functional becomes well-defined.
\medskip

 In Heegaard Floer homology of links in 3-manifolds, Hamiltonian invariance is less relevant: the important invariance properties are those which give different presentations of a fixed link (handleslide moves and stabilisations), which one shows respect the topological information held by the filtrations determined by the $D_p$. In the example of Remark \ref{rmk:link_Floer_2}, if $A_1 \neq A_2$ then the Lagrangian link (pair of circles on $S^2$) is displaceable. Nonetheless, there is a non-trivial Heegaard Floer cohomology (over $\Z/2$) for a link in a 3-manifold presented by a diagram comprising the given two circles as $\alpha$-circles, their images under a small Hamiltonian isotopy as $\beta$-circles, and with one $z$-base-point and one $w$-base-point in each of the original 3 complementary regions. 
\end{remark}

The unobstructedness of $\Sym(\ul{L})$ follows broadly as in its Heegaard Floer counterpart. (More precisely,  in the link setting, 
a `weak admissibility’ condition is imposed on Heegaard diagrams to rule out bubbling which would obstruct the Floer complex over $\Z/2$, whereas our analysis works over characteristic $0$.)
To compute Floer cohomology, we first consider the special case  in which $\Sigma = \PP^1$
and the $B_j$ are discs for $j=1,\dots,k$.
We show the corresponding $\Sym(\ul{L})$ is  isotopic to a Clifford-type torus in $X=\Sym^k(\PP^1)=\PP^k$, and use that isotopy to compare the holomorphic discs they bound.  In the general case, the fact that the  regions $B_j \subset \Sigma$  are planar domains enables us to reduce aspects of the holomorphic curve theory to the case $\Sigma = \PP^1$.   Our proof incorporates local systems because non-vanishing of Floer cohomology is detected, as in \cite{Charest-Woodward,Cho-Oh}, by considering the Floer boundary operator under variation of the local system. 
We obtain a spectral invariant $c_{\cE}$ for any local system $\cE \to \Sym(\ul{L})$ with respect to which Floer cohomology is non-trivial. In fact Floer cohomology is non-zero for the trivial local system on $\Sym(\ul{L})$, and (after rescaling by the number of components) it is the spectral invariant $c_{\cE}$ for the trivial local system which is the $c_{\ul{L}}$ which appears in Theorem \ref{t:spectral}. 
\medskip

For unobstructedness of the Floer cohomology of $\Sym(\ul{L})$, we will need control over the Maslov indices of holomorphic discs with boundary on that torus.  To that end, we next show that when $\Sigma = \PP^1$ and the circles $L_j$ bound pairwise disjoint discs $B_j$, with $1 \leq j \leq k=s-1$, the torus $\Sym(\ul{L})$ is isotopic to a Clifford-type torus in projective space.

\subsection{Co-ordinates on the symmetric product\label{Sec:coords_symmetric_product}}

The symmetric product $\Sym^k(\mathbb{P}^1)$ is naturally a complex manifold, biholomorphic to $\PP^k$.  To fix notation, we recall that isomorphism. Let $x_{0,i}, x_{1,i}$ denote homogeneous co-ordinates on the $i$-th factor of $(\PP^1)^k$. 
Define $Q_0(x),\dots,Q_k(x) \in \C[x_{0,1},x_{1,1},\dots,x_{0,k},x_{1,k}]$ by the identity
\begin{align*}
\prod_{j=1}^k(x_{0,j}X+x_{1,j}Y)=\sum_{j=0}^k Q_j(x)X^{k-j}Y^j.
\end{align*}
Let $Y_0,\dots,Y_k$ be the homogenous co-ordinates of $\PP^k$.
We define  $\pi:(\PP^1)^k \to \PP^k$ by
\[ [Y_0:\dots:Y_k]=\pi([x_{0,1}:x_{1,1}],\dots,[x_{0,k}:x_{1,k}])= [Q_0(x):\dots:Q_k(x)].\]
It is an $S_k$-invariant holomorphic map which, by the fundamental theorem of algebra, descends to a bijective map
$\Sym^k(\mathbb{P}^1) \simeq \mathbb{P}^k$.
Under this identification,  $\Sym^k(\mathbb{P}^1)$ is equipped with the structure of a complex manifold.
\medskip

Let $a_1,\dots,a_{k+1} $ be  $k+1$ pairwise distinct points in $\PP^1$.
We identify $\PP^1$ as $\C \cup \{\infty\}$ and assume that $a_{k+1}=\infty$.
For each $i=1,\dots,k+1$, we define 
\begin{align*}
\widetilde{D}_i\ :=\ \left\{\prod_{j=1}^k (x_{1,j}-a_ix_{0,j})=0\right \} \subset (\PP^1)^k.
\end{align*}
Note that $\widetilde{D}_i$ is $S_k$-invariant and it descends to 
\begin{align*}
D_i:=\pi(\widetilde{D}_i) = \left\{\sum_{j=0}^k (-a_i)^{k-j}Y_j=0\right\} \subset \PP^k.
\end{align*}

When $i=k+1$, the divisors $\widetilde{D}_{k+1}$ and $D_{k+1}$ are understood as $\{\prod_{j=1}^k x_{0,j}=0\}$
and $\{Y_0=0\}$ respectively.

\begin{remark}\label{r:symD}
The divisor $D_i$ is precisely the image of $\Sym^{k-1}(\Sigma) \to \Sym^k(\Sigma)=X$ under the map $D \mapsto a_i +D$  (i.e.\ $(p_1,\dots,p_{k-1}) \mapsto (p_1,\dots,p_{k-1},a_i)$ but written in a form that regards $D=p_1+\dots+p_{k-1}=(p_1,\dots,p_{k-1})$ as a divisor in $\Sigma$).
In particular, the $D_i$ are pairwise homologous, and  $\cup_{i=1}^{k+1} D_i$ is an anticanonical divisor of $X$ (i.e. the divisor class $\sum_{i=1}^{k+1} [D_i]$ is $-K_X$, where $K_X$ is the divisor class of the top exterior power of the cotangent bundle of $X$).
\end{remark}

Note that $(\PP^1)^k \setminus \tilde{D}_{k+1}=\C^k$ and $\pi|_{\C^k}:\C^k \to \C^k$ is a $S_k$-invariant holomorphic map which descends to a biholomorphism $\Sym^k(\C) \simeq \C^k$\footnote{For any Riemman surface $\Sigma$, the complex structure on $\Sym^k(\Sigma)$ is defined by applying this identification $\Sym^k(\C) \simeq \C^k$ locally.}.
For $i=1,\dots,k$, we define $x_i:=\frac{x_{1,i}}{x_{0,i}}$ and $y_i:=\frac{Y_i}{Y_0}$, which give coordinates on the complements of $\tilde{D}_{k+1}$ and $D_{k+1}$ respectively.
Since $q_j:=\frac{Q_j}{Q_0}$ is precisely the $j^{th}$ elementary symmetric polynomial of $\{x_i\}_{i=1}^k$, 
the map $\pi|_{\C^k}$ can be written as
\begin{align*}
(y_1,\dots,y_k)&=\pi(x_1,\dots,x_k)=(q_1(x),\dots,q_k(x))\\
q_j(x)&=\frac{1}{j! (k-j)!}\sum_{\sigma \in S_k} x_{\sigma(1)}\dots x_{\sigma(j)}.
\end{align*}
In affine coordinates, for $i=1,\dots,k$, we have
\begin{align*}
\widetilde{D}_i\setminus \widetilde{D}_{k+1} &=\ \left\{\prod_{j=1}^k (x_{j}-a_i)=0\right \} \text{ and } D_i \setminus D_{k+1}= \left\{\sum_{j=0}^k (-a_i)^{k-j}y_j=0\right\}.
\end{align*}
\medskip

Since the $\{a_i\}$ are pairwise distinct, the Vandermonde matrix
\begin{equation*}
A = 
\begin{pmatrix}
(-a_1)^{k-1} & (-a_1)^{k-2} & \dots & 1 \\
(-a_2)^{k-1} & (-a_2)^{k-2} & \dots & 1 \\
\vdots & & & \vdots \\
(-a_k)^{k-1} & (-a_k)^{k-2} & \dots & 1 \\
\end{pmatrix}
\end{equation*}
is non-degenerate.
We define $g_i=\sum_{j=0}^k (-a_i)^{k-j}y_j$ so that $ D_i \setminus D_{k+1}=\{g_i=0\}$;  the non-degeneracy of $A$ implies that $\{g_i\}_{i=1}^k$ is an invertible linear change of coordinates of $\{y_i\}_{i=1}^k$.

\subsection{Relation to the Clifford torus}

For $\eps>0$ small, we define the Clifford torus in $X$ as
\begin{align*}
L_{\eps}:=\{(g_1,\dots,g_k) \in \C^k : |g_i|=\eps \text{ for each }i\}.
\end{align*}
The main result of this section, Corollary~\ref{Cor:Heegaard_is_Clifford} below, asserts that when $\eps$ is small, $L_{\eps}$ is $C^1$ close to $\Sym(\ul{L})$ for an appropriate $\ul{L}$.
\medskip

For a small neighborhood $G$ of $(g_1,\dots,g_k)=0$, $\pi|_{\pi^{-1}(G)}: \pi^{-1}(G) \to G$ is a trivial covering map with $k!$ sheets. For example, 
\[
\pi^{-1}(\{(g_1,\dots,g_k)=0\}) \ = \ \bigcup_{\sigma \in S_k}\{ x_i=a_{\sigma(i)} \text{ for } 1 \leq i\leq k\}.\]
Therefore, when $\eps>0$ is small, $\pi^{-1}(L_{\eps})$ is a collection of $k!$ pairwise-disjoint totally real $k$-tori in $(\PP^1)^k$. 
More explicitly, we have
\begin{align*}
\pi^{-1}(L_{\eps})=\left\{ \left|\prod_{j=1}^k (x_j-a_i)\right|=\eps  \text{ for each }i=1,\dots,k\right\}.
\end{align*}

Let $\widetilde{G}$ be the connected component of $\pi^{-1}(G)$ containing the point with co-ordinates $x_i=a_i$ for each $i$.
For $\eps>0$ sufficiently small, there exists $\delta > \eps$ such that 
\begin{align*}
&\left\{|x_i-a_i| <\delta \text{ for each }i, 1 \leq i \leq k\right\} \subset \widetilde{G} \\
&\widetilde{L}_{\eps}:=\pi^{-1}(L_{\eps}) \cap  \widetilde{G} =\left\{ \left|\prod_{j=1}^k (x_j-a_i)\right|=\eps \text{ and } |x_i-a_i| <\delta \text{ for each }i, \ 1\leq i\leq k \right\} \\ 
& \text{If }  x \in \widetilde{L}_{\eps} \text{ and } i \neq j, \text{ then }
|x_i-a_j|>3\delta. 
\end{align*}

We will need the following lemma to relate $\widetilde{L}_{\eps}$ to $\Sym(\ul{L})$ for an appropriate $\ul{L}$ (cf. Corollary \ref{Cor:Heegaard_is_Clifford}).

\begin{lemma}\label{l:realFamily}
For $\kappa>0$, there exists a small  $\eps >0$ and a family of diffeomorphisms $(\Phi^t)_{t \in [0,1]}$ of $(\PP^1)^k$ 
supported inside $\widetilde{G}$ with the following properties: 
\begin{itemize}
\item $\Phi^0$ is the identity;
\item the $C^1$-norm of $\Phi^t$ is less than $\kappa$ for all $t \in [0,1]$;
\item $\Phi^1( \{|x_i-a_i||\prod_{j \neq i} (a_j-a_i)|=\eps  \text{ for all }i\})=\widetilde{L}_{\eps}$;

\item $\Phi^t(\widetilde{D}_i \cap \widetilde{G})=\widetilde{D}_i \cap \widetilde{G}$ for all $t\in[0,1]$ and all $i=1,\dots,k$.
\end{itemize}
\end{lemma}

The proof will be postponed after the following main consequence.

\begin{corol} \label{Cor:Heegaard_is_Clifford}
If $\eps > 0$ is sufficiently small, and if $\ul{L}_{a_1,\dots,a_k,\eps}$ is the union of circles
\[
\ul{L}_{a_1,\dots,a_k,\eps} = \bigcup_{i=1}^k \left\{|x_i-a_i|\left|\prod_{j \neq i} (a_j-a_i)\right|=\eps \right\} \subset \C,
\] there is a $C^1$-small isotopy $\Phi^t_G$ supported in $G \subset \PP^k$, with 
\[
\Phi^1_G(\Sym(\ul{L}_{a_1,\dots,a_k,\eps}))=L_{\eps} \qquad \textrm{and} \qquad 
\Phi^t_G(D_i)=D_i \ \textrm{for all} \ i,t.
\] 
\end{corol}

\begin{proof}[Proof of Corollary \ref{Cor:Heegaard_is_Clifford} assuming Lemma \ref{l:realFamily}] 
The submanifold 
\begin{align*}
\widetilde{L}_{a_1,\dots,a_k,\eps}:= \left\{|x_i-a_i|\left|\prod_{j \neq i} (a_j-a_i)\right|=\eps  \text{ for all }i\right\} \cap \tilde{G}
\end{align*}
is a product of circles, and 
$\pi(\widetilde{L}_{a_1,\dots,a_k,\eps})$ is precisely $\Sym(\ul{L}_{a_1,\dots,a_k,\eps})$.
Since the isotopy $\Phi^t$ constructed in Lemma \ref{l:realFamily}  is supported in $\tilde{G}$ and $\pi|_{\tilde{G}}: \tilde{G} \to G$ is a diffeomorphism, we can descend 
$\Phi^t$ to a family of diffeomorphisms $\Phi^t_G$ supported in $G$ such that $\Phi^1_G(\Sym(\ul{L}_{a_1,\dots,a_k,\eps}))=L_{\eps}$ and $\Phi^t_G(D_i)=D_i$ for all $i$.
\end{proof}

We will use Corollary \ref{Cor:Heegaard_is_Clifford} to obtain control over holomorphic discs on $\Sym(\ul{L}_{a_1,\dots,a_k,\eps})$. We will discuss how to extend that control from $\ul{L}_{a_1,..,a_k,\eps}$ to a more general $\ul{L}$ associated to a collection of disjoint discs in Corollary \ref{c:discclass} and Proposition \ref{p:SymL=Clifford2}.
We finish this section wih the following proof.

\begin{proof}[Proof of Lemma \ref{l:realFamily}]
For simplicity of notation, we will give the proof in the case in which $a_i \in \R$ for each $i$.
\medskip

Let $u_i+\sqrt{-1} v_i=x_i-a_i$.
Then $\widetilde{D}_i \cap \widetilde{G}=\{(u,v)\in \widetilde{G}: u_i=v_i=0\}$.
The system of equations 
\begin{align}
(x_i-a_i)\left((1-t)\prod_{j \neq i} (a_j-a_i)+t\prod_{j \neq i} (x_j-a_i)\right)\ =\ \alpha_i +\sqrt{-1} \beta_i  \qquad 1\leq i\leq k\label{eq:linear}
\end{align}
for $t \in [0,1]$ and $\alpha_i, \beta_i \in \R$ becomes, in the $u_i, v_i$ co-ordinates,  
\begin{align}
(u_i+\sqrt{-1} v_i)\left((1-t)\prod_{j \neq i} (a_j-a_i)+t\prod_{j \neq i} (u_j+\sqrt{-1} v_j+a_j-a_i)\right)=\alpha_i +\sqrt{-1} \beta_i. \label{eq:linear2}
\end{align}
Taking real and imaginary parts, we obtain 
\begin{align*}
u_i \prod_{j \neq i} (a_j-a_i) + t H_{u_i}(u,v)=\alpha_i,\\
v_i \prod_{j \neq i} (a_j-a_i)  + t H_{v_i}(u,v)= \beta_i,
\end{align*}
where $H_{u_i}(u,v)$ and $H_{v_i}(u,v)$ are polynomials in $u_j,v_j$ in which each term has degree at least two.
\medskip

Let $\rho: \mathbb{R}_{\ge 0} \to \mathbb{R}_{\ge 0}$ be a cut-off function such that
$\rho(s)=1$ for $s<\frac{\eps^2}{2}$, $\rho(s)=0$ for $s>\eps^2$
and $|\rho'(s)|<\frac{C}{\eps^2}$ for some constant $C$ independent of $\eps$ and for all $s$.
We denote $\sum_i u_i^2+v_i^2$ by $|(u,v)|^2$.
Let 
\begin{align*}
F_{u_i}^{t}(u,v,\alpha,\beta):=u_i \prod_{j \neq i} (a_j-a_i) + t\rho(|(u,v)|^2) H_{u_i}(u,v)-\alpha_i,\\
F_{v_i}^{t}(u,v,\alpha,\beta):=v_i \prod_{j \neq i} (a_j-a_i)  + t\rho(|(u,v)|^2) H_{v_i}(u,v)-\beta_i.
\end{align*}
The $2k \times 2k$ square matrix
\begin{equation*}
D_{u,v}F^t:=
\begin{pmatrix}
\frac{\partial F_{u_1}^{t}}{\partial u_1} & \frac{\partial F_{u_1}^{t}}{\partial v_1} & \dots & \frac{\partial F_{u_1}^{t}}{\partial u_k} &\frac{\partial F_{u_1}^{t}}{\partial v_k}  \\
\frac{\partial F_{v_1}^{t}}{\partial u_1} & \frac{\partial F_{v_1}^{t}}{\partial v_1}  & \dots  & \frac{\partial F_{v_1}^{t}}{\partial u_k} & \frac{\partial F_{v_1}^{t}}{\partial v_k}  \\
\vdots & \vdots & & \vdots & \vdots  \\
\frac{\partial F_{u_k}^{t}}{\partial u_1} & \frac{\partial F_{u_k}^{t}}{\partial v_1}  & \dots  & \frac{\partial F_{u_k}^{t}}{\partial u_k} & \frac{\partial F_{u_k}^{t}}{\partial v_k}  \\
\frac{\partial F_{v_k}^{t}}{\partial u_1} & \frac{\partial F_{v_k}^{t}}{\partial v_1}  & \dots  & \frac{\partial F_{v_k}^{t}}{\partial u_k} & \frac{\partial F_{v_k}^{t}}{\partial v_k}  
\end{pmatrix}
\end{equation*}
can be written as a sum $A+t\rho B_1 + t\rho' B_2$, where $A$ is the diagonal matrix with entries $\prod_{j \neq i} (a_j-a_i)$ at both the $(2i-1,2i-1)^{th}$ and $(2i,2i)^{th}$ positions, for each $i=1,\dots,k$, and where the entries of $B_\ell$ are polynomials, with each non-zero term having degree at least $1$ when $\ell=1$ and degree at least $3$ when $\ell=2$.
\medskip

Since the support of $\rho$ is $[0,\eps^2]$, when $\eps>0$ is small we have
\begin{align*}
\|D_{u,v}F^t-A\|=t\,O(\eps)
\end{align*}
for all $t \in [0,1]$ and for all points $(u,v)$. By the implicit function theorem, there exists a unique $g^t(\alpha,\beta)$ such that
\begin{align}
F_{u_i}^{t}(g^t(\alpha,\beta),\alpha,\beta)=0 \label{eq:soln1}\\
F_{v_i}^{t}(g^t(\alpha,\beta),\alpha,\beta)=0 \label{eq:soln2}
\end{align}
for all $i=1,\dots,k$.  Define a smooth isotopy starting at the identity by 
\[
\Phi^t(u,v):=g^t \circ (g^0)^{-1}(u,v).
\]
We can control the $C^1$-norm of the isotopy as follows.  We have $D\Phi^t=D(g^t) \circ D((g^0)^{-1})$.
Since $g^t$ solves the equations \eqref{eq:soln1} and \eqref{eq:soln2} for each $i$, by differentiating with respect to $\alpha$ and $\beta$, we have
\begin{equation*}
Dg^t  
 = -(D_{u,v}F^t)^{-1}
\begin{pmatrix}
\frac{\partial F_{u_1}^{t}}{\partial \alpha_1} & \frac{\partial F_{u_1}^{t}}{\partial \beta_1} & \dots & \frac{\partial F_{u_1}^{t}}{\partial \beta_k}  \\
\frac{\partial F_{v_1}^{t}}{\partial \alpha_1} & \frac{\partial F_{v_1}^{t}}{\partial \beta_1}  & \dots & \frac{\partial F_{v_1}^{t}}{\partial \beta_k}  \\
\vdots & & & \vdots \\
\frac{\partial F_{v_k}^{t}}{\partial \alpha_1} & \frac{\partial F_{v_k}^{t}}{\partial \beta_1}  & \dots & \frac{\partial F_{v_k}^{t}}{\partial \beta_k}  \\
\end{pmatrix}=(D_{u,v}F^t)^{-1}.
\end{equation*}
As a result, we have
\begin{align*}
\|Dg^t-A^{-1}\|=\frac{t\,O(\eps)}{\|A\|^2}.
\end{align*}
Moreover, when $t=0$, we have exactly $Dg^0=A^{-1}$.
Therefore, we have
\begin{align*}
\|D\Phi^t-\id\|=\frac{t\, O(\eps)}{\|A\|}
\end{align*}
so the $C^1$-norm of $\Phi^t$ is smaller than the prescribed $\kappa$ whenever  $\eps$ is sufficiently small.
\medskip

We now check the remaining conditions. 
Equation  \eqref{eq:linear2} implies that 
\[
(g^t)^{-1}(\{u_i=v_i=0\}) \subset \{\alpha_i=\beta_i=0\},
\] so 
\[\Phi^t(\widetilde{D}_i \cap \widetilde{G})=\widetilde{D}_i \cap \widetilde{G}
\]
for every $t, i$.  It is also clear from the construction, cf. \eqref{eq:linear}, that there exists $0<\eps' \ll \eps$ such that 
\begin{align*}
\Phi^1\left( \left\{|x_i-a_i|\left|\prod_{j \neq i} (a_j-a_i)\right|=\eps'  \text{ for all }i\right\}\right)=\widetilde{L}_{\eps'}.
\end{align*}
Therefore, replacing $\eps$ by $\eps'$,  the final claim of the statement holds.
\end{proof}

\subsection{Tautological correspondence}\label{s:taut}

We return to the general setting, in which the Riemann surface $\Sigma$ has genus $g$ and $\ul{L} \subset \Sigma$ comprises $k$ pairwise disjoint circles. Let $S$ denote the unit disc.
We  can understand holomorphic discs in $\Sym^k(\Sigma)$ with boundary on $\Sym(\ul{L})$  via
the `tautological correspondence' between a holomorphic map
\begin{equation} \label{eqn:tautological_u}
u:(S,\partial S) \to (X, \Sym(\ul{L}))
\end{equation}
and a pair of holomorphic maps $(v,\pi_{\WH{S}})$,  where 
\begin{equation} \label{eqn:tautological_v}
v:(\WH{S},\partial \WH{S}) \to (\Sigma, \ul{L})
\end{equation}
and  $\pi_{\WH{S}}:(\WH{S},\partial \WH{S}) \to (S, \partial S)$ is a $k:1$ branched covering with all the branch points lying inside the interior of $S$.
The correspondence arises as follows (see also \cite[Section 13]{Lip06}, \cite[Section 3.1]{MS19} and the references therein).
Let $\Delta \subset X$ be the `big diagonal' comprising all unordered $k$-tuples of points in $\Sigma$ at least two of which co-incide.  We denote by $J_X$ the standard complex structure on $X$ induced by $J_{\Sigma}$.
Given a continuous map $u:(S,\partial S) \to (X, \Sym(\ul{L}))$ that is $J_X$-holomorphic near $\Delta$, we have a pull-back diagram\footnote{When $u$ is $J_X$-holomorphic everywhere, the pull-back diagram exists in the holomorphic category. It guarantees the existence of the pull-back diagram even if $u$ is holomorphic  only near $\Delta$, because $\Delta$ is the branch locus of $\pi$.} 
\begin{equation} \label{eqn:tautological_square}
\xymatrix{
\WT{S} \ar[rr]^{V} \ar[d]_{\pi_{\WT{S}}} & & \Sigma^k  \ar[d]_{\pi}\\
S \ar[rr]_{u} &&  X
}
\end{equation}
so that $\WT{S}$ a surface with an $S_k$ action, and the quotient map by this action is $\pi_{\WT{S}}$.
By construction, $V$ is $S_{k}$-equivariant, and there is a unique conformal structure $J_{\WT{S}}$ on $\WT{S}$ such that $\pi_{\WT{S}}$ is holomorphic.
Moreover, $V$ is $J_{\WT{S}}$ holomorphic if and only if $u$ is $J_X$ holomorphic.
 \medskip
 
Let $\pi_1: \Sigma^k \to \Sigma$ be the projection to the first factor.
The map $\pi_1 \circ V$ is invariant under the subgroup $S_{k-1} \subset S_k$ which stabilises that first factor,  so $\pi_1 \circ V$ factors through a $(k-1)!$-fold branched covering $\WT{S} \to \WH{S}$. We denote the induced map $\WH{S} \to \Sigma$ by $v$.
We also have an induced $k$-fold branched covering $\pi_{\WH{S}}:\WH{S} \to S$, which is holomorphic with respect to the induced complex structure $J_{\WH{S}}$ on $\WH{S}$.
Note that $\partial \WH{S}$ has $k$ connected components and different connected components are mapped to different connected components of $\ul{L}$ under $v$.
\medskip

On the other hand, given a $k$-fold branched covering $\pi_{\WH{S}}$ and a continuous map $v$ as in \eqref{eqn:tautological_v} such that different connected components of $\partial \WH{S}$ are mapped to different connected components of $\ul{L}$,  we define a map as in \eqref{eqn:tautological_u} by $u(z)=v(\pi_{\WH{S}}^{-1}(z))$\footnote{More precisely, $\pi_{\WH{S}}^{-1}(z)$ is an unordered $k$-tuple of points in $\WH{S}$ so $v(\pi_{\WH{S}}^{-1}(z))$ is an unordered $k$-tuple of points in $\Sigma$ and hence can be regarded as a point in $\Sym^k(\Sigma)=X$.}.
The map $v$ is $J_{\WH{S}}$ holomorphic if and only if $u$ is $J_X$ holomorphic.

\begin{example}
If $u(S) \cap \Delta =\emptyset$, then $\WT{S}=\sqcup_{\sigma \in S_k} S_{\sigma}$ and $S_{\sigma}=S$ for all $\sigma \in S_k$.
The map $V|_{S_{id}}$ is a lift of $u$ to $\Sigma^k$ and $V|_{S_\sigma}=\sigma V|_{S_{id}}$.
The surface $\WH{S}=\sqcup_{[\sigma] \in S_k/S_{k-1}} S_{[\sigma]}$ and $S_{[\sigma]}=S$ for all $[\sigma] \in S_k/S_{k-1}$.
The map $v|_{S_{[\sigma]}}$ is canonically identified with $\pi_1 \circ V|_{S_{\sigma}}$.
\end{example}

\begin{remark}\label{r:branchdiagonal}
Note that if $z$ is a branch point of $\pi_{\WH{S}}$, then $u(z) \in \Delta$. In general, $u(z) \in \Delta$ does not guarantee that $z$ is a branch point of $\pi_{\WH{S}}$. 
\end{remark}

\begin{remark}
Fix two collections of pairwise-disjoint circles $\ul{L}$ and $\ul{K}$ (there may however be intersections between circles from $\ul{L}$ and ones from $\ul{K}$).
A continuous map
\begin{align*}
u:(\R \times [0,1], \R \times \{0\}, \R \times \{1\}) \to (X, \Sym(\ul{L}), \Sym(\ul{K}))
\end{align*}
that is $J_X$-holomorphic near $\Delta$ analogously gives rise to a tautologically corresponding pair, comprising a $k$-fold branched covering 
$\pi_{\WH{S}}: \WH{S} \to \R \times [0,1]$ together with a map 
$v:(\WH{S}, \partial_0 \WH{S}, \partial_1 \WH{S}) \to (\Sigma, \ul{L},\ul{K})$
where $\partial_i \WH{S}=\pi_{\WH{S}}^{-1}(\R \times \{i\})$.
\end{remark}

\subsection{Basic disc classes}

The identification of the Heegaard torus $\Sym(\ul{L}_{a_1,\dots,a_k,\eps})$ with a Clifford-type torus in Corollary \ref{Cor:Heegaard_is_Clifford} yields a helpful basis of $H_2(X, \Sym(\ul{L}))$.

\begin{corol}\label{c:discclass}
Suppose that $\Sigma=\PP^1$  and the $B_i$ are discs for $i=1,\dots,k=s-1$. 
Suppose also that $a_i \in B_i ^{\circ}$ for $i=1,\dots,k+1$.
Then $H_2(X,\Sym(\ul{L}))$ is freely generated by $k+1$ primitive classes $\{[u_i]\}_{i=1}^{k+1}$ such that $[u_i] \cdot D_j=\delta_{ij}$.
Moreover, each of these primitive classes has Maslov index $\mu(u_i)=2$.
\end{corol}

\begin{proof}
First we consider the special case that $\ul{L}=\ul{L}_{a_1,\dots,a_k,\eps}$ for small $\eps$.
Since $\Sym(\ul{L})$ is smoothly isotopic to $L_{\eps}$, there is an isomorphism of relative homology groups 
\begin{equation}\label{eqn:relative_h2_iso}
H_2(X,\Sym(\ul{L})) \cong H_2(X, L_{\eps}).
\end{equation}  
Since $\Phi^t_G$ is $C^1$-small and being totally real is an open condition, we can take the isotopy to be through totally real tori, in which case the isomorphism \eqref{eqn:relative_h2_iso} preserves the Maslov class \cite{RS93}. Furthermore, the isotopy $\Phi^t_G$ is supported away from the anticanonical divisor $\cup_{i=1}^{k+1} D_i$, so the isomorphism \eqref{eqn:relative_h2_iso} does not change the intersection number with $D_i$.
Since $L_{\eps}$ is a Clifford torus, it is known that $H_2(X,L_{\eps})$ admits a basis $\{[u_i]\}_{i=1}^{k+1}$ such that $[u_i] \cdot D_j=\delta_{ij}$.
Moreover, it is also known that $\mu(u_i)=2$ for all $i$ \cite{Cho-Clifford}.
Hence, the same is true for $\Sym(\ul{L})$.

For general $\ul{L}$ in $\PP^1$ such that $B_i$ are discs for $i=1,\dots,k=s-1$, we can find a 
smooth family of $(\ul{L}_t)_{t \in [0,1]}$ in $\PP^1$ such that $\ul{L}_0=\ul{L}$
and $\ul{L}_1=\ul{L}_{a_1,\dots,a_k,\eps}$ for some small $\eps>0$.
Moreover, we can assume that $\ul{L}_t$ is disjoint from $\{a_1,\dots,a_k\}$ for all $t$.
Therefore, we get a smooth family of Lagrangian tori $\Sym(\ul{L}_t)$ that is disjoint from $D_i$ for all $i$ and all $t$.
The result follows.
\end{proof}

In the course of the proof of the next lemma, we explain how to construct the disc classes $[u_i]$ in Corollary \ref{c:discclass} from the tautologically corresponding pairs of maps $(v_i,\pi_{\WH{S}_i})$, and use this to compute the intersection numbers $[u_i] \cdot \Delta$.

\begin{lemma}\label{l:branchedpoints0}
Suppose $\Sigma=\PP^1$ and the $B_i$ are discs for $i=1,\dots,k=s-1$. 
Suppose also that $a_i \in B_i^{\circ}$ for $i=1,\dots,k+1$ and $[u_i]$ is as in Corollary \ref{c:discclass}.
We have $[u_i]\cdot \Delta=0$ for $1 \leq i\leq k$ and $[u_{k+1}] \cdot \Delta = 2(k-1).$ 
\end{lemma}

\begin{proof}
Let $\WH{S}'=\sqcup_{j=1}^k S_j$ where $S_j=S$ is a unit disc for each $j$.
For $i=1,\dots,k$, let $v'_i:\WH{S}' \to S^2$ be a map such that $v'_i|_{S_j}$ is a constant map to a point in $L_j$ if $j \neq i$
and $v'_i|_{S_i}$ is a biholomorphism to $B_i$.
Let $\pi_{\WH{S}'}:\WH{S}' \to S$ be the trivial covering  map.
Since $[v_i'] \cdot a_j =\delta_{ij}$ for $j=1,\dots,k+1$ (see Remark \ref{r:symD}), the map $u'_i:S \to X$ obtained from the tautological correspondence from $(v'_i,\pi_{\WH{S}'})$ satisfies $[u'_i] \cdot D_j=\delta_{ij}$ for $j=1,\dots,k+1$.
Therefore, $u_i'$ represents the class $[u_i]$.
Since $\pi_{\WH{S}'}$ is a trivial covering and $\{v_i'|_{S_j}\}_{j=1}^k$ have pairwise disjoint images, the image of $u'_i$ is disjoint from $\Delta$ so we have $[u'_i]\cdot \Delta=0$ for $i=1,\dots,k$.

On the other hand, if  $\WH{S}''=\PP^1$, the Riemann-Hurwitz formula shows that a simple $k$-fold branched covering $\pi_{\WH{S}''}: \WH{S}'' \to \PP^1$ has $2(k-1)$ branch points because by definition, every critical point of a simple branched covering is of multiplicity $2$ and there is at most one critical point over each branch point.
Let $v'':\WH{S}'' \to \PP^1$ be a biholomorphism, and $u'' :\PP^1 \to X$ be the $J_X$-holomorphic map tautologically corresponding to $(v'',  \pi_{\WH{S}''})$.
We know that $u''$ represents the class $\sum_{i=1}^{k+1} [u_i]$, because $[v''] \cdot a_j=[u''] \cdot D_j=1$ for every $j$.
Since $v''$ is a biholomorphism, $u''(z) \in \Delta$ if and only if $z$ is a branch point of $\pi_{\WH{S}''}$.
The assumption that $\pi_{\WH{S}''}$ is a simple branched covering guarantees that the intersection multiplicity between $u$
and $\Delta$ at every branch point of $\pi_{\WH{S}''}$ is $1$ (this fact can be checked by a local calculation).
Therefore, we know that $[u''] \cdot \Delta =2(k-1)$ because $\pi_{\WH{S}''}$ is a simple branched covering with  $2(k-1)$ branch points. 
As a result, $[u_{k+1}] \cdot \Delta=([u'']-\sum_{i=1}^k [u_i']) \cdot \Delta=2(k-1)$.
\end{proof}

In the situation of Lemma \ref{l:branchedpoints0}, $\tau_i=1$ for $i \leq k$ and $\tau_{k+1}=k$, so 
one can write the conclusion as saying that
\begin{equation} \label{eqn:will_generalise}
[u_i]\cdot \Delta=2(\tau_i-1) \quad \text{for } i=1,\ldots, s=k+1.
\end{equation}
We next establish the analogue of Corollary \ref{c:discclass}, and in particular establish \eqref{eqn:will_generalise},  for general $\Sigma$ and $\ul{L}$. Recall that  $B_1,\ldots, B_s$ enumerate the closures of the planar regions comprising $\Sigma \backslash \ul{L}$. Pick a point $a_i \in B_i^{\circ} \subset B_i$ for each $i$.
Let $D_i$ be the divisor of $\Sym^k(\Sigma)$ which is the image of the map (cf. Remark \ref{r:symD})
\begin{align*}
\Sym^{k-1}(\Sigma) &\to \Sym^k(\Sigma) \\
D &\mapsto D+a_i.
\end{align*}
Let $\Delta$ be the diagonal.
\medskip

For each $i$, we can construct a continuous map $u_i:S \to (X, \Sym(\ul{L}))$ using a
pair of maps $v_i$ and $\pi_{\WH{S}_i}$ as in the proof of Lemma \ref{l:branchedpoints0}.
More precisely, let 
\[
\WH{S}_i=B_i \sqcup \left(\sqcup_{j=1}^{k-k_i} S_j\right) \text{ where } S_j=S \text{  for all } j.
\]
Let $\pi_{\WH{S}_i}: \WH{S}_i \to S$ be a $k$-fold branched covering such that $\pi_{\WH{S}_i}|_{S_j}$   is a biholomorphism and  $\pi_{\WH{S}_i}|_{B_i}$ is a 
$\tau_i$-fold simple branched covering to  $S$. 
Let $v_i: \WH{S}_i \to \Sigma$ be such that $v_i|_{B_i}$ is the identity map to $B_i$ and the $v_i|_{S_j}$ are constant maps to the various connected components of $\ul{L}$ that are not boundaries of $B_i$.
We define $u_i:=v_i \circ \pi_{\WH{S}_i}^{-1}$.
It is clear that
\begin{equation}
[u_i] \cdot D_j=\delta_{ij}.\label{eq:intersection-with-Di}
\end{equation}

\begin{lemma}\label{l:spherehom}
The image of $\pi_2(X,\Sym(\ul{L})) \to H_2(X,\Sym(\ul{L}))$ is freely generated by $\{[u_i]\}_{i=1}^s$. The image of $\pi_2(X) \to H_2(X,\Sym(\ul{L}))$ is  freely generated by $\sum_{i=1}^s[u_i]$.
\end{lemma}

\begin{proof} Let $A$ denote the image of $\pi_2(X, \Sym(\ul{L})) \to H_2(X,\Sym(\ul{L}))$, and $\zeta:A\to\Z^s$ the linear map given by
  \[a\mapsto(a\cdot D_1,\dots, a\cdot D_k).\]
Equation (\ref{eq:intersection-with-Di}) implies that $(\zeta(u_1),\dots, \zeta(u_s))$ is the standard basis of $\Z^s$. In particular, $\zeta$ is surjective. To prove that the $[u_i]$ freely generate $A$, there remains to prove that $\zeta$ is injective hence an isomorphism.
  
  We have the following commutative diagram with exact rows.
\[
\xymatrix{
0 \ar[r] \ar[d] &\pi_2(X) \ar[r] \ar[d] & \pi_2(X,\Sym(\ul{L})) \ar[r] \ar[d]  & \pi_1(\Sym(\ul{L}))   \ar[r] \ar[d]  &  \pi_1(X) \ar[d] \\
H_2(\Sym(\ul{L}))  \ar[r]      & H_2(X) \ar[r]            &  H_2(X,\Sym(\ul{L})) \ar[r]  &   H_1(\Sym(\ul{L}))   \ar[r] &H_1(X)
}
\]
In this diagram, the rows correspond to the relative long exact sequences of the pair $(X,\Sym(\ul L))$ respectively in homotopy and homology, and the vertical arrows are given by the Hurewicz map. The top left entry is $0$, since $\Sym(\ul L)$ is a torus hence has vanishing $\pi_2$.

Let $I:=\operatorname{im}(\pi_2(X) \to \pi_2(X,\Sym(\ul{L})))$
and $K:=\ker(\pi_1(\Sym(\ul{L})) \to \pi_1(X))$ so we have a short exact sequence
\begin{align*}
0 \to I \to  \pi_2(X,\Sym(\ul{L})) \to K  \to 0.
\end{align*}
The image of $\pi_2(X) \to H_2(X)$ is isomorphic to $\mathbb{Z}$
 (see \cite[Theorem 9.2]{BT01}).
Therefore, the rank of the image of $I \to H_2(X,\Sym(\ul{L}))$ is  at most $1$.

On the other hand, we have isomorphisms  $\pi_1(X)=H_1(X)=H_1(\Sigma)$ (see e.g. \cite[Lemma 2.6]{OS04}),  $\pi_1(\Sym(\ul{L}))=H_1(\Sym(\ul L))= H_1(\ul L)$.
Moreover, the map $\pi_1(\Sym(\ul{L})) \to \pi_1(X)$ can be identified with the map $H_1(\ul{L}) \to H_1(\Sigma)$ via the commutative diagram\footnote{As explained in the proof of Lemma 2.6 in \cite{OS04}, the right vertical arrow in this diagram is given by the fact that a closed (generic) curve in $X$ corresponds to a map from a $k$-fold cover of $\S^1$ to $\Sigma$ hence yields a homology class in $H_1(\Sigma)$. This construction sends in particular the standard generators of $H_1(\Sym(\ul L))$ to the standard generators of $H_1(\ul L)$, whence the commutativity of the diagram.}
\[
\xymatrix{  
   \pi_1(\Sym(\ul{L}))=H_1(\Sym(\ul{L})) \ar[r] \ar[d]^\simeq & \pi_1(X)=H_1(X) \ar[d]^\simeq\\
  H_1(\ul L)  \ar[r] & H_1(\Sigma).
}
\]
Now the map  $H_1(\ul{L}) \to H_1(\Sigma)$ also sits inside the relative long exact sequence for the pair $(\Sigma,\ul{L})$
\begin{align*}
H_2(\Sigma) \xrightarrow{f_1} H_2(\Sigma, \ul{L}) \xrightarrow{f_2} H_1(\ul{L}) \xrightarrow{f_3} H_1(\Sigma)
\end{align*}
where $H_2(\Sigma)=\mathbb{Z}$, $H_2(\Sigma, \ul{L})=\mathbb{Z}^s$ and $f_1$ is injective.
Therefore, we have $K=\ker(f_3)=\operatorname{im}(f_2)=\operatorname{coker}(f_1)=\mathbb{Z}^{s-1}$. 

Since $K$ is free, we have $\pi_2(X,\Sym(\ul{L})) \simeq I \oplus K$.
Therefore, the image of $\pi_2(X,\Sym(\ul{L})) \to H_2(X,\Sym(\ul{L}))$ (which we called $A$) is isomorphic to the image of a linear map 
$\mathbb{Z}^s \to H_2(X,\Sym(\ul{L}))$.
Composing with $\zeta$, we obtain a linear map \[\Z^s\to A \stackrel{\zeta}{\to}\Z^s\] which is surjective (as a composition of two surjective maps), hence an isomorphism. As a consequence, the first map $\Z^s \to A$ is injective.
This implies that the first map is an isomorphism and hence $\zeta$ is an isomorphism as well.

This shows that $\{[u_i]\}_{i=1}^s$ freely generates the image of $\pi_2(X,\Sym(\ul{L})) \to H_2(X,\Sym(\ul{L}))$.
Moreover, we know that the image of $I \to H_2(X,\Sym(\ul{L}))$ is isomorphic to $\mathbb{Z}$.
To conclude the proof, it suffices to find a continuous map $u:\PP^1 \to X$ representing the class $\sum_{i=1}^s[u_i]$.

We can construct $u$ using tautological correspondence.
Let $\WH{S}=\Sigma$ and $v:\WH{S} \to \Sigma$ be the identity map.
Let $\pi_{\WH{S}}:\WH{S} \to \PP^1$ be a topological $k$-fold simple branched covering.
The map $u= v \circ \pi_{\WH{S}}^{-1}$ satisfies $[u] \cdot D_i=1$ for all $i=1,\dots,s$, so we have $[u]=\sum_{i=1}^s[u_i]$.
\end{proof}

\begin{remark}
$\pi_2(\Sym^k(\Sigma))$ may have rank $> 1$ (see \cite[Theorem 5.4]{BR14}). The hypothesis on the link $\ul{L}$ (that the $B_j$ are planar) implies that the number of components $k\geq g+1$, where $g$ is the genus of $\Sigma$. If we  restrict to links with $k\geq 2g-1$ components,  then $\Sym^k(\Sigma)$ is a projective bundle over $\mathrm{Jac}(\Sigma)$, and $\pi_2(\Sym^k(\Sigma))  = \Z$ (see \cite[Ch VII, Proprosition 2.1]{ACGH1}); this gives a simpler proof of Lemma \ref{l:spherehom} for such cases. 
\end{remark}

In \cite[Section 7]{TimHandle}, Perutz explains how, given an open neighbourhood $V\supset \Delta$ of the diagonal, one can modify $\omega_X$ inside $V$, and in particular away from $\Sym(\ul{L})$ if $V$ is sufficiently small,  to get a smooth K\"ahler
form $\omega_V$ such that 
\begin{equation} \label{eqn:scaling_of_kahler}
[\omega_V]=(1/k!)\, \pi_*[\omega_M]=:[\omega_X].
\end{equation}
 The space of K\"ahler forms one obtains in this way (as $V$ varies) is connected. We will refer to such forms as being of `Perutz-type'.

\begin{definition}[Topological energy]\label{d:topologicalenergy}
Let $\omega_V$ be a Perutz-type K\"ahler form smoothing the current $\omega_X$. Then we set 
\begin{equation} \label{eqn:via_Perutz}
\omega_X(u):=\omega_V(u)
\end{equation}
for any $u \in H_2(X,\Sym(\ul{L}))$ in the span of the $\{[u_i]\}_{i=1}^s$. The definition is independent of the choice of $V$ as long as $\Sym(\ul{L}) \cap V=\emptyset$.
\end{definition}

The following definition is a variant of that from \cite{OS04}.

\begin{definition}\label{d:choiceJ}
Let $V$ be an open neighborhood of $\Delta \cup \left(\cup_{i=1}^s D_i\right)$.
The space $\mathcal{J}(V)$ of nearly symmetric almost complex structures on $X$ consists of those $J$ such that
\begin{itemize}
\item $J=J_X$ in $V$
\item $J$ tames $\omega_X$ outside $V$.
\end{itemize}
If $V$ is only an open neighborhood of $\Delta$, then we use $\mathcal{J}_{\Delta}(V)$ to denote the space satisfying the two conditions above.
\end{definition}

\begin{remark}
Note that $\omega_V$ tames $J$ for any $J \in \mathcal{J}_{\Delta}(V)$ and any choice of Perutz-type K\"ahler form $\omega_V$ as above. 
\end{remark}

When we consider $J$-holomorphic maps with boundary on $\Sym(\ul{L})$ for some $J \in \mathcal{J}(V)$ or $J \in \mathcal{J}_{\Delta}(V)$, we always assume that the open neighborhood 
 $V$ is disjoint from $\Sym(\ul{L})$.
When the particular choice of $V$ is not important, we will write $\cJ$ and $\cJ_{\Delta}$ 
for $\cJ(V)$ and $\cJ_{\Delta}(V)$, respectively.  Since $\Sym(\ul{L})$ is totally real with respect to any $J \in \cJ_{\Delta}$,  a smooth disc $(S,\partial S) \to (X,\Sym(\ul{L}))$ has a well-defined Maslov index with respect to any such $J$.

\begin{lemma}\label{l:Maslov}
If $u:(S,\partial S) \to (X,\Sym(\ul{L}))$ has class $[u]=\sum_i c_i [u_i] \in H_2(X,\Sym(\ul{L}))$, its Maslov index is $2\sum_i c_i=2\sum_i [u] \cdot D_i$ with respect to $J \in \cJ_{\Delta}$.
\end{lemma}

\begin{proof}
It suffices to prove that $\mu(u_i)=2$ for all $i$. 
Since $\cJ_{\Delta}$ is connected, it suffices to consider $J_X$.

Let $(v_i,\pi_{\WH{S}_i})$ tautologically correspond to $u_i$. Recall that $\WH{S}_i=B_i \sqcup_{j=1}^{k-\tau_i} S_j$, $v_i|_{B_i}$ is the identity map and  the $v_i|_{S_j}$ are constant maps. 
It follows that $u_i$ factors through the following holomorphic embedding (i.e. $\im(u_i)$ lies inside the image of the following map)
\begin{align}
\Sym^{\tau_i}(B_i) \times  \prod_{L_j \nsubseteq \partial B_i} D^*L_j  &\to \Sym^{k}(\Sigma)=X \label{eq:prodde}\\
([x_1,\dots,x_{\tau_i}], p_1,\dots,p_{k-\tau_i})&\mapsto [x_1,\dots,x_{\tau_i},p_1,\dots,p_{k-\tau_i}]
\end{align}
where $D^*L_j$ is a neighborhood of $L_j \subset \Sigma$ such that $\{B_i \} \cup \{ D^*L_j\}_{L_j \nsubseteq \partial B_i}$ are pairwise disjoint.
With respect to the product decomposition of the LHS of \eqref{eq:prodde}, we can write $u_i=(\bar{u}_i, c_1,\dots,c_{k-\tau_i})$ where $\bar{u}_i:(S,\partial S) \to (\Sym^{\tau_i}(\Sigma), \Sym(\partial B_i))$ and $c_i$ are constant maps.
It follows that $u_i^*(TX, T\Sym(\ul{L}))$ has $k-\tau_i$ trivial factors, which contribute $0$ to the Maslov index.
Therefore, it suffices to prove that $\mu(\bar{u}_i)=2$. 
Notice that $\bar{u}_i$ tautologically corresponds to the pair $(v_i|_{B_i}, \pi_{\WH{S}_i}|_{B_i})$. 
Since $B_i$ is a planar domain, we may choose an embedding $B_i \hookrightarrow \PP^1$ to obtain  a map (of the same Maslov index) $\bar{u}_i:S \to (\Sym^{\tau_i}(B_i), \Sym(\partial B_i)) \subset (\Sym^{\tau_i}(\PP^1), \Sym(\partial B_i))$.
By Corollary \ref{c:discclass}, the Maslov index of $\bar{u}_i$ is $2$.
\end{proof}

\begin{lemma}\label{l:branchedpoints}
For $u_i$ as in Lemma \ref{l:Maslov}, we have $[u_i]\cdot \Delta=2(\tau_i-1)$ for $i=1,\dots,s$.
\end{lemma}

\begin{proof}
We use the notation of the proof of Lemma \ref{l:Maslov}.
Since $v_i|_{S_j}$ are constant maps,  we have $[u_i]\cdot \Delta=[\bar{u}_i]\cdot \bar{\Delta}$
where $\bar{\Delta}$ is the diagonal in $\Sym^{\tau_i}(B_i)$.
By regarding $\bar{u}_i$ as a map from $S$ to $\Sym^{\tau_i}(B_i) \subset \Sym^{\tau_i}(\PP^1)$,
we can apply Lemma \ref{l:branchedpoints0} and (\ref{eqn:will_generalise}) to conclude the result. 
\end{proof}

\begin{corol}\label{c:spherebubble}
If $u:\PP^1 \to X$ is a non-constant $J$-holomorphic map for some $J \in \mathcal{J}_{\Delta}$, then $\mu(u) \ge 4$ and $[u] \cdot \Delta \ge \sum_{i=1}^s 2(\tau_i-1)$. 
\end{corol}

\begin{proof}
Suppose $J \in \mathcal{J}$. By Lemma \ref{l:spherehom},
$[u]$ is a multiple of $\sum_{i=1}^s [u_i]$. By positivity of intersection with $D_i$, $u$ is a positive multiple of $\sum_{i=1}^s [u_i]$. Hence the result follows from Lemma \ref{l:Maslov} because $B_i$ being all planar implies that $s \ge 2$. 

 Now suppose $J \in \mathcal{J}_{\Delta}$. If the image of $u$ is contained in $\Delta$, then $u$ is actually $J_X$-holomorphic and we reduce to the previous case. If the image of  $u$ is not contained in $\Delta$, we have positivity of intersection between $u$ and $\Delta$, so  $[u]$ is still a positive multiple of $\sum_{i=1}^s [u_i]$. The result follows from Lemma \ref{l:branchedpoints}. 
\end{proof}

\begin{remark}\label{rmk:first_chern}
Let $x$ be  Poincar\'e dual to the divisor $D_p = \{p\}\times\Sym^{k-1}(\Sigma)$ and $\theta$ be the pullback of the theta-divisor from the Jacobian under the Abel-Jacobi map (for relevant background see \cite{GH, ACGH1}).  
The first Chern class of $X$ is $-\theta- (g-k-1)x$ (see \cite[Ch VII, Section 5]{ACGH1}).
When $s \ge 2$ and hence $k+1-g \ge 2$, we have $\langle c_1(X), [u] \rangle =[u] \cdot (-\theta - (g-k-1)x)=-(g-k-1)[u] \cdot x \ge 2$, for any $u: \PP^1 \to X$. Recalling that the Maslov index of such a holomorphic $u$ (viewed as a disc with trivial boundary condition) is given by twice its first Chern number, this gives a more direct proof that  $\mu(u) \ge 4$ for sphere components $u$. 
\end{remark}

\begin{lemma}[Monotonicity]\label{l:unobstructed}
Suppose that there is an $\eta \ge 0$ such that $A_j+2(\tau_j-1)\eta$ is independent of $j$ and denote this common value by $\lambda$.
Then 
for all $u \in \pi_2(X,\Sym(\ul{L}))$, we have
\begin{align}
\omega_X(u)+\eta[u] \cdot \Delta =\frac{\lambda}{2} \mu(u). \label{eq:mon}
\end{align}
As a result, $\Sym(\ul{L})$ does not bound any non-constant $J$-holomorphic disc of non-positive Maslov index for any $J \in \mathcal{J}_{\Delta}$.
\end{lemma}

\begin{proof}
It is easy to check that $\omega_X(u_i)=A_i$ (see Definition \ref{d:topologicalenergy} and \eqref{eqn:scaling_of_kahler}).
Therefore, \eqref{eq:mon} is a direct consequence of applying Lemmas \ref{l:spherehom}, \ref{l:Maslov} and \ref{l:branchedpoints} to all the $[u_i]$.
The last sentence follows from the positivity of $\omega_X(u)$ and non-negativity of $\eta [u] \cdot \Delta $ for a non-constant $J$-holomorphic disc $u$ such that $J \in \mathcal{J}_{\Delta}$.
\end{proof}

 If $\ul{L}$ is not $\eta$-monotone, we still have the following.

\begin{lemma}\label{l:unobstructed2}
 The Lagrangian $\Sym(\ul{L})$ does not bound any non-constant $J$-holomorphic disc of non-positive Maslov index for any $J \in \mathcal{J}$.
\end{lemma}

\begin{proof}
For $J \in \mathcal{J}$, we have positivity of intersection between $D_i$ and a $J$-holomorphic disc $u$ with boundary on $\Sym(\ul{L})$. Therefore, Lemma \ref{l:Maslov} guarantees that $u$ has non-positive Maslov index if and only if $[u]=0$.
In this case, $u$ is a constant map.
\end{proof}

\begin{remark} \label{rmk:eta_values}
Suppose $\Sigma = \PP^1$ and $\ul{L} \subset \PP^1$ is an $\eta$-monotone link.  Suppose moreover that the total $\omega$-area of $\Sigma$ is $1$.  If the link has a unique component, necessarily it is an equator, which is $0$-monotone for $\eta = 0$. 
If $k>1$, there is at least one planar domain $B_j$ with $\tau_j \geq 2$, from which one sees that the monotonicity constant $\lambda := A_j+2(\tau_j-1)\eta > 2\eta$. 
On the other hand, we have
\begin{align*}
s\lambda&=\sum_{j=1}^s  A_j+2(\tau_j-1)\eta=1+2\eta \sum_{j=1}^s (\tau_j -1)=1+2\eta (s-2)
\end{align*}
where last equality uses $\sum_{j=1}^s \tau_j=2k=2(s-1)$.
It shows that $s(2\eta - \lambda) = 4\eta - 1$, so $\eta$-monotone links can only exist for $\eta \in [0,\frac14)$.  Moreover, links consisting of $k\geq 2$ parallel circles on the sphere  can take any value of $\eta \in [0, \frac14)$.  Hence, we see that the set of all values of $(k, \eta)$ for which there exists a $k$-component $\eta$-monotone link $\ul{L}$ with $k \ge 2$ is exactly  $ \left\lbrace (k,\eta) : k\in\N_{\ge 2}, \eta\in [0,\frac14)\right\rbrace.$
\end{remark}

If $\ul{L}$ is a $0$-monotone link (i.e. the areas of the $B_i$ are the same for all $i$), then $\Sym(\ul{L})$ is a monotone Lagrangian submanifold with respect to a Perutz-type K\"ahler form $\omega_{V}$ as in  Definition \ref{d:topologicalenergy} for any $V \supset \Delta$ disjoint from $\Sym(\ul{L})$.

When $g=0$ and $\ul{L}$ is an $\eta$-monotone link for $\eta > 0$,  we can `inflate'  the symplectic form  on $X$ near the diagonal to make $\Sym(\ul{L})$ a monotone Lagrangian submanifold,  as follows.

\begin{lemma}\label{rmk:monotone_Lagrangian}
When $g=0$ and $\ul{L}$ is an $\eta$-monotone link for $\eta > 0$, there is a symplectic form $\omega_{V,\eta}$ on $X$ such that $\omega_{V,\eta}=\omega_V$ outside $V$ and $\Sym(\ul{L})$ is a monotone Lagrangian with respect to $\omega_{V,\eta}$.
\end{lemma}

\begin{proof}
Let $V \supset \Delta$ be disjoint from $\Sym(\ul{L})$ and $\omega_{V}$ be a Perutz-type K\"ahler form.
Since $\Sym^k(\PP^1)=\PP^k$, we know that $\Delta$ is a very ample divisor and its complement is affine.
In particular, we can find a neighborhood $V'$ of $\Delta$ 
such that its closure $\bar{V}'$ is contained in $V$ and  $\bar{V}'$ admits a concave contact boundary (see e.g. \cite[Section 4b]{Seidel:bias}  for the existence of $V'$).

Let $X_-:=X \setminus V'$ be equipped with the restricted symplectic form $\omega_V|_{X_-}$. For $R>1$, let $X_{0,R}:=([1,R] \times \partial \bar{V}', d(r \theta))$ where $r$ is the coordinate on $[1,R]$ and $\theta$ is the contact form on $\partial \bar{V}'$ induced by $\omega_V$.
Let $X_{+,R}:=(\bar{V}', R\omega_V|_{\bar{V}'})$.  
We can form a symplectic manifold by gluing their boundaries
\begin{align*}
(X(R),\omega_{X(R)}):=X_- \cup_{\{1\} \times \partial \bar{V}'} X_{0,R} \cup_{\{R\} \times \partial \bar{V}'} X_{+,R}.
\end{align*}
We can also find a diffeomorphism $F:X \to X(R)$ such that $F$ is the identity map over $X_-$ and near $\Delta$.
The symplectic form $F^*\omega_{X(R)}$ lies in the cohomology class $[\omega_V]+f(R)\mathrm{PD}[\Delta]$ for a strictly increasing function $f$ such that $f(1)=0$ and $\lim_{R \to \infty}f(R)=\infty$.
Therefore, it is clear that $\Sym(\ul{L})$ is a monotone Lagrangian in $(X,F^*\omega_{X(R_{\eta})})$
for the $R_{\eta}$ such that $f(R_{\eta})=\eta$.

We denote $F^*\omega_{X(R_{\eta})}$ by $\omega_{V,\eta}$. The dependence on the choices made in the construction will not be important in the paper. 
\end{proof}

\begin{remark} \label{rmk:keep_eta}
When $\Sigma = \PP^1$,  the symplectic forms $\omega_{V,\eta}$ have cohomology class varying with $\eta$, cf. Remark \ref{rmk:cohomology_class_after_inflation}, but they can be rescaled to be cohomologous and hence isotopic, even as one varies $\eta$. They are therefore related by a global smooth isotopy, by Moser's theorem, so if $\omega(\Sigma)=1$ then $(X,\omega_{V,\eta})$ is 
symplectomorphic to the Fubini-Study form normalized so that the symplectic area of $[\PP^1]$ is $(k+1)\lambda$, where $\lambda=A_j+2(\tau_j-1)\eta$.

However, this isotopy will not respect the diagonal, and the resulting isotopy of $\Sym(\ul{L}) \subset \PP^k$ is not through Lagrangian submanifolds associated to links. For the purposes of studying links and the geometry of $\Sigma$, it therefore makes sense to keep track of $\eta$ even in this case. 
\end{remark}

This inflation process for higher genus $g$ does not work as explained in the following remark.

\begin{remark} 
Suppose $k \geq 2g-1$, so $\Sym^k(\Sigma_g)$ is a projective bundle $\PP(V)$ over the Jacobian. We follow the notation of Remark \ref{rmk:first_chern}.  If $\omega$ is an integral K\"ahler form on $\Sigma$ of area 1, the current $\omega_X$ defines the cohomology class $x$. (This is ample, and indeed the tautological class $\mathcal{O}_{\PP(V)}(1)$.)  The diagonal divisor $\Delta$ has class $2[(k+g-1)x-\theta]$. Cones of divisors of $\Sym^k(\Sigma)$ were studied in \cite{Kouvidakis, Pacienza}; the diagonal is on the boundary of the pseudo-effective cone. It follows that if $g \ge 1$, $\Delta$ is not ample and $[\omega_X] + \eta\cdot \mathrm{PD}[\Delta]$ will not lie in the ample cone for sufficiently large $\eta \gg 0$, which means that it cannot be the cohomology class of a K\"ahler form.
\end{remark}

\section{Unobstructedness}\label{s:potential}

Our next goal is to define a version of Floer cohomology for the torus $\Sym(\ul{L})$, and to determine when it is non-zero.  As in many examples of this nature, the non-triviality of the Floer cohomology will be determined by
the {\it disc potential function} associated to $\Sym(\ul{L})$ (see Definition \ref{defn:disc_potential} and Lemma \ref{l:critialW}).  We are going to compute the disc potential function in this section.

\subsection{The disc potential}

We recall the spaces of almost complex structures $\cJ(V)$ and $\cJ_{\Delta}(V)$ from Definition \ref{d:choiceJ}.
For a fixed $\ul{L} \subset \Sigma$ and hence $\Sym(\ul{L}) \subset X$,
we continue to use $\cJ$ (resp. $\cJ_{\Delta}$) to denote $\cJ(V)$ (resp. $\cJ_{\Delta}(V)$)
for an open neighborhood $V$ of $\Delta \cup \cup_{i=1}^s D_i$ (resp. $\Delta$) that is disjoint from $\Sym(\ul{L})$.
For $J \in \cJ$, consider the moduli space $\cM_A(\Sym(\ul{L});J)$ of Maslov index $2$ $J$-holomorphic discs $u:(S, \partial S) \to (X,\Sym(\ul{L}))$ with $1$ boundary marked point and in the relative homology class $A \in H_2(X,\Sym(\ul{L}))$.
The evaluation map at the boundary marked point defines a map  $ev: \cM_A(\Sym(\ul{L});J) \to \Sym(\ul{L})$.

\begin{lemma}
If $J\in \cJ$ is generic, $\cM_A(\Sym(\ul{L});J)$ is a compact manifold of dimension $k$. 
The same is true for generic $J\in \cJ_{\Delta}$ if $\ul{L}$ is $\eta$-monotone.
\end{lemma}

\begin{proof} 
By Lemma \ref{l:unobstructed} and \ref{l:unobstructed2}, $\Sym(\ul{L})$ does not bound non-constant $J$-holomorphic discs with non-positive Maslov index,  so the Gromov compactification of $\cM_A(\Sym(\ul{L});J)$ is the space itself. The condition $\mu(A)=2$ implies that $A$ is primitive because $\Sym(\ul{L})$ cannot bound discs of Maslov index $1$. Therefore discs in class $A$ are necessarily somewhere injective, by \cite{Lazzarini}. 
The existence of a somewhere injective point implies that elements in  $\cM_A(\Sym(\ul{L});J)$ are regular for a generic $J \in \cJ$ or a generic $J\in \cJ_{\Delta}$ (see \cite[Theorem 10.4.1, Corollary 10.4.8]{OhBook1}).
In this case, $\cM_A(\Sym(\ul{L});J)$ is a manifold of dimension the same as the virtual dimension which equals to $(k-3)+ \mu(A)+1=k$.
\end{proof}

A choice of orientation and spin structure on $\Sym(\ul{L})$ defines an orientation of $\cM_A(\Sym(\ul{L});J)$, with respect to which the evaluation map has a well-defined degree.  
Equivalently, the fiber product between $\cM_A(\Sym(\ul{L});J)$ and a generic point in $\Sym(\ul{L})$ under  the evaluation map therefore defines a compact oriented zero dimensional manifold. In a minor abuse of notation, we denote the algebraic count of points of this $0$-manifold by $\# \cM_A(\Sym(\ul{L});J)$. 
\medskip

\begin{definition} \label{defn:disc_potential}
For $\eta$-monotone $\ul{L}$ and generic $J \in \mathcal{J}_{\Delta}$, the disc potential function 
\[
W := W_{\Sym(\ul{L})}(\cdot,J) : H^1(\Sym(\ul{L});\C^*) \longrightarrow  \C
\]
is defined by 
\begin{align}\label{eq:PotentialDef}
W_{\Sym(\ul{L})}(x,J)= \sum_{A \in H_2(X,\Sym(\ul{L}))} (\# \cM_A(\Sym(\ul{L});J))x^{\partial A}.
\end{align}
\end{definition}

The notation $x^{\partial A}$ is defined to be $x(\partial A)$ using the pairing
$H^1(\Sym(\ul{L});\C^*) \times H_1(\Sym(\ul{L})) \to \C^*$.
More explicitly, let $\{q_1,\dots,q_k\}$ be a basis of $H_1(\Sym(\ul{L}),\Z)$.
We have $\partial A=\sum_{i=1}^k c_iq_i$ for some $c_i \in \Z$.
In coordinates,  we have $x^{\partial A}=\prod_{i=1}^k x_i^{c_i}$, where $\{x_i\}_{i=1}^k$ is dual to $\{q_i\}_{i=1}^k$.

\begin{remark}\label{r:Offcodim2}
When elements in $\cM_A(\Sym(\ul{L});J)$ are regular off a set of real codimension $2$, the degree of the evaluation map is still well-defined (see \cite[Chapter 6.5 and 6.6]{mcduff-salamon2}).
In this case, $\# \cM_A(\Sym(\ul{L});J)$ is well-defined and the potential function is defined in the same way as in \eqref{eq:PotentialDef}.
\end{remark}

The potential function depends on the choice of orientation and spin structure on $\Sym(\ul{L})$, but these choices will not play a significant role in the sequel (we will be interested in  the existence of  critical points of the disc potential; a different choice of orientation or spin structure will change the value of the critical point, not the existence of critical points). Concretely, we will fix an orientation by orienting and ordering the constituent circles $L_i \subset \ul{L}$, and will take the unique  translation-invariant spin structure (this follows the usual convention for Lagrangian toric fibres from \cite{Cho-Clifford, Cho-Oh}). 
\medskip

We compute $W_{\Sym(\ul{L})}(x,J)$ in the subsequent sections.

\begin{definition} \label{rmk:Novikov_disc_potential}
Let $\Lambda$ be the Novikov field with real exponent. That is
\begin{align*}
\Lambda:=\left\{\sum_{i=0}^{\infty} c_iT^{b_i} | c_i \in \C, b_i \in \R,  b_0<b_1<\dots, \lim_{i \to \infty} b_i =\infty \right\}.
\end{align*}

The non-Archimedean valuation $val: \Lambda\setminus \{0\} \to \R$ is defined to be $val(\sum_{i=0}^{\infty} c_iT^{b_i})=\min\{b_i| c_i \neq 0\}$.
For not necessarily $\eta$-monotone $\ul{L}$ and generic $J \in \mathcal{J}$, we define the {\it $\eta$-disc potential function} as a function $H^1(\Sym(\ul{L});U_{\Lambda}) \to \Lambda\setminus\{0\}$, where $U_{\Lambda} = val^{-1}(0)$  is the unitary subgroup of  $\Lambda$. 

In that case, the $\eta$-disc potential is given by
\begin{align*}
W_{\Sym(\ul{L})}^{\eta}(x,J)=\sum_{A \in H_2(X,\Sym(\ul{L}))} (\# \cM_A(\Sym(\ul{L});J))\, T^{\omega_X(A)+\eta A \cdot \Delta}\;x^{\partial A}.
\end{align*}
When $\ul{L}$ is $\eta$-monotone, then $W_{\Sym(\ul{L})}=W_{\Sym(\ul{L})}^{\eta}|_{T=1}$.
\end{definition}

\subsection{Potential in the Clifford-type case}

We return to the running example in which $\ul{L}=\ul{L}_{a_1,\dots,a_k,\eps}=\cup_{i=1}^k \{|x-a_i||\prod_{j \neq i} (a_j-a_i)|=\eps \}$ from above.
The general case will be explained in the next section.
\medskip

We orient the circles as boundaries of complex discs in $\C$, and take the product orientation on $\Sym(\ul{L})$. The fundamental classes of the circles $L_i \subset \ul{L}$ also give us preferred basis co-ordinates $x_i$ on $H^1(\Sym(\ul{L});\C^*)$. 
We use the notation from Section \ref{s:set-up} that $A_i:=\omega(B_i)$ for $i=1,\dots,k+1$.

\begin{prop}\label{p:SymL=Clifford}
Let $\ul{L}=\ul{L}_{a_1,\dots,a_k,\eps}$.
For sufficiently small $\eps>0$, sufficiently small open sets $V \supset \Delta \cup \left(\cup_{i=1}^s D_i\right)$ disjoint from $\Sym(\ul{L})$ and generic $J \in \mathcal{J}(V)$, we have $W_{\Sym(\ul{L})}^{\eta}(x,J)=\sum_{i=1}^k T^{A_i} x_i+\frac{T^{A_{k+1}+2(k-1)\eta}}{x_1 \dots x_k}$.
\end{prop}

\begin{proof}

We return to the setting and notation of Lemma \ref{l:realFamily}, Corollary \ref{Cor:Heegaard_is_Clifford} and their proofs. Recall that this introduced a small $\eps > 0$ and an open neighbourhood $G$ of $L_{\eps}$.
Since $(\Phi^t_G)^*$ is $C^1$-small when $\eps$ is small, we can assume that $(\Phi^t_G)^*\omega_{X}$ tames the standard complex structure $J_X$ for all $t$. We fix once and for all such an $\eps$,  and recall that $\Phi^t_G$ is supported away from  $\Delta$.
The idea of the proof is to show that under the identification $H_2(X,\Sym(\ul{L}))=H_2(X,L_{\eps})$ induced by $\Phi^t_G$, we have
$\# \cM_A(\Sym(\ul{L});J)=\# \cM_A(L_{\eps};J')$ for appropriate almost complex structures $J$ and $J'$. 
\medskip

The family of diffeomorphisms $\Phi_G^t$ identifies
the $4$-tuple
$(X,J_X,(\Phi_G^t)^*\omega_{X}, \Phi_G^{-t}(L_{\eps}))$ with
$(X,(\Phi_G^t)_*J_X,\omega_{X}, L_{\eps})$,
so we can take the perspective that $\Phi_G^t$ induces a  one-parameter family of $\omega_{X}$-tamed almost complex structures $(\Phi_G^t)_*J_X$, and that we work 
with a fixed Lagrangian and a fixed symplectic form (more properly, a fixed symplectic current which is singular along $\Delta$).
Note that $(\Phi_G^t)_*J_X=J_X$ near $\Delta$ and $\cup_{i=1}^{k+1} D_i$ is preserved under $\Phi^t_G$.
Therefore, we may fix an open neighborhood $V$ of $\Delta \cup \cup_{i=1}^{k+1} D_i$
such that $(\Phi_G^t)_*J_X=J_X$ in $V$. It follows that $(\Phi_G^t)_*J_X \in \mathcal{J}(V)$ for all $t$.
\medskip

For any $J \in \cJ(V)$ a $J$-holomorphic disc $u$ with boundary on $L_{\eps}$ has Maslov index (see Lemma \ref{l:Maslov})
\begin{align} \label{eqn:index}
\mu(u)=2 \sum_{j=0}^k [u] \cdot  [D_j].
\end{align} 
By positivity of intersections, $L_{\eps}$ cannot bound non-constant $J$-holomorphic discs with non-positive Maslov index for any $J \in \cJ(V)$.
\medskip

It remains to relate the $(\Phi_G^t)_*J_X$-holomorphic discs with Maslov index $2$ for $t=0,1$.
When $t=0$, we have $(\Phi_G^t)_*J_X=J_X$ and the Maslov two discs with boundary on $L_{\eps}$ are well-known to be regular \cite{Cho-Clifford, Cho-Oh}.
\medskip

At this point, we do not know that the Maslov two discs for $(\Phi^1_G)_*J_X$ are regular.  We instead choose a generic $C^2$-small perturbation $J'_t$ of the path $((\Phi_G^t)_*J_X)_{t \in [0,1]}$ relative to the end-point $t=0$ (but not necessarily fixing the end-point at $t=1$). In particular, $J'_1$ is a generic perturbation of $(\Phi_G^1)_*J_X$.
\medskip

 The parametrized moduli space of Maslov two $J'_t$-holomorphic discs $u$ with boundary on $L_{\eps}$ for some $t$ could in general fail to be regular: there can be  finitely many interior times $t$ where bifurcation occurs. 
A necessary condition for bifurcation to occur at time $t_0$ is that there are at least two non-constant $J'_{t_0}$-holomorphic discs $v, v'$ with $\mu(v) + \mu(v') \leq 2$. At least one of $v, v'$ then has Maslov index strictly less than $2$, and hence (by orientability of $\Sym(\ul{L})$) index less than or equal to $0$,  which contradicts \eqref{eqn:index}.  Therefore, there is no bifurcation and the parametrized moduli space for a generic path $J'_t$ is a smooth compact cobordism between the moduli spaces for $t=0,1$.
\medskip

Since $J'_1$ is a generic perturbation of $(\Phi_G^1)_*J_X$, this implies that for generic $J \in \mathcal{J}(V)$, the algebraic counts of Maslov two $J$-holomorphic discs with boundary on $\Sym(\ul{L})$ are the same as those of the Clifford-type torus $L_\eps$.
The result now follows from Lemma \ref{l:branchedpoints0} and the fact that $\# \cM_{[u_i]}(L_{\eps};J_X)=1$ for all $i=1,\dots,k+1$
and $\# \cM_{A}(L_{\eps};J_X)=0$ for $A \neq [u_i]$
(see \cite{Cho-Clifford}).  
\end{proof}

Now we consider a slightly more general class of $\ul{L}$ in $\Sigma =\PP^1$.
We still assume that $B_j$ are pairwise disjoint topological discs with smooth boundary for $j=1,\dots,k$
but we do not require that $\ul{L}=\ul{L}_{a_1,\dots,a_k,\eps}$.

\begin{prop}\label{p:SymL=Clifford2}
Let $B_j \subset \PP^1$ are pairwise disjoint topological discs with smooth boundary for $j=1,\dots,k$ and $\ul{L}=\cup_{i=1}^k \partial B_i$.
For sufficiently small open set $V \supset  \Delta \cup \cup_{i=1}^s D_i$ and generic $J \in \mathcal{J}(V)$, we have $W_{\Sym(\ul{L})}^{\eta}(x,J)=\sum_{i=1}^k T^{A_i} x_i+\frac{T^{A_{k+1}+2(k-1)\eta}}{x_1 \dots x_k}$.

Moreover, if $\ul{L}$ is $\eta$-monotone, then for generic $J \in \mathcal{J}_{\Delta}$ we have 
 $W_{\Sym(\ul{L})}(x,J)=\sum_{i=1}^k  x_i+\frac{1}{x_1 \dots x_k}$.
\end{prop}

\begin{proof}
Similar to the proof of Corollary \ref{c:discclass}, we can find a smooth family of $(\ul{L}_t)_{t \in [0,1]}$
such that $\ul{L}_0=\ul{L}$ and $\ul{L}_1=\ul{L}_{a_1,\dots,a_k,\eps}$ for some $a_i$ and small $\eps$.
We can assume that $\ul{L}_t$ is disjoint from $\{a_1,\dots,a_k\}$ for all $t \in [0,1]$.
We can assume that $V$ is disjoint from $\Sym(\ul{L}_t)$ for all $t$.
By Lemma \ref{l:unobstructed2}, $\Sym(\ul{L}_t)$ does not bound non-constant $J$-holomorphic disc with non-positive Maslov index for all $J \in \mathcal{J}(V)$.
As in the proof of Proposition \ref{p:SymL=Clifford}, we can form a smooth compact cobordism between the moduli spaces of Maslov two holomorphic discs for $t=0,1$.
This proves the first statement.

For the second statement, we want to show that the potential function can be computed for generic $J \in \mathcal{J}_{\Delta}$ that are not necessarily in $\mathcal{J}(V)$. 
It follows from applying a further cobordism argument,  using Lemma \ref{l:unobstructed} instead of \ref{l:unobstructed2}, to a family of almost complex structures in $\mathcal{J}_{\Delta}$. 
\end{proof}

\subsection{Regularity}
We can upgrade Proposition \ref{p:SymL=Clifford2} to a statement for the canonical complex structure $J_X$ if $J_{\Sigma}$ is chosen appropriately relative to $\ul{L}$\footnote{A similar claim is made in \cite[Proposition 3.9]{OS04}}.  This (as well as its generalization to the cases $\Sigma \neq \PP^1$) will be explained in Section \ref{s:potentialgeneral}. 
The key result we prove in this subsection is 
 that elements in $\cM_{[u_i]}(\Sym(\ul{L});J_X)$ are regular
off a set of codimension at least $2$ (see Corollary \ref{cor:regular_off_codim2}) when the complex structure on $B_i$ is appropriate in the same sense. 
We do not assume $\ul{L}=\ul{L}_{a_1,\dots,a_k,\eps}$ in this section.

The tautological correspondence of Section \ref{s:taut} shows that $S_k$-equivariant maps of any regularity  $V:(\WT{S},\partial \WT{S}) \to (\Sigma^k,\pi^{-1}(\Sym(\ul{L})))$ can be identified with 
maps (of the same regularity) from $v:(\WH{S}, \partial \WH{S}) \to (\Sigma,\ul{L})$,  see \cite[Section 3.1]{MS19} for more details.  There is a similar dictionary for maps valued in vector fields or endomorphisms. In particular, we have
\begin{align}
H^j_{\bar{\partial}}((\WT{S},\partial \WT{S}), (V^*T\Sigma^k, V|_{\partial \WT{S}}^*T(\pi^{-1}(\Sym(\ul{L})))))^{S_k}=H^j_{\bar{\partial}}((\WH{S}, \partial \WH{S}), (v^*T\Sigma, v|_{\partial \WH{S}}^*T\ul{L})) \label{eq:regular}
\end{align}
for all $j$, where $H^j_{\bar{\partial}}$ denotes Dolbeault cohomology $H^{0,j}_{\bar{\partial}}$.

\begin{remark} \label{rmk:dolbeault}
Let $(E,F) \to (\Sigma,\partial \Sigma)$ be a holomorphic vector bundle over a Riemann surface $\Sigma$ with totally real boundary condition $F \subset E|_{\partial \Sigma}$. There is a sheaf $\mathcal{O}(E,F)$ of locally holomorphic sections of $E$ with boundary values in $F$. Let $\bar\partial$ denote the standard Cauchy-Riemann differential operator in $T\Sigma$, restricted to elements with boundary values in $T\partial \Sigma$.  This defines an elliptic operator on suitable Sobolev completions of the space of $(E,F)$-valued smooth sections valued in $(E,F)$-valued $(0,1)$-forms.  The Dolbeault groups $H^{0,j}_{\bar{\partial}}(E,F)$, which are defined by the kernel (respectively cokernel) of $\bar\partial$ for $j=0$ (respectively $j=1$) (hence are relevant to the question of regularity of holomorphic curves, cf. Remark \ref{rmk:Floer_survey}), are isomorphic to the sheaf cohomology groups $H^j(\mathcal{O}(E,F))$ (which vanish if $j>1$).  The proof follows the usual case for bundles over closed Riemann surfaces, using a $\bar\partial$-Poincar\'e Lemma for functions on a half-plane which are real-valued on the boundary). In particular the Dolbeault groups satisfy long exact sequences for exact sequences of sheaves.  For general background, see \cite{Huybrechts} and \cite[Appendix C]{mcduff-salamon2}.
\end{remark}

\begin{prop}\label{p:regular}
Let $u:(S,\partial S) \to (X,\Sym(\ul{L}))$ be a $J_X$-holomorphic map and $(v,\pi_{\WH{S}}):\WH{S} \to \Sigma \times S$ the map tautologically corresponding to $u$.
Suppose that $v$ is regular and that 
$\pi_{\WH{S}}$ is a simple branched covering with $[u] \cdot \Delta$  simple branch points.
Then $u$ is regular.
\end{prop}

Before the proof, we formulate a lemma comparing virtual dimensions of the maps $u$ and $v$. 
Let $\WH{S}$ be a Riemann surface (so its conformal structure is fixed).
Let  $\vdim(v,\WH{S})$ be the virtual dimension of the space of maps $v:\WH{S} \to \Sigma$ with boundary on $\ul{L}$.
Let $\vdim(u)$ be  the virtual dimension of the moduli space of discs $u$,  where we divide out by the action of the $3$-dimensional automorphism group $\PP SL(2,\R)$ of $S$.

\begin{lemma}\label{l:simVdim}
Let $u$ and $(v,\pi_{\WH{S}})$ be as in Proposition \ref{p:regular}, then
\begin{align}
\vdim(u)+3=\vdim(v,\WH{S})+2[u] \cdot \Delta. \label{eq:regCom}
\end{align}
\end{lemma}

\begin{proof}
First, we recall that the virtual dimension of a pseudo-holomorphic map $u:S \to (X^{2k},J)$ with Lagrangian boundary condition $u(\partial S) \subset L$ from a compact Riemann surface $S$ is given by
\[
\WT{\vdim}(u)=k\chi(S)+\mu(u)
\] 
where $\chi(S)$ is the Euler characteristics of $S$ and $\mu(u)$ is the Maslov index of $u$ \cite[Proposition 11.13]{SeidelBook}.
The virtual dimension $\WT{\vdim}(u)$ does not take into account the automorphism of $S$ and other conformal structures on the underlying topological space of $S$.

We can write $[u]$ as a sum $\sum_{i=1}^s c_i[u_i]$ where $c_i \ge 0$ for all $i$.
The LHS of \eqref{eq:regCom} is
\begin{align}
\vdim(u)+3=(\dim_{\C}(X)\cdot \chi(S)+\mu(u)-3)+3=k+\mu(u)=k+2\sum_{i=1}^s c_i. \label{eq:chiS}
\end{align}
The term $-3$ in $\vdim(u)$ comes from dividing out  by the action of the $3$-dimensional automorphism group $\PP SL(2,\R)$ of $S$.

On the other hand, we have $[v]=\sum_{i=1}^s c_i[v_i]$ and
\begin{align}
\vdim(v,\WH{S})=\dim_{\C}\Sigma \cdot \chi(\WH{S})+ \mu(v)=\chi(\WH{S})+\sum_{i=1}^s c_i \mu(v_i)=\chi(\WH{S})+2\sum_{i=1}^s (2-\tau_i)c_i. \label{eq:chiS1}
\end{align}
where $\mu(v_i)$ is the Maslov index of the class $[v_i] \in H_2(\Sigma,\ul{L})$
and it is given by $2(2-\tau_i)$ because the inclusion $B_i \hookrightarrow \Sigma$ represents $[v_i]$
and the Maslov index of a planar domain with $\tau_i$ boundary components is $2(2-\tau_i)$.
By Lemma \ref{l:branchedpoints}, we have $[u]\cdot \Delta= \sum_{i=1}^s 2(\tau_i-1)c_i$.
Since we assume that  $\pi_{\WH{S}}$ is a simple branched covering with $[u]\cdot \Delta$ many branch points, the Riemann-Hurwitz formula yields
\begin{align}
\chi(\WH{S})&=k-\sum_{i=1}^s 2(\tau_i-1)c_i.  \label{eq:chiS2}
\end{align}
Combining \eqref{eq:chiS}, \eqref{eq:chiS1} and \eqref{eq:chiS2}, we get \eqref{eq:regCom}.
\end{proof}

\begin{proof}[Proof of Proposition \ref{p:regular}]
Recall from \eqref{eqn:tautological_square} the pull-back diagram
\[
\xymatrix{
\WT{S} \ar[rr]^{V} \ar[d]_{\pi_{\WT{S}}} & & \Sigma^k  \ar[d]_{\pi}\\
S \ar[rr]_{u} &&  X.
}
\]
We have a short exact sequence of sheaves over $\WT{S}$
\begin{align*}
0 \to V^*T\Sigma^k \to  \pi_{\WT{S}}^*(u^*TX)  \to Z \to 0
\end{align*}
where under the identification $V^*(\pi^*TX)=\pi_{\WT{S}}^*(u^*TX) $, the second arrow is induced by $\pi_*:T\Sigma^k \to \pi^*TX$  and $Z$ is defined to be the cokernel.
Here we are abusing notations and use $V^*T\Sigma^k$ and $\pi_{\WT{S}}^*(u^*TX)$ to denote the sheaves representing the respective vector bundles. The cokernel $Z$ is  a sheaf which does not come from a vector bundle.  
Indeed, since $\pi$ is a ramified covering, $Z$ is supported on the critical points of $\pi_{\WT{S}}$; at each critical point, the stalk has complex rank equal to the ramification index minus $1$, see \cite[Ch IV, Prop 2.2]{Hartshorne}.  Consider the induced long exact sequence in cohomology (where for simplicity we omit the boundary condition from the notation)
\begin{align*}
0 \to H^0_{\bar{\partial}}(\WT{S},V^*T\Sigma^k)  
\to H^0_{\bar{\partial}}(\WT{S},\pi_{\WT{S}}^*(u^*TX)  )  \to    H^0_{\bar{\partial}}(\WT{S},Z) \to H^1_{\bar{\partial}}(\WT{S},V^*T\Sigma^k) \to \dots
\end{align*}
Taking $S_k$-invariants is an exact functor over $\mathbb{C}$, so we have
\begin{align*}
0 \to H^0_{\bar{\partial}}(\WT{S},V^*T\Sigma^k)^{S_k}
\to H^0_{\bar{\partial}}(\WT{S},\pi_{\WT{S}}^*(u^*TX)  )^{S_k}  \to H^0_{\bar{\partial}}(\WT{S},Z)^{S_k} \to H^1_{\bar{\partial}}(\WT{S},V^*T\Sigma^k)^{S_k} 
\end{align*}
By \eqref{eq:regular} and the assumption that $v$ is regular, this reduces to 
\begin{align*}
0 \to H^0_{\bar{\partial}}(\WH{S},v^*T\Sigma)
\to H^0_{\bar{\partial}}(\WT{S},\pi_{\WT{S}}^*(u^*TX)  )^{S_k} \to H^0_{\bar{\partial}}(\WT{S},Z)^{S_k} \to 0
\end{align*}
Since $\pi_{\WT{S}}$ is a branched covering, 
we have
\begin{align*}
H^0_{\bar{\partial}}(\WT{S},\pi_{\WT{S}}^*(u^*TX) )^{S_k}=H^0_{\bar{\partial}}(S, u^*TX).
\end{align*}
Since $\pi_{\WH{S}}$ is simply branched, the complex rank of $H^0_{\bar{\partial}}(\WT{S},Z)^{S_k}$ is precisely the number of critical points, so it has real dimension $2[u] \cdot \Delta$.
Therefore, we have
\begin{align*}
& \dim_{\R} H^0_{\bar{\partial}}(S, u^*TX) \\
= &\dim_{\R} H^0_{\bar{\partial}}(\WH{S},v^*T\Sigma) +\dim_{\R} H^0_{\bar{\partial}}(\WT{S},Z)^{S_k} \\
 = &\vdim(v,\WH{S})+2[u] \cdot \Delta \\
 = &\vdim(u)+3,
\end{align*}
where the last equality comes from Lemma \ref{l:simVdim}.  Given that we have not divided out by the automorphism group of $S$, this exactly says that $u$ is regular. 
\end{proof}

\begin{corol} \label{cor:regular_off_codim2}
Suppose that the non-simple $\tau_i$-fold holomorphic branched coverings from $(B_i,J_{\Sigma}|_{B_i})$ to $S$ form a set of real codimension two among all $\tau_i$-fold branched coverings.
Then $\cM_{[u_i]}(\Sym(\ul{L}),J_X)$ is regular off a set of real codimension $2$ and $\# \cM_{[u_i]}(\Sym(\ul{L}),J_X)=1$.
\end{corol}

\begin{proof}
If $u$ is a holomorphic map which gives rise to an element in $\cM_{[u_i]}(\Sym(\ul{L}),J_X)$ and $(v,\pi_{\hat{S}})$ is tautologically corresponding to $u$, then $[v]=[v_i]$.
By the open mapping theorem, $\im(v) \cap B_j^{\circ}$ is either a point or the entire $B_j^{\circ}$ for each $j$.
Therefore, the Lagrangian boundary condition of $v$ together with    $[v]=[v_i]$ implies that 
there is a connected component $\WH{S}_0$ of $\hat{S}$ such that $v|_{\WH{S}_0}$ is a degree $1$ map to $B_i$.
Moreover, the other connected components of $\WH{S}$ are biholomorphic to $S$ and $v$ restricts to a constant map on these components. Clearly, $v$ is regular.

By Proposition \ref{p:regular}, to show that $\cM_{[u_i]}(\Sym(\ul{L}),J_X)$ is regular off a set of real codimension $2$,
 it suffices to show that among all the $k$-fold branched coverings $\WH{S} \to S$, the ones that are not simply branched  with $[u] \cdot \Delta$ many critical points form a subset of real codimension at least $2$.
The Riemann-Hurwitz formula shows that all $k$-fold branched coverings $\WH{S} \to S$
have $[u] \cdot \Delta$ many critical points when counted with multiplicity.
Therefore, we just need to show that the locus of non-simple branched coverings forms a subset of real codimension at least $2$.
This immediately follows from our assumption on $(B_i,J_{\Sigma}|_{B_i})$ because $\WH{S}=\WH{S}_0 \sqcup  \sqcup_{j=1}^{k-\tau_i} S=B_i \sqcup  \sqcup_{j=1}^{k-\tau_i} S$.

Therefore, $\# \cM_{[u_i]}(\Sym(\ul{L}),J_X)$ is well-defined (see Remark \ref{r:Offcodim2}). Moreover, it 
can be computed using the algebraic count of the tautologically corresponding pair  $(v,\pi_{\WH{S}})$, which can in turn be computed by embedding $B_i$ into $\PP^1$.
In Proposition \ref{p:SymL=Clifford2}, we have already done the computation in $\PP^1$.
The outcome is that each region in the complement of the link contributes to a Maslov two disc.
Therefore, by embedding $B_i$ into $\PP^1$, we can apply Proposition \ref{p:SymL=Clifford2} to conclude that $\# \cM_{[u_i]}(\Sym(\ul{L}),J_X)=1$.
\end{proof}

Note that there exists a complex structure on $B_i$ such that the hypothesis of Corollary \ref{cor:regular_off_codim2} is satisfied.
It is because holomorphic branched coverings of $S$ correspond to a choice of (branch) points in $S$ with monodromy data.
In particular, $2(\tau_i-1)$ branch points in $S$ (counted with multiplicity) together with a monodromy representation into $S_{\tau_i}$ (the symmetric group on $k_i$ elements) determines a complex structure on $B_i$.
Non-simple branched coverings arise when branch points co-incide, which is a codimension two phenomenon. 
Therefore, by a dimension count, the hypothesis of Corollary \ref{cor:regular_off_codim2} holds for the generic complex structure on $B_i$.

We are going to make use of Corollary \ref{cor:regular_off_codim2} to calculate the potential function in the next subsection.

\subsection{Potential in general}\label{s:potentialgeneral}

We return to the case in which $\Sigma$ is a closed surface with arbitrary genus, and is equipped with a symplectic form $\omega$.
In contrast to the previous sections, we do not fix the conformal structure on $\Sigma$ at this point.
Let $\ul{L} \subset \Sigma$ be a $k$-component $\eta$-monotone link whose complement comprises $s$ domains with planar closures $B_i$.  Recall $s=k-g+1$.

\begin{theo}\label{t:curvature}
Let $\Sigma$ and $\ul{L} \subset (\Sigma,\omega)$ be as above.
There is a complex structure $J_{\Sigma}$ on $\Sigma$, for which $\omega$ is a K\"ahler form, and moreover, for the induced complex structure $J_X$ on $X$, the Maslov two $J_X$-holomorphic discs with boundary on $\Sym(\ul{L})$ are regular (off a set of real codimension two) and the disc potential is given by
\begin{align}\label{eq:genPo}
W_{\Sym(\ul{L})}(x,J_X)=\sum_{i=1}^s x^{\partial B_i},
\end{align}
where the term $x^{\partial B_i}$ should be understood via the isomorphism $H_1(\ul{L}) \simeq H_1(\Sym(\ul{L}))$\footnote{
Let $c_i$ be a point in $L_i$ and $\gamma_i(t):[0,1] \to L_i$ be a loop representing the fundamental class of $L_i$.  Then the isomorphism is given by sending $[\gamma_i(t)]$ to $[\Gamma_i(t)]$ where $\Gamma_i(t)=[c_1,\dots,c_{i-1},\gamma_i(t),c_{i+1},\dots,c_k]$ and extending linearly. The isomorphism is independent of the choice of $c_i$ and $\gamma_i$.}.

Furthermore, $x=(1,\dots,1)$ is a critical point of $W_{\Sym(\ul{L})}(x,J_X)$.
\end{theo}
\begin{proof}
We can find a Hamiltonian diffeomorphism $\varphi$ of $(\Sigma,\omega)$ supported near the connected components of $\ul{L}$ such that $\varphi(\ul{L})$ consists of real analytic curves.
Therefore, it suffices to consider the case that $\ul{L}$ consists of real analytic curves.

For each $i=1,\dots,s$, we pick a complex structure on $B_i$ such that the set of non-simple $\tau_i$-fold branched coverings to the unit disc is of real codimension at least two among all $\tau_i$-fold branched coverings.
We can glue the complex structures on $B_i$ for all $i$ because their common boundaries are real analytic.
This gives a complex structure on $\Sigma$.

Choose a K\"ahler form $\omega’$ on $\Sigma$, and let $g’$ be the K\"ahler metric on $\Sigma$ induced by $\omega'$.
For any smooth function $f:\Sigma \to \R_{>0}$, the metric $fg'$ is also a K\"ahler metric on $\Sigma$.
We can pick $f$ such that the K\"ahler form $\omega''$ induced by $fg'$ has area $A_i$ over $B_i$.
By applying the Weinstein neighborhood theorem to $\ul{L}$ in $(\Sigma, \omega)$ and $(\Sigma, \omega'')$, we can find a diffeomorphsim $G: \Sigma \to \Sigma$ such that $G(\ul{L})=\ul{L}$, $G^* \omega''=\omega$ near $\ul{L}$ and $G^*\omega''(B_i)=A_i$ for all $i$. By a Moser argument, $G^* \omega''$ is isotopic to $\omega$ relative to $\ul{L}$.
Therefore, we can find a symplectomorphism $F:(\Sigma, \omega) \to (\Sigma, \omega'')$ such that $F(\ul{L})=\ul{L}$.
The pull-back of the complex structure on $\Sigma$ along $F$ is the $J_{\Sigma}$ we need. 
It has the property that $J_{\Sigma}|_{B_i}$ satisfies the hypothesis of Corollary \ref{cor:regular_off_codim2} for all $i$.

Notice that if $A \in H_2(X,\Sym(\ul{L}))$, $\mu(A)=2$ and $A \neq [u_i]$, then Lemma \ref{l:Maslov} implies that $A=\sum c_i [u_i]$ with some $c_i$ being negative.
Therefore, by positivity of intersection, we have $ \cM_{A}(\Sym(\ul{L}),J_X) = \emptyset$.
On the other hand, by our choice of $J_{\Sigma}$,
we have $\# \cM_{[u_i]}(\Sym(\ul{L}),J_X)=1$ by Corollary \ref{cor:regular_off_codim2}.
All this together give us \eqref{eq:genPo}.

\medskip

For each $i$, there are precisely two terms involving $x_i$, one with exponent $1$ and the other with exponent $-1$. From this, 
it is straightforward to check that $x=(1,\dots,1)$ is a critical point of $W_{\Sym(\ul{L})}(x,J_X)$.
\end{proof}

In the next section, we will only consider the Floer cohomology between $\Sym(\ul{L})$ and its Hamiltonian translates.
We will always assume that $J_{\Sigma}$ is chosen as in Theorem \ref{t:curvature} so that the potential function of $\Sym(\ul{L})$
is given by \eqref{eq:genPo}. 

\begin{remark}\label{r:choiceJ}
A cobordism argument, as in Proposition \ref{p:SymL=Clifford} and \ref{p:SymL=Clifford2}, implies that for generic $J \in \mathcal{J}_{\Delta}$, the potential function of $\Sym(\ul{L})$ (as well as its Hamiltonian translates) will also be given by the RHS of \eqref{eq:genPo}.
\end{remark}

\begin{remark} \label{rmk:Novikov_disc_potential_2} In the more general setting of Definition \ref{rmk:Novikov_disc_potential}, 
when $A_i+2(\tau_i-1)\eta$ is not independent of $i$, the $\eta$-disc potential function is
\begin{align*}
W_{\Sym(\ul{L})}^{\eta}(x,J_X)=\sum_{i=1}^s T^{A_i+2(\tau_i-1)\eta}x^{\partial B_i}.
\end{align*}
\end{remark}

\begin{lemma} \label{lem:non-degenerate}
Suppose that $\Sigma=\PP^1$ and $\ul{L}$ is $\eta$-monotone. 
Then the critical points of $W_{\Sym(\ul{L})}(x,J_X)$ are non-degenerate.
\end{lemma}

\begin{proof}
First we consider the case that $\eta=0$.
When $B_1,\dots,B_k$ are discs, $W_{\Sym(\ul{L})}(x,J_X)$ is given by 
$\sum_{i=1}^k x_i+\frac{1}{x_1 \dots x_k}$
(cf. Proposition \ref{p:SymL=Clifford2} and Theorem \ref{t:curvature}).
One sees that all the critical points of $W_{\Sym(\ul{L})}(x,J_X)$ are non-degenerate. It remains to see how changing the configuration of circles affects the disc potential. 
\medskip

Consider two components $L_1, L_2$ of $\ul{L}$ which are boundary components of a planar domain $B_1$.  There is a `handleslide' move, depending on a choice of arc connecting $L_1$ and $L_2$ (and lying in the complement of other other circles $L_i$),  which replaces  $(L_1, L_2)$ by the pair $(L_1', L_2)$ where $L_1'$ is obtained as the connect sum of $L_1$ and $L_2$ along the arc (see Figure \ref{fig:handleslide}). Let $B_1'$ denote the planar component after the handleslide which contains $L_1'$ but not $L_2$, and $B_2'$ the component containing both $L_1'$ and $L_2$. 
By smoothly isotoping $L_1'$ appropriately (in the complement of $L_i$ for $i \ge 2$), we can assume that the area of $B_1$, $B_1'$ and $B_2'$ are the same.

\begin{figure}[h!]
  \centering
   \def\svgwidth{0.7\textwidth}
\begingroup%
  \makeatletter%
  \providecommand\color[2][]{%
    \errmessage{(Inkscape) Color is used for the text in Inkscape, but the package 'color.sty' is not loaded}%
    \renewcommand\color[2][]{}%
  }%
  \providecommand\transparent[1]{%
    \errmessage{(Inkscape) Transparency is used (non-zero) for the text in Inkscape, but the package 'transparent.sty' is not loaded}%
    \renewcommand\transparent[1]{}%
  }%
  \providecommand\rotatebox[2]{#2}%
  \newcommand*\fsize{\dimexpr\f@size pt\relax}%
  \newcommand*\lineheight[1]{\fontsize{\fsize}{#1\fsize}\selectfont}%
  \ifx\svgwidth\undefined%
    \setlength{\unitlength}{226.25929396bp}%
    \ifx\svgscale\undefined%
      \relax%
    \else%
      \setlength{\unitlength}{\unitlength * \real{\svgscale}}%
    \fi%
  \else%
    \setlength{\unitlength}{\svgwidth}%
  \fi%
  \global\let\svgwidth\undefined%
  \global\let\svgscale\undefined%
  \makeatother%
  \begin{picture}(1,0.27205179)%
    \lineheight{1}%
    \setlength\tabcolsep{0pt}%
    \put(0,0){\includegraphics[width=\unitlength,page=1]{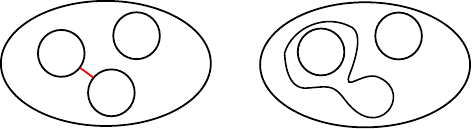}}%
    \put(0.09235058,0.13940177){\color[rgb]{0,0,0}\makebox(0,0)[lt]{\lineheight{1.25}\smash{\begin{tabular}[t]{l}$L_2$\end{tabular}}}}%
    \put(0.20556975,0.04592525){\color[rgb]{0,0,0}\makebox(0,0)[lt]{\lineheight{1.25}\smash{\begin{tabular}[t]{l}$L_1$\end{tabular}}}}%
    \put(0.33019309,0.10263885){\color[rgb]{0,0,0}\makebox(0,0)[lt]{\lineheight{1.25}\smash{\begin{tabular}[t]{l}$B_1$\end{tabular}}}}%
    \put(0.64624978,0.13341099){\color[rgb]{0,0,0}\makebox(0,0)[lt]{\lineheight{1.25}\smash{\begin{tabular}[t]{l}$L_2$\end{tabular}}}}%
    \put(0.75319153,0.05444666){\color[rgb]{0,0,0}\makebox(0,0)[lt]{\lineheight{1.25}\smash{\begin{tabular}[t]{l}$B_2'$\end{tabular}}}}%
    \put(0.8354989,0.06343275){\color[rgb]{0,0,0}\makebox(0,0)[lt]{\lineheight{1.25}\smash{\begin{tabular}[t]{l}$L_1'$\end{tabular}}}}%
    \put(0.88534282,0.12591066){\color[rgb]{0,0,0}\makebox(0,0)[lt]{\lineheight{1.25}\smash{\begin{tabular}[t]{l}$B_1'$\end{tabular}}}}%
  \end{picture}%
\endgroup%

\caption{After sliding $L_1$ across $L_2$ along the red arc, we get $L_1'$.}
\label{fig:handleslide}
\end{figure}
\medskip

Let $\ul{L}'=L_1' \cup \cup_{i=2}^k L_i$.
By the assumption on the area, $\ul{L}'$ is also a $0$-monotone link.
The disc potential functions $W_{\Sym(\ul{L})}$ and $W_{\Sym(\ul{L}')}$ differ in the two terms which previously contained the monomial $x_1$, which is replaced by a monomial $x_1'$ and which arises from the two terms in the potential given by the regions  $B_i'$.
More precisely, for 
$\eps \in \{-1,1\}$ depending on the orientation of $L_2$, 
the modified potential $W_{\Sym(\ul{L}')}$ is obtained from $W_{\Sym(\ul{L})}$ by setting $x_1'=x_1x_2^{\eps}$ and $x_i'=x_i$ for $i \neq1$.
Direct computation shows that such a change of coordinates preserves non-degeneracy of the critical points. Finally, any two planar unlinks can be related by a sequence of such handleslide moves.

For general $\eta$, we can find a smooth family of $\ul{L}_t$ such that $\ul{L}_0$ is $\eta$-monotone and $\ul{L}_1$ is $0$-monotone. There is a cobordism between the Maslov two holomorphic discs that $\Sym(\ul{L}_0)$ and $\Sym(\ul{L}_1)$ bound. We can hence deduce the result for the $\eta>0$ case from the $\eta=0$ case.
\end{proof}

\section{Quantitative Heegaard Floer cohomology}\label{s:QHF}

In this section, we assume that $\ul{L}$ is $\eta$-monotone.
We now introduce the version of Lagrangian Floer cohomology that will underlie our link spectral invariant, which will be introduced in Equation  \eqref{eqn:normalise}.
\subsection{The Floer complex}\label{s:CFcom}

Let $\cE$ be a rank $1$ $\C^*$-local system over $\Sym(\ul{L})$ associated to an element $x \in \Hom(\pi_1(\Sym(\ul{L})), \C^*)=H^1(\Sym(\ul{L}),\C^*)$.
Let $H \in C^{\infty}([0,1] \times \Sigma)$ and $\varphi=\phi^1_H$.
The associated homeomorphisms $\Sym(\phi^t_H)$  of the symmetric product are only Lipschitz along the diagonal $\Delta$, but they are smooth away from $\Delta$ and they induce a well-defined Hamiltonian flow away from $\Delta$. That flow extends as a continuous flow to $\Sym^k(X)$ (induced from the globally smooth $S_k$-equivariant flow on $\Sigma^k$), and in particular the flow exists for all times on the open stratum $\Sym^k(X)\backslash \Delta$.  There is accordingly an  induced rank $1$ local system $\varphi(\cE):=(\Sym(\varphi)^{-1})^*\cE$ on $\Sym(\varphi(\ul{L}))$, with monodromy 
\begin{align}
\varphi(x):=x\circ (\Sym(\varphi)^{-1})_* \in \Hom(\pi_1(\Sym(\varphi(\ul{L}))), \C^*)=H^1(\Sym(\varphi(\ul{L}), \C^*).\label{eq:monodromy}
\end{align}
Suppose that $\Sym(\ul{L}) \pitchfork \Sym(\varphi(\ul{L}))$.
\medskip

Fix a base point $\bfx \in  \Sym(\ul{L})$.
Let $\bfy(t):= \Sym(\phi^{1-t}_H)(\bfx)$,  so $\bfy$ is a path from $\Sym(\varphi(\ul{L}))$ to $\Sym(\ul{L})$.
Let  $\mathcal{P}$ denote the connected component of the space of continuous paths from $\Sym(\varphi(\ul{L}))$ to  $\Sym(\ul{L})$ that contains $\bfy$.
The Floer complex we are going to use will be generated by capped intersection points in the component $\mathcal{P}$.
Suppose that $y \in \Sym(\ul{L}) \cap \Sym(\varphi(\ul{L}))$ lies inside $\mathcal{P}$ as a constant path from $\Sym(\varphi(\ul{L}))$ to $\Sym(\ul{L})$\footnote{If there is no such $y$, the Floer complex and hence cohomology introduced in \eqref{eq:FloercomplexY} will vanish.}.
A {\it capping} of $y$ is  a smooth map $\hat{y}:[0,1] \times [0,1] \to X$ such that $\hat{y}(1,t)=\bfy(t)$, $\hat{y}(0,t)=y$ and $\hat{y}(s,i) \in  \Sym(\phi^{1-i}_H(\ul{L}))$ for $i=0,1$. 

\medskip

For $y_0,y_1 \in \Sym(\ul{L}) \cap \Sym(\varphi(\ul{L}))$ with cappings $\hat{y}_0$ and $\hat{y}_1$ respectively, we say that 
$\hat{y}_0$ and $\hat{y}_1$ are equivalent if $y_0=y_1$ and 
$\omega_X(\hat{y}_0)+\eta [\hat{y}_0] \cdot \Delta=\omega_X(\hat{y}_1)+\eta [\hat{y}_1] \cdot \Delta$.
We denote the set of equivalence classes by $\mathcal{S}$.
Let $u:[0,1]\times [0,1] \to X$ represent an element in  $\pi_2(X, \Sym(\ul{L}))$ such that $u(i,t)=\bfx$ for $i=0,1$.
We can form the concatenation $\hat{y}[u]:=\hat{y} \# (\phi^{1-t}_H(u(s,t)))$ and the equivalence class $\hat{y}[u] \in \mathcal{S}$ is independent of the choice of $u$ representing $[u]$.
Since $\Sym(\ul{L})$ is monotone in the sense of Lemma \ref{l:unobstructed}, we have a free $\Z$ action on $\mathcal{S}$ given by $n\hat{y} \mapsto \hat{y}(n[u_j])$ where $[u_j]$ is one of the basic classes in Corollary \ref{c:discclass}.

 Writing $\cE_y$ for the stalk of the local system at $y$, let 
\begin{align}\label{eq:FloercomplexY} 
CF_{\circ}(\cE;\Sym(H))=\bigoplus_{\hat{y} \in \mathcal{S}} \, \Hom(\varphi(\cE)_y,\cE_y)_{\hat{y}}
\end{align}
where $\Hom(\varphi(\cE)_y,\cE_y)_{\hat{y}}=\Hom_{\C}(\varphi(\cE)_y, \cE_y)$. In other words, there is one copy of $\Hom_{\C}(\varphi(\cE)_y, \cE_y)$ in $CF_{\circ}(\cE;\Sym(H))$ for each equivalence class of capping of $y$.
We denote an element $f \in \Hom(\varphi(\cE)_y, \cE_y)_{\hat{y}}$ by $(f,\hat{y})$.
The $\Z$ action on $\mathcal{S}$ induces a free $\C[T,T^{-1}]$-module structure on $CF_{\circ}(\cE;\Sym(H))$.
Explicitly, for $(f,\hat{y}) \in CF_{\circ}(\cE;\Sym(H))$, $T^n(f,\hat{y})$ is the element $ (f,n\hat{y}) \in CF_{\circ}(\cE;\Sym(H))$ where we are using the identification
$\Hom(\varphi(\cE)_y,\cE_y)_{n\hat{y}} \simeq \Hom(\varphi(\cE)_y,\cE_y) \simeq \Hom(\varphi(\cE)_y,\cE_y)_{\hat{y}}$.
Let  
\begin{equation} \label{eqn:coefficient_ring} 
R=\C[[T]][T^{-1}]=\left \{ \sum_{i=0}^{\infty} a_iT^{b_i} | a_i \in \C, b_i \in \Z, b_0 <b_1< \dots\right \}
\end{equation} 
and define
\begin{equation}\label{eqn:first_Floer_complex}
CF(\cE,\Sym(H))=CF_{\circ}(\cE;\Sym(H)) \otimes_{\C[T,T^{-1}]} R
\end{equation}
which is a free $R$-vector space whose rank is the number of intersection points in $\Sym(\ul{L}) \cap \Sym(\varphi(\ul{L}))$ that lie in $\mathcal{P}$.\footnote{
The differential (see \eqref{eq:diffeta}) of $CF(\cE,\Sym(H))$ will only involve finitely many terms due to monotonicity (Lemma \ref{l:unobstructed}) so the Floer complex $CF_{\circ}(\cE;\Sym(H))$ is well-defined. However, we would like to work over a field instead in order to be in line with the literature on spectral invariants.}

\begin{remark}\label{r:grading}
Since we only consider Floer cohomology for a Lagrangian and its Hamiltonian translate, the usual relative grading in Floer cohomology gives a well-defined absolute $\mathbb{Z}/N$-grading, for $N$ the minimal Maslov index (in our case $N=2$).
Although not needed in this section, we can give a well-defined $\mathbb{Z}$-grading on $CF(\cE,\Sym(H))$
by grading the Novikov variable $T$ with $\deg(T)=2$.
\end{remark}

\begin{definition}\label{d:spectrum}
The action of $(f,\hat{y})$  with respect to $\Sym(H)$  is 
\begin{align}
\cA^{\eta}_{H}(f,\hat{y}):=\int_{t=0}^1\Sym(H_t)(\bfx)dt- \int \hat{y}^*\omega_X -\eta [\hat{y}] \cdot \Delta. \label{eq:action}
\end{align}
The action spectrum of $\Sym(H)$ is $\Spec(\Sym(H):\ul{L}):=\{\cA^{\eta}_{H}(f,\hat{y})\, \big | \, (f,\hat{y}) \in CF(\cE,\Sym(H))\}$.

\medskip

We also define $\Spec(H:\ul{L}):=\frac{1}{k}\Spec(\Sym(H):\ul{L})$ which will be the spectrum where the spectral invariant $c_{\ul{L}}$ in Theorem \ref{t:spectral} lies.
\end{definition}

\begin{remark}
We have $\cA^{\eta}_{H}((f,\hat{y})T)=\cA^{\eta}_{H}(f,\hat{y}[u_j])=\cA^{\eta}_{H}(f,\hat{y})-\omega_X(u_j) -\eta [u_j] \cdot \Delta = \cA^{\eta}_{H}(f,\hat{y}) - \lambda$, where $\lambda$ is the monotonicity constant of the link $\ul{L}$ as defined in Definition \ref{def:monotone_link}. 
\end{remark}

\begin{remark}\label{rmk:constant_in_action}
The integral $\int_{t=0}^1\Sym(H_t)(\bfx)dt$ is a constant which is independent of $y$ and $\hat{y}$.
Adding this constant is only for making the formula of action in \eqref{eq:action2} more natural.
Note that $\int \hat{y}^*\omega_X$ is well-defined even though $\omega_X$ is singular along $\Delta$, cf.  Definition \ref{d:topologicalenergy}. 
\end{remark}

\begin{remark}\label{rmk:spectrum_measure_zero}
The action spectrum  $\Spec(H:\ul{L})$ is a closed and nowhere dense subset of $\R$; this can be proven by adapting the arguments from \cite{Oh05}.
\end{remark}

Since $\Sym(\phi^t_H(\ul{L}))$ is disjoint from all of $\Delta$ for all $t$, we can choose an open neighborhood $V$ of $\Delta$
that is disjoint from $\Sym(\phi^t_H(\ul{L}))$ for all $t$.

Let $\{J_t\}_{t \in [0,1]}$ be a path of almost complex structures such that $J_t \in \cJ_{\Delta}(V)$ for all $t$.
Let $\cM(y_0;y_1; \{J_t\}_{t \in [0,1]})$ be the moduli space of smooth maps $u:\R \times [0,1] \to X$ such that
\begin{equation} \label{eqn:usual_CR}
\left\{
\begin{array}{ll}
      u(s,0) \in \Sym(\varphi(\ul{L})), u(s,1) \in \Sym(\ul{L}) \\
      \lim_{s \to  -\infty}u(s,t)=y_0,  \lim_{s \to  \infty}u(s,t)=y_1\\
      \partial_s u + J_t \partial_t u=0
\end{array}
      \right\} 
      \end{equation}
modulo the $\R$-action by translation in the $s$-coordinate.  

Let $\cM(y_0;y_1; \{J_t\})^{\circ}$ be the set of virtual dimension $0$ solutions (modulo translation) in $\cM(y_0;y_1; \{J_t\}_{t \in [0,1]})$.
For generic $\{J_t\}_{t \in [0,1]}$, every solution $u \in \cM(y_0;y_1; \{J_t\})^{\circ}$ is regular (see e.g. \cite[Proposition 15.1.5]{OhBook2}). Let   $\omega_X(u)$ be defined as in Definition \ref{d:topologicalenergy}.

\medskip

By the monotonicity Lemma \ref{l:unobstructed}, there is a uniform upper bound for the energy of Maslov index $1$ solutions $u$ with given asymptotics.
Therefore, we can apply  Gromov compactness to constrain the structure of  the zero dimensional subset  $\cM(y_0;y_1; \{J_t\})^{\circ}$.
For every non-constant irreducible component $u'$ of a pseudo-holomorphic stable strip arising from the Gromov compactification, we have $\omega_X(u')+\eta [u'] \cdot \Delta >0$.
There are two kinds of possible non-constant irreducible components in a stable strip.
The first kind is a pseudo-holomorphic disc with boundary on either $\Sym(\varphi(\ul{L}))$ or $\Sym(\ul{L})$, which necessarily has postive Maslov index by monotonicity (Lemma \ref{l:unobstructed}).
The other kind is a pseudo-holomorphic strip, which by regularity, can only exist if the Maslov index is at least one.
Therefore, every irreducible  component of a pseudo-holomorphic stable strip has positive Maslov index.
Since elements in $\cM(y_0;y_1; \{J_t\})^{\circ}$ have Maslov index one, any pseudo-holomorphic stable strip can only have one irreducible component.
It means that the Gromov compactification of  $\cM(y_0;y_1; \{J_t\})^{\circ}$ is the space itself, which is therefore a finite set.
\medskip

For each $u \in \cM(y_0;y_1; \{J_t\}_{t \in [0,1]})$ and $f \in \Hom(\varphi(\cE)_{y_1}, \cE_{y_1})_{\hat{y}_1}$, we define
\begin{align*}
u(f):= P_{\cE, u(s,1)}^{-1} \circ f \circ  P_{\varphi(\cE), u(s,0)} \in (\Hom(\varphi(\cE)_{y_0}, \cE_{y_0}))_{u\#\hat{y}_1}
\end{align*}
where $P$ denotes the appropriate parallel transport map.
\medskip

For $f \in \Hom(\varphi(\cE)_{y_1}, \cE_{y_1})_{\hat{y}_1}$, the differential on $CF(\cE,\Sym(H))$ is defined by 
\begin{align}
m_1(f,\hat{y}_1)= \sum_{y_0} \sum_{u \in \cM(y_0;y_1; \{J_t\})^{\circ}} (-1)^{\eps(u)} 
(u(f),u \# \hat{y}_1) \label{eq:diffeta}
\end{align}
and extending $R$-linearly, where $\eps(u) \in \{0,1\}$ is the orientation sign of $u$.

\begin{lemma}\label{l:d2=0}
$(m_1)^2=0$.
\end{lemma}

\begin{proof}
By construction, the Hamiltonian isotopy $\Sym(\phi^t_H)$ maps $(\Sym(\ul{L}),\cE)$ to $(\Sym(\varphi(\ul{L})), \varphi(\cE))$, compatibly with the  orientations and spin structures on the Lagrangians. 
As explained in Remark \ref{r:choiceJ}, we chose $J_{\Sigma}$
such that Theorem \ref{t:curvature} applies.
In this case, $W_{\Sym(\ul{L})}(-,J)$ and $W_{\Sym(\varphi(\ul{L}))}(-,J)$ are given by \eqref{eq:genPo} for generic $J \in \cJ_{\Delta}$.
Therefore, we have $W_{\Sym(\ul{L})}(x,J)=W_{\Sym(\varphi(\ul{L}))}(\varphi(x),J)$ (cf. \ref{eq:monodromy}).

The boundary of the Gromov compactification of the $1$ dimensional component of the moduli space $\cM(y_0;y_1; \{J_t\})$ has two strata, arising from stable maps which comprise a constant strip 
glued to a Maslov two disc bubble, which can form on either boundary $\Sym(\ul{L})$ or $\Sym(\varphi(\ul{L}))$. These configurations are counted algebraically by the terms $W_{\Sym(\ul{L})}(x,J_1) (f,\hat{y}_1)$ and $W_{\Sym(\varphi(\ul{L}))}(\varphi(x),J_0) (f,\hat{y}_1)$, respectively\footnote{In the language of curved $A_{\infty}$ algebra, the terms $W_{\Sym(\ul{L})}(x,J_1) (f,\hat{y}_1)$ and $W_{\Sym(\varphi(\ul{L}))}(\varphi(x),J_0) (f,\hat{y}_1)$ are the curvature of the Fukaya algebra of $(\Sym(\ul{L}),\cE)$ and $(\Sym(\varphi(\ul{L})), \varphi(\cE))$ respectively.}. 
\medskip

Taking account of the (standard) orientation signs, we therefore have
\begin{align*}
(m_1)^2(f, \hat{y}_1)=(W_{\Sym(\ul{L})}(x,J_1)-W_{\Sym(\varphi(\ul{L}))}(\varphi(x),J_0)) (f,\hat{y}_1)=0
\end{align*}
as required.
\end{proof}

A routine argument shows that the homology of $(CF(\cE,\Sym(H)), m_1)$, which we denote by $HF(\cE,\Sym(H))$, is independent of the choice of generic $(J_t)_{t \in [0,1]}$ with $J_t \in \mathcal{J}_{\Delta}(V) $.

\begin{remark}\label{rmk:floer_is_usual_floer}(Comparison with the standard monotone Floer theory)
Given an open neighbourhood $V\supset \Delta$ as in the paragraph after Remark \ref{rmk:spectrum_measure_zero}, one can pick a smooth K\"ahler form $\omega_V$ on $\Sym^k(\Sigma)$ as in Definition \ref{d:topologicalenergy} making $\Sym(\phi^t_H(\ul{L}))$ Lagrangian for all $t \in [0,1]$. If $\Sigma=\PP^1$ (or $\eta=0$), one can then inflate this along the diagonal (or do nothing) to obtain a symplectic form $\omega_{V,\eta}$  making $\Sym(\phi^t_H(\ul{L}))$ monotone Lagrangian submanifolds for all $t \in [0,1]$, cf. Lemma \ref{rmk:monotone_Lagrangian}. 
Let $CF(\cE, \Sym(H), \omega_{V,\eta})=CF(\cE,\Sym(H))$ as $R$-vector spaces and equip the former one with the usual Floer differential
defined as in \eqref{eqn:usual_CR}.
If we define the action of elements in $CF(\cE, \Sym(H), \omega_{V,\eta})$ by 
\[
\cA^{\eta}_{H,\omega_{V,\eta}}(f,\hat{y}):=\int_{t=0}^1\Sym(H_t)(\bfx)dt- \int \hat{y}^*\omega_{V,\eta}
\]
then there is an equality $\cA^{\eta}_{H,\omega_{V,\eta}}(f,\hat{y}) = \cA^{\eta}_{H}(f,\hat{y})$.
Therefore, if the Floer differentials of $CF(\cE, \Sym(H), \omega_{V,\eta})$ and $CF(\cE,\Sym(H))$ agree, then we conclude that there is an action preserving chain isomorphism between them.
This is the case if $J_t$ is $\omega_{V,\eta}$-tamed for all $t$.

If $J_t$ is not $\omega_{V,\eta}$-tamed, we can still get an action preserving quasi-isomorphism between the two by a routine homotopy argument, which we sketch here.
Without loss of generality, we can assume that the inflation is realised by a smooth family of symplectic forms $\omega_{V,e}$, for $e \in [0,\eta]$.
We can pick a smooth family $J_{t,e}$ such that $J_{t,e}$ equals $J_X$ near $\Delta$ and is $\omega_{V,e}$-tamed for all $t \in [0,1]$ and $e \in [0,\eta]$. Moreover, we assume $J_{t,0}=J_t \in \cJ_\Delta(V)$.
For every $e \in [0,\eta]$, there is an open subset $I \subset [0,\eta]$ containing $e$ such that $J_{t,e'}$ is $\omega_{V,e}$-tamed for all $e' \in I$.
This homotopy of  almost complex structures parametrized by $I$ gives us an action-preserving chain map
$CF(\cE,\Sym(H), \omega_{V,e'}) \to CF(\cE, \Sym(H), \omega_{V,e})$ for every $e' \in I$.
With respect to the action filtration, this chain map is an upper triangular matrix with $1$'s on the diagonal, so it is a quasi-isomorphism.
Since $[0,\eta]$ is compact, we obtain an action-preserving quasi-isomorphim $CF(\cE,\Sym(H)) \to CF(\cE, \Sym(H), \omega_{V,\eta})$ by composing finitely many action-preserving quasi-isomorphims.
\end{remark}

\subsection{A direct system and Hamiltonian invariance\label{Sec:direct_system}}

We have set up the Floer complex and its action filtration using the unperturbed Cauchy-Riemann equation, to avoid discussing the vector field $X_{\Sym(H)}$, since the Hamiltonian $\Sym(H)$ is only Lipschitz continuous and the corresponding $C^0$-flow is only stratum-wise smooth (relative to the stratification by partition type)  along $\Delta$.  For simplicity, we are going to modify $\Sym(H)$ near $\Delta$ to rewrite the action filtration in more familiar terms, cf. \eqref{eq:action2}, whilst working only with smooth functions and flows. 
\medskip

Since $\Sym(\phi^t_H)$ preserves the diagonal $\Delta$, the moving Lagrangian $\Sym(\phi^t_H)(\Sym(\ul{L}))$ is disjoint from $\Delta$ for all $t$.
We say a Hamiltonian $K \in C^{\infty}([0,1] \times X)$ \emph{compatible with $H$} if there is an open neighborhood $V \supset \Delta$ that is disjoint from $\cup_{t \in [0,1]} \Sym(\phi^t_H)(\Sym(\ul{L}))$ such that 
\begin{align}
&\text{$K=\Sym(H)$ outside $V$;}\label{eq:KCond1}\\
&\text{$K_t$ is a ($t$-dependent) constant near $\Delta$}.\label{eq:KCond2}
\end{align}

\begin{remark}\label{r:flexK}
It is possible to construct $K$ as above such that it furthermore satisfies $\min_{X } \Sym(H_t) \le K_t \le \max_{X } \Sym(H_t) $ for all $t$.
To do this, let $\chi:X \to [0,1]$  be a  cut-off function which equals $1$ outside $V$ and equals $0$ near $\Delta$.
Then we can define $K_t=(\Sym(H_t)-k\int_\Sigma H_t \omega) \chi +k\int_\Sigma H_t \omega$.
It satisfies $\min_{X } \Sym(H_t) \le K_t \le \max_{X } \Sym(H_t) $
because $k\int_\Sigma H_t \omega \in [\min_X \Sym(H_t), \max_X \Sym(H_t)]$.

The flexibility of having $K$ equal to a constant near $\Delta$ which is not necessarily $0$ is used in the proof of Lemma \ref{lem:hofer_bound}.

\end{remark}

Let $\phi_K^t$ be the time-$t$ Hamiltonian diffeomorphism of $K_t$, which is well-defined because $K_t$ is a constant near $\Delta$.
Note that $\phi_K^t(\Sym(\ul{L}))=\Sym(\phi^t_H)(\Sym(\ul{L}))$ for all $t \in [0,1]$ so in particular, $K$ is non-degenerate because we have assumed $\Sym(\ul{L}) \pitchfork \Sym(\varphi(\ul{L}))$.
\medskip

There is a canonical way to define a filtered complex $CF(\cE,X_K)$ isomorphic to $CF(\cE,\Sym(H))$, 
but in which the differential is given by counting solutions to an $X_K$-perturbed equation instead of  the unperturbed $J$-holomorphic curve equation.
We recall the construction of $CF(\cE,X_K)$.
For each $(y,\hat{y})$ as above, we let  $x(t)=(\phi^{t-1}_K)(y)$ and $\hat{x}(s,t):=(\phi^{t-1}_K)(\hat{y}(s,t))$.
Using the bijective correspondence between $(y,\hat{y})$ and $(x,\hat{x})$, we can use $\mathcal{S}$
to denote the equivalence classes of $\hat{x}$ which are defined analogous to that of $\hat{y}$.
Since $\phi^t_K$ is supported away from $\Delta$, $\hat{x}(s,t)$ is a smooth map.
We set 
\begin{equation} \label{eqn:second_Floer_complex}
CF_{\circ}(\cE,X_K):=\oplus_{\hat{x} \in \mathcal{S}} (\Hom(\cE_x, \cE_x))_{\hat{x}}. 
\end{equation}
It carries a free $\C[T,T^{-1}]$-module structure,  like its counterpart $CF(\cE,\Sym(H))$.
We define 
\begin{align}
CF(\cE,X_K):=CF_{\circ}(\cE,X_K)\otimes_{\C[T,T^{-1}]} R. \label{eq:EHchi}
\end{align}
By abuse of notation, we denote by $(\phi^{t-1}_K)_*$ the isomorphism 
$(\Hom(\cE_x, \cE_x))_{\hat{x}} \to (\Hom(\varphi(\cE)_y, \cE_y))_{\hat{y}}$
induced by $\phi^{t-1}_K$.
It gives an isomorphism of $R$-vector spaces
$CF(\cE,X_K) \to CF(\cE,\Sym(H))$. 
The differential for the complex $CF(\cE,X_K)$ is given by counting rigid curves satisfying
\[   \left\{
\begin{array}{ll}
      u^K(s,0) \in \Sym(\ul{L}), u^K(s,1) \in \Sym(\ul{L}) \\
      \lim_{s \to  -\infty}u^K(s,t)=x_0(t):=(\phi_K^{t-1})(y_0),  \lim_{s \to  \infty}u^K(s,t)=x_1(t):=(\phi_K^{t-1})(y_1)\\
      \partial_s u^K + J_t^K (\partial_t u^K- X_{K}(u^K))=0
\end{array}
    \right\}. \]

These are in bijection with elements in $\cM(y_0;y_1; \{J_t\}_{t \in [0,1]})$ via 
\begin{equation} \label{eqn:bijection}
u^K(s,t):=(\phi_K^{t-1})(u(s,t)) \text{ where } J^K_t= J_t \circ (\phi_K^{1-t})_*.
\end{equation}
We have a more familiar formula for the action of elements in $CF(\cE,X_K)$.
Let $f \in (\Hom(\cE_x, \cE_x))_{\hat{x}}$.
\begin{align}
\cA^{\eta}_{K}(f,\hat{x}):=&\int_{t=0}^1 \Sym(H_t)(x(t)) dt - \int \hat{x}^*\omega_X -\eta [\hat{x}] \cdot \Delta  \label{eq:action2}\\ 
=&\cA^{\eta}_{H}((\phi^{t-1}_K)_* \circ f,\hat{y}) 
\end{align}

because\footnote{Recall that our convention is  $ \omega(X_{K_t},\cdot) = d K_t.$} 
\begin{align*}
\int \hat{y}^*\omega_X&=\int_{t=0}^1 \int_{s=0}^{1} \omega_X(\partial_s \hat{y}, \partial_t \hat{y}) ds dt \\
&=\int_{t=0}^1 \int_{s=0}^{1}\omega_X(\partial_s \hat{x}, \partial_t \hat{x}-X_{K_t}(\hat{x}(s,t))) ds dt\\
&=\int \hat{x}^*\omega_X+\int_{t=0}^1 \int_{s=0}^{1} \frac{\partial K_t(\hat{x}(s,t))}{\partial s} ds dt \\
&=\int \hat{x}^*\omega_X-\int_{t=0}^1 K_t(x(t))dt+ \int_{t=0}^1K_t(\bfx)dt\\
&=\int \hat{x}^*\omega_X-\int_{t=0}^1 \Sym(H_t)(x(t))dt+ \int_{t=0}^1\Sym(H_t)(\bfx)dt.
\end{align*}

Notice that even though $\hat{x}$ depends on the choice of $K$, $\int \hat{x}^*\omega_X$ is a topological quantity that is independent of the choice of $K$, provided that $K$ is compatible with $H$.
This identification gives an action preserving chain isomorphism 
\begin{align}
CF(\cE,\Sym(H)) \simeq CF(\cE,X_K) \label{eq:Hchi}
\end{align}
for any $K$ compatible with $H$.

The benefit of working with $CF(\cE,X_K)$, rather than $CF(\cE,\Sym(H))$, is that for the globally smooth Hamiltonian function $K$ the standard proof applies to show that the PSS map (induced by $K$)
\begin{align} \label{eqn:PSS}
\Phi_K: CF(\cE,\cE) \to CF(\cE,X_K)
\end{align}
 is a quasi-isomorphism, where $CF(\cE,\cE)$ is a Morse cochain complex underlying the  pearl model for the Floer  cohomology of $\cE \to \Sym(\ul{L})$.
In other words, $CF(\cE,\cE)$ is generated by critical points of a Morse function and the differential counts rigid pearly gradient trajectory with two asymptotes going to critical points \cite{BC09, BC12,Zapolsky}. 
 On the other hand, given non-degenerate $H^i=(H^i_t)_{t \in [0,1]} \in C^{\infty}([0,1] \times \Sigma)$  and $K^i$ compatible with $H^i$, we also have the continuation map (induced by a regular homotopy $K^s$ from
$K^0$ to $K^1$
such that $K^s_t$ equals to a $(s,t)$-dependent constant near $\Delta$ for all $(s,t)$)
\begin{align}
\Phi_{K^0,K^1}: CF(\cE,X_{K^1}) \to CF(\cE,X_{K^0}). \label{eq:ContinK}
\end{align}
These continuation maps satisfy $\Phi_{K^0} = \Phi_{K^1,K^0} \circ \Phi_{K^1}$ and
$\Phi_{K^2,K^0} = \Phi_{K^1,K^0} \circ \Phi_{K^2,K^1}$ (and $\Phi_{K^0, K^0} = \id$). The upshot is that 
 we have a direct system of filtered chain complexes indexed by pairs $(H,K)$, where $H \in C^{\infty}([0,1] \times \Sigma)$ is non-degenerate and $K$ is compatible with $H$.
\medskip

\begin{lemma}\label{lem:Ham_invariance}
If $H^i \in C^{\infty}([0,1] \times \Sigma)$ is non-degenerate, for $i=0,1$, there is an isomorphism $HF(\cE,\Sym(H^0)) \simeq HF(\cE, \Sym(H^1))$.
\end{lemma}

\begin{proof}
Pick $K^i$  compatible with $H^i$ for $i=0,1$. Invertibility of the continuation map $\Phi_{K^0,K^1}$ from the direct system gives a chain of isomorphisms
\begin{align*}
HF(\cE,\Sym(H^1)) \simeq HF(\cE,X_{K^1}) \simeq HF(\cE,X_{K^0}) \simeq HF(\cE,\Sym(H^0)).
\end{align*}
The result follows.
\end{proof}

\begin{remark}\label{r:diffComplex}(Comparison of different cochain complexes)
We have introduced several related cochain complexes.
The first one $CF(\cE,\Sym(H))$ is introduced in \eqref{eqn:first_Floer_complex} \eqref{eq:diffeta}.
It is action-preserving isomorphic \eqref{eq:Hchi} to $CF(\cE,X_K)$, which is introduced in 
\eqref{eq:EHchi}, under the assumption that $K$ is compatible with $H$ (see \eqref{eq:KCond1} \eqref{eq:KCond2}). 
We have also introduced a pearl model $CF(\cE,\cE)$ and a PSS quasi-isomorphism to $CF(\cE,X_K)$ in \eqref{eqn:PSS}.

Later on in Section \ref{sec:link=monotone}, we will use the relation between these cochain complexes and the standard monotone Lagrangian Floer complex of $\Sym(\ul{L})$ inside $(\Sym^k(\Sigma),\omega_{V,\eta})$ when either $g=0$ or $\eta=0$.
Therefore, $CF(\cE,\Sym(H), \omega_{V,\eta})$ is introduced in Remark \ref{rmk:floer_is_usual_floer} and its relation to $CF(\cE,\Sym(H))$ is explained.
We want to introduce one more chain complex denoted by $CF(\cE,X_K, \omega_{V,\eta})$.
It is the standard monotone Lagrangian Floer complex of $\Sym(\ul{L})$ with respect to a non-degenerate smooth Hamiltonian function $K$ on  $(\Sym^k(\Sigma),\omega_{V,\eta})$ so that the generators are Hamiltonian chords of $K$ from $\Sym(\ul{L})$ to itself and the differential counts rigid Floer solutions with asymptotes going to Hamiltonian chords.

In summary, we have two complexes which do not rely on a choice of $\omega_V$:
\begin{itemize}
\item  $CF(\cE,\Sym(H))$ generated by intersection points $\Sym(\ul{L}) \pitchfork  \Sym(\phi_H^1(\ul{L}))$ with differential counting rigid pseudo-holomorphic bigons with respect to some $J_t \in \cJ_{\Delta}(V)$,
\item $CF(\cE,X_K)$ generated by Hamiltonian chords from $\Sym(\ul{L})$ to itself with differential counting Floer solutions, need $K$ compatible with $H$, 
\end{itemize}
(and a pearl model $CF(\cE,\cE)$ for the second of these, with its PSS quasi-isomorphism), and two complexes defined only when either $g=0$ or $\eta=0$ and that do rely on a choice of $\omega_{V,\eta}$:
\begin{itemize}
\item $CF(\cE,\Sym(H), \omega_{V,\eta})$ generated by intersection points $\Sym(\ul{L}) \pitchfork  \Sym(\phi_H^1(\ul{L}))$ with differential counting rigid pseudo-holomorphic bigons with respect to some $\omega_{V,\eta}$-tamed almost complex structures.
\item $CF(\cE,X_K, \omega_{V,\eta})$ generated by Hamiltonian chords from $\Sym(\ul{L})$ to itself with differential counting Floer solutions, no need for $K$ to be compatible with $H$, since this is standard monotone Floer theory.
\end{itemize}

Even though their definitions are different, it is not hard to see that they are all action-preserving quasi-isomorphic when all of them are well-defined (i.e. when either $g=0$ or $\eta=0$, and for $K$ compatible with $H$).
This will be used in Lemma \ref{lem:spectral_stabilises}.

\end{remark}

\subsection{The disc potential revisited}

A standard criterion for non-vanishing of Floer cohomology for a Lagrangian torus is the existence of a critical point of an appropriate potential (usually, a potential defined from the curved $A_{\infty}$-structure on a space of weak bounding cochains, or a disc potential in the sense introduced previously, cf. Remark \ref{rmk:Floer_survey}). See \cite{MS19} for a rapid overview and references, \cite[Section 5.3]{Seidel:flux} for a `monotone' version closely related to that used here, and 
\cite[Proposition 4.34]{Charest-Woodward} for a detailed treatment in a general formalism (which would also apply over the Novikov field in the setting of Definition \ref{rmk:Novikov_disc_potential}).  \medskip

Recall the disc potential is a map
\[
W_{\Sym(\ul{L})}(-,J): H^1(\Sym(\ul{L}), \C^*) \to \C.
\] 
As explained in Remark \ref{r:choiceJ}, we chose $J_{\Sigma}$
such that Theorem \ref{t:curvature} applies.
In this case, $W_{\Sym(\ul{L})}(-,J)$ is given by \eqref{eq:genPo} for generic $J \in \cJ_{\Delta}$.

\begin{lemma} \label{l:critialW}
Suppose $x \in H^1(\Sym(\ul{L}),  \C^*) \cong (\C^*)^k$ is a critical point of $W_{\Sym(\ul{L})}$.  Let $\cE$ denote the corresponding local system. Then for any non-degenerate $H \in C^{\infty}([0,1] \times \Sigma)$, the Floer cohomology $HF(\cE, \Sym(H)) \simeq HF(\cE, \cE) \neq 0$ and is isomorphic to $H^*(T^k;R)$ as an $R$-vector space.
\end{lemma}

\begin{proof}
As in $CF(\cE,\cE)$ above, we use a pearl model to compute $m_1$; the equivalence of the pearl model and the Hamiltonian model of Floer cohomology (for monotone Lagrangians, but non-existence of discs of non-positive index suffices)  is addressed in \cite[Section 5.6]{BiranCornea_quantum} and \cite{Zapolsky}.  
Our set-up differs from the usual one only because we use Lemma \ref{l:unobstructed} to obtain the well-definedness of $W_{\Sym(\ul{L})}$;
this has no effect on the proof of the equivalence of the two models. 
To bring the comparison of Hamiltonian model of Floer cohomology and the pearl model into the framework considered in \cite{BiranCornea_quantum}, one can use a globally smooth function $K$ compatible to $H$ to replace $\Sym(H)$  as in Section \ref{Sec:direct_system}.
\medskip

Given that, the same statement and proof as in  \cite[Proposition 4.33]{Charest-Woodward} applies in our case.  The result then follows from \cite[Proposition 4.34]{Charest-Woodward},  because $\Sym(\ul{L})$ is a Lagrangian torus so its classical cohomology is generated by degree one classes.
\end{proof}

When $\Sigma = \PP^1$, Lemma \ref{lem:non-degenerate} shows that the potential function has non-degenerate critical points. Therefore, the Floer multiplicative structure on $HF(\cE, \cE)$ can be derived as in \cite{Cho-Oh}.

\begin{remark} In the situation of Definition \ref{rmk:Novikov_disc_potential} one can define Floer cohomology over $\Lambda$.  If the potential function from Remark  \ref{rmk:Novikov_disc_potential_2} has critical points in $(U_{\Lambda})^k$, Floer cohomology of the corresponding rank one unitary local system is non-zero over $\Lambda$. 
\end{remark}

\subsection{Proof of Theorem  \ref{t:spectral}}\label{sec:proof-spectral}
In this section, we define our spectral invariant $c_{\ul{L}}$; see Equation \eqref{eqn:normalise}.  The properties of $c_{\ul{L}}$, as stated in Theorem \ref{t:spectral},  can be proven in a manner very similar to the case of classical monotone Lagrangian spectral invariant.  Hence, as an illustration, we only prove the Hofer-Lipschitz continuity, the spectrality and the homotopy invariance properties.  Moreover, as stated earlier, when $g=0$ or $\eta =0$, our spectral invariant coincides with the invariant from the classical monotone Lagrangian Floer theory, see Lemma \ref{lem:spectral_stabilises}; this immediately implies Theorem \ref{t:spectral} in the case where $g=0$ or $\eta =0$.

For $a\in\R$, let $CF(\cE,\Sym(H))^{<a}$ be the $\C$-subspace of $CF(\cE,\Sym(H))$ generated by those $(f,\hat{y})$ for which the action $\cA^{\eta}_{H}(f,\hat{y})$ is less than $a$.

\begin{lemma}\label{l:actionD}
The differential on $CF(\cE,\Sym(H))$ preserves the $\C$-subspace $CF(\cE,\Sym(H))^{<a}$.
\end{lemma}

\begin{proof}
If $u$ contributes to the $(f_0,\hat y_0)$-coefficient of $m_1(f_1,\hat{y}_1)$, then we have
$\cA^\eta_H(f_1,\hat{y}_1)- \cA^\eta_H(f_0,\hat{y}_0)=\omega_X(u) + \eta [u] \cdot \Delta$ which is positive.
\end{proof}

We define $CF(\cE,X_K)^{<a} \subset CF(\cE,X_K)$ analogously and Lemma \ref{l:actionD} holds by replacing 
$CF(\cE,\Sym(H))$ and  $CF(\cE,\Sym(H))^{<a}$ with $CF(\cE,X_K)$ and $CF(\cE,X_K)^{<a}$.

\begin{definition} Suppose  $\cE$ is a local system corresponding to a critical point of the disc potential. Let  $0 \neq e_{\cE} \in HF(\cE,\cE)$ be the unit. Fix a Hamiltonian $H$ for which $\Sym(\ul{L}) \pitchfork \Sym(\varphi(\ul{L}))$ and
a Hamiltonian $K$ compatible with $H$. Define
\begin{align*}
c_{\cE}(K):=\inf  \left\{a \in \R\, | \,  \Phi_K(e_{\cE}) \in \operatorname{im}\left(HF(\cE; X_K)^{<a} \to HF(\cE; X_K)\right)\right\}.
\end{align*}
Noting that $CF(\cE; X_K)$ is canonically action-preserving isomorphic to $CF(\cE,\Sym(H))$, we then define
\begin{align*}
c_{\cE}(H):=c_{\cE}(K)
\end{align*}
which is independent of $K$.
\end{definition}

Recall that whenever the assumption of Theorem \ref{t:spectral} is satisfied, we know that the disc potential has a critical point at the trivial local system $\cE$, corresponding in our earlier co-ordinates to $x = (1,\ldots, 1)$.
When $\cE$ is the trivial local system on $\Sym(\ul{L})$, we  denote $$c_{\Sym(\ul{L})} := c_\cE.$$

\begin{lemma}\label{lem:hofer_bound}
Let $H^i=(H_t^i)_{t \in[0,1]} \in C^{\infty}([0,1] \times \Sigma)$ for $i=0,1$ be such that $ \Sym(\varphi^i(\ul{L})) \pitchfork \Sym({\ul{L}})$ for both $i=0,1$.
Then 
\begin{align}
\int_0^1 \min_X(\Sym(H^0_t)-\Sym(H^1_t)) dt \le c_{\cE}(H^0)-c_{\cE}(H^1) \le \int_0^1 \max_X(\Sym(H^0_t)-\Sym(H^1_t)) dt. \label{eq:ContinuityX}
\end{align}
\end{lemma}

\begin{proof}

We are going to prove $ c_{\cE}(H^0)-c_{\cE}(H^1) \le \int_0^1 \max_X(\Sym(H^0_t)-\Sym(H^1_t)) dt$.
By interchanging the role of $H^0$ and $H^1$, we also have the other inequality.

Let $K^i$ be compatible with $H^i$.
It suffices to show that 
\begin{align*}
c_{\cE}(K^0)-c_{\cE}(K^1) \le \int_0^1 \max_X(K^0_t-K^1_t) dt.
\end{align*}
The reason the above is sufficient is that
there exist $K^i$ compatible with $H^i$ such that 
$\max_X(K^0_t-K^1_t) \le \max_X(\Sym(H^0_t)-\Sym(H^1_t))$ (cf. Remark \ref{r:flexK}).
More explicitly, let $V \supset \Delta$ be an open neighborhood disjoint from $\Sym(\phi^t_{H^1}(\ul{L}))$, $\Sym(\phi^t_{H^0}(\ul{L}))$ and $\Sym(\phi^t_{H^0-H^1}(\ul{L}))$.
We can find $K \in C^{\infty}([0,1] \times X)$ such that it is compatible with $H^0-H^1$, $K=\Sym(H^0-H^1)$ outside $V$
and $\max_X K_t \le \max_X \Sym(H^0_t-H^1_t)$ for all $t$.
Note that $\Sym(H^0-H^1)=\Sym(H^0)-\Sym(H^1)$ so for any $K^1$ compatible with $H^1$ such that $K^1=\Sym(H^1)$ outside $V$, $K^0:=K+K^1$ will be compatible with $H^0$ and we have $\max_X(K^0_t-K^1_t) \le \max_X(\Sym(H^0_t)-\Sym(H^1_t))$.

In turn, it suffices to find a continuation map (of the form \eqref{eq:ContinK})
$CF(\cE;X_{K^1}) \to CF(\cE;X_{K^0})$
which restricts to 
\begin{align*}
CF(\cE;X_{K^1})^{<a} \to CF(\cE;X_{K^0})^{<(a+b)} 
\end{align*}
for
\begin{align*}
b:=\int_0^1 \max_X(K^0_t-K^1_t)dt.
\end{align*}
As in the standard proof, we consider the homotopy of Hamiltonian functions 
\[
K^s:=K^0+\beta(s) (K^1-K^0)
\]
for a smooth function $\beta:\R \to [0,1]$ satisfying $\beta(s)=0$ for $s \le 0$ and $\beta(s)=1$ for $s \ge 1$.
Note that $K^s_t$ equals to an $(s,t)$-dependent constant near $\Delta$ for all $(s,t)$.

Suppose that $u$ is a solution contributing to the $(f_0,\hat{x}_0)$-coefficient of $\Phi_{K^0,K^1}(f_1,\hat{x}_1)$.

Then we have
\begin{align*}
&\cA^{\eta}_{K^0}(f_0,\hat{x}_0) - \cA^{\eta}_{K^1}(f_1,\hat{x}_1)\\
=&\int_{t=0}^1 \Sym(H^0_t)(x_1(t)) dt - \int \hat{x}_0^*\omega_X -\eta [\hat{x}_0] \cdot \Delta 
-\int_{t=0}^1 \Sym(H^1_t)(x_0(t)) dt + \int \hat{x}_1^*\omega_X +\eta [\hat{x}_1] \cdot \Delta \\
=&\int_{t=0}^1 K^0_t(x_0(t)) dt - \int \hat{x}_0^*\omega_X 
-\int_{t=0}^1 K^1_t(x_1(t)) dt + \int \hat{x}_1^*\omega_X -\eta [\hat{x}_0] \cdot \Delta +\eta [\hat{x}_1] \cdot \Delta \\
\le & b-\eta [\hat{x}_0] \cdot \Delta +\eta [\hat{x}_1] \cdot \Delta\\
=&b-\eta [u] \cdot \Delta
\end{align*}
where the inequality is obtained from the energy estimate in the standard proof (see, for example, \cite[Sec.\ 3.2]{LZ18}\footnote{They use a different set of sign conventions.})
and the last equality comes from the fact that $[\hat{x}_0]=[u]\#[\hat{x}_1]$.
Since $u$ is $J_X$-holomorphic near $\Delta$ and $\eta \ge 0$, we have $\eta [u] \cdot \Delta \ge 0$ so the result follows.

\end{proof}

Lemma \ref{lem:hofer_bound} guarantees that the invariant $c_{\cE}$ is Lipschitz.
Therefore for $H \in C^0([0,1] \times \Sigma)$, we can define $c_{\cE}(H)$ as the limit of $c_\cE(H_n)$
for a sequence of non-degenerate Hamiltonians $H_n$ which converge uniformly to $H$.

\begin{lemma}\label{lem:spectrality}
For any $H \in C^\infty([0,1] \times \Sigma)$, $c_{\cE}(H)$ belongs to the action spectrum of $\Sym(H)$.
\end{lemma}

\begin{proof}
Under the $\eta$-monotonicity assumption of Theorem \ref{t:spectral}, Lemma \ref{l:unobstructed} implies that 
\[
\left \{ \omega_X(u)+ \eta [u] \cdot \Delta \ : \ u \in \pi_2(X,\Sym{\ul{L}}) \right \} \ \subset \, \R
\] is a discrete subset of  $\R$.  
The spectrality of $c_{\cE}$ then follows from \cite[Proof of Theorem 27]{LZ18}.
\end{proof}

\begin{lemma}\label{lem:homotopy_invariance}
If $H^s$ is a family of mean normalized Hamiltonians such that $\phi^1_{H^s}$ is independent of $s$, then the action spectrum $\Spec(\Sym(H^s):\Sym(\ul{L}))$ is independent of $s$.  Hence,  $c_{\cE}(H^s)$ is independent of $s$.  
\end{lemma}

\begin{proof}
Let $(y_0,\hat{y}_0)$ be a critical point of the action functional $\cA^\eta_{H^0}$.
Let $(x_0,\hat{x}_0)=(\Sym(\phi^{t-1}_{H^0}))(y_0,\hat{y}_0)$
and let $u(s,t)=\Sym(\phi^t_{H^s})(x_0(0))$.
Let $\hat{x}_1=\hat{x}_0 \#u$.
It suffices to show that $\cA^\eta_{H^0}(\hat{x}_0)=\cA^\eta_{H^1}(\hat{x}_1)$.

Note that the image of $u$ is disjoint from $\Delta$.
As explained in \cite[Step 1, Proof of Lemma 32]{LZ18}, we have
\begin{align*}
\cA^\eta_{H^0}(\hat{x}_0)-\cA^\eta_{H^1}(\hat{x}_1)=\int_{[0,1]\times [0,1]} (\partial_s \Sym(H^s_t)) \circ u \, ds dt
\end{align*}
so it suffices to show that the right hand side vanishes. 
If $\WT{H^s_t}:=\Sym(H^s_t) \circ \pi:\Sigma^k \to \R$,  we have
\begin{align}
\int_{[0,1]\times [0,1]} (\partial_s \Sym(H^s_t)) \circ u \, ds dt=\int_{[0,1]\times [0,1]} (\partial_s \WT{H^s_t}) \circ \tilde{u} \, ds dt \label{eq:normalized}
\end{align} 
where, in the notation from \eqref{eqn:tautological_square},  $\tilde{u}$ is the restriction of the lift $V$ of $u$ to one of its $k!$ many connected components.
To show that the right hand side of \eqref{eq:normalized} vanishes, we follow the reasoning in \cite{LZ18}.
First, we want to show that
\[
\int_{[0,1]\times [0,1]} (\partial_s \WT{H^s_t} \circ \phi_{\WT{H^s}}^t) (p) ds dt
\]
is independent of $p \in \Sigma^k$.
The proof is exactly the same as \cite[Step 2, Proof of Lemma 32]{LZ18} because $\phi^1_{\WT{H^s}}$ is independent of $s$. Indeed, we are just considering the special case that the family of Hamiltonian functions is given by $\WT{H^s}$ on $\Sigma^k$.
Note that $\tilde{u}(s,t)=\phi^t_{\WT{H^s}}(\tilde{u}(0,0))$ so we can copy the computation in \cite[End of the proof, Proof of Lemma 32]{LZ18} 
\begin{align*}
0&=\int_0^1  \int_{\Sigma^k} (\WT{H^1_t}(p) -\WT{H^0_t}(p)) \omega_{\Sigma^k}^k dt\\
&=\int_{[0,1]\times [0,1]}  \int_{\Sigma^k} (\partial_s \WT{H^s_t} \circ \phi_{\WT{H^s}}^t) (p) \omega_{\Sigma^k}^k dt\\
&=(\cA^\eta_{H^0}(\hat{x}_0)-\cA^\eta_{H^1}(\hat{x}_1)  )\int_{\Sigma^k} \omega_{\Sigma^k}^k
\end{align*}
to conclude that $\cA^\eta_{H^0}(\hat{x}_0)-\cA^\eta_{H^1}(\hat{x}_1)=0$. Here $\omega_{\Sigma^k}=\omega \times \dots \times \omega$ is the symplectic form on $\Sigma^k$.

 Finally, Lemmas \ref{lem:hofer_bound} and \ref{lem:spectrality}, combined with the fact that the action spectrum is nowhere dense, imply that  $c_{\cE}(H^s)$ is independent of $s$. 
\end{proof}

We now define
\begin{equation} \label{eqn:normalise}
c_{\ul{L}} := (1/k) c_{\Sym(\ul{L})}
\end{equation}
(where $\ul{L}$ has $k$ components), 
recalling that the right hand side denotes $c_{\cE}$ with $\cE$ the trivial local system over $\Sym(\ul{L})$. 
Since, for $H^0,H^1 \in C^0([0,1] \times \Sigma)$, the maximum and minimum values of $\Sym(H^1)-\Sym(H^0)$ are exactly $k$ times the corresponding values for $H^1-H^0$, the normalisation \eqref{eqn:normalise} yields 
the Hofer continuity property 
\[ 
\int_0^1 \min (H_t - H'_t) \, dt \, \leq c_{\ul{L}}(H) - c_{\ul{L}}(H')  \leq \int_{0}^1 \max (H_t - H'_t) \, dt .
\]
Of the properties listed in Theorem \ref{t:spectral}, spectrality is an immediate consequence of Lemma \ref{lem:spectrality}, homotopy invariance follows from Lemma \ref{lem:homotopy_invariance}, monotonicity is a direct consequence of the Hofer Lipschitz property, and Lagrangian control, Support conntrol, and Shift propeties can be derived via formal, and standard, arguments from Lipschitz continuity and spectrality.  The remaining property, i.e.\ subadditivity,   can be proved by following the arguments in \cite{FOOOspectral}, \cite{LZ18}, using the compatible functions to reduce to the case of globally smooth Hamiltonians as in the proof of Lemma \ref{lem:hofer_bound}. More precisely, let $H$ and $H'$ be non-degenerate Hamiltonians.
For $\eps >0$, let $H''$ be a non-degenerate Hamiltonian whose $C^0$-distance with $H \# H'$ is less than $\eps$.
We can find Hamiltonians $K$, $K'$ and $K''$ compatible with $H$, $H'$ and $H''$ respectively such that 
the $C^0$-distance between $K''$ and $K \# K'$ is less than $2\eps$.
Now, as in \cite[Triangle inequality, Proof of Theorem 35]{LZ18},
for any solution $u$ contributing to the product
$CF(\cE,X_K) \times CF(\cE,X_{K'}) \to CF(\cE,X_{K''})$
with input $(f,\hat{y})$, $(f',\hat{y}')$ and output $(f'',\hat{y}'')$, we have
\begin{align*}
\cA^{\eta}_{K}(f,\hat{y})+\cA^{\eta}_{K'}(f',\hat{y}')-\cA^{\eta}_{K''}(f'',\hat{y}'') \ge -4\eps + \eta [u] \cdot \Delta.
\end{align*}
The non-negativity of $\eta [u] \cdot \Delta$ and the fact that the Floer product is compatible with PSS maps
imply that $c_{\cE}(K \# K') \le c_{\cE}(K)+c_{\cE}(K')$. This will in turn give the subadditivity property.  

\section{Closed-open maps and quasimorphisms}\label{sec:closed-open-quasimorphisms}
In this section we prove that, when $g=0$ or $\eta = 0$, our link spectral invariants coincide with the spectral invariants of the classical monotone Lagrangian Floer theory.  This allows us to use the closed-open map to obtain upper bounds for our link spectral invariant $c_{\cE}(H)$ in terms of the Hamiltonian Floer spectral invariant of $\Sym(H)$; see Corollary \ref{cor:requested_estimate}.  We then use this inequality, in Section \ref{sec:quasimorphisms}, to prove our results on quasimorphisms.

\subsection{Notation review}

If $(X,\omega)$ is any spherically monotone symplectic manifold, and $L\subset X$ is a (connected orientable and spin) monotone Lagrangian submanifold such that $HF(L,L) \neq 0$, there are classically constructed Hamiltonian and Lagrangian spectral invariants, cf. \cite{Oh05, LZ18}
\[
c(\bullet,\omega): C^0(S^1 \times X) \to \R, \qquad \ell(\bullet,\omega): C^0([0,1]\times X) \to \R.
\]

The values of these on a  non-degenerate $C^{\infty}$-Hamiltonian $H$ with time-one-flow $\varphi$ are defined by the infimal action values (with respect to the action functional associated to $H$) at which the unit lies in the image of the PSS maps (we use the notation $\Phi'$ to differentiate it from the PSS map $\Phi$ in \eqref{eqn:PSS})
\[
QH^*(X,\omega) \xrightarrow{\Phi_H'} HF^*(X,X_H)\cong HF^*(X,H) \text{ respectively } HF^*(L,L) \xrightarrow{\Phi_H'} HF^*(L,X_H) \cong HF^*(L,H)
\]
where $HF^*(-,X_H)$ denotes the cohomology of the cochain complex generated by Hamiltonian orbits/chords and with differential counting solutions to an $X_H$-perturbed equation, while $HF^*(-,H)$ denotes  the cohomology of the cochain complex generated by $\varphi$-fixed points with differential counting solutions to an unperturbed Cauchy-Riemann equation. (Compare to the notation introduced for \eqref{eqn:first_Floer_complex} and \eqref{eq:EHchi}) 

  We remark that the value of the spectral invariants $c(\bullet, \omega)$ and $\ell(\bullet, \omega)$ depends on the choice of the Novikov coefficients used in the constructions of the Floer complexes; we will work over the Novikov field $R=\C[[T]][T^{-1}]$ introduced in \eqref{eqn:coefficient_ring}, where the degree of the Novikov variable $T$ is the minimal Maslov number of the Lagrangian $L$; in the case of our Lagrangian $\Sym(\ul{L})$ the degree of $T$ will be $2$ (cf. Remark \ref{r:grading}).

One can always reparametrize $H \in C^\infty([0,1]\times X)$ to be 1-periodic without affecting $\ell(H, \omega)$ (the spectral invariant depends only on the homotopy class of the path of associated Hamiltonian diffeomorphisms with fixed end-points). With that understood, there is an inequality 
\begin{equation}\label{eqn:closed-open-inequality}
\ell(H,\omega) \leq c(H,\omega),
\end{equation}
cf. \cite[Proposition 5]{LZ18}. Recall that the closed-open map 
\[
\mathcal{CO}: HF^*(X,X_H) \to HF^*(L,X_H),
\]
counts discs with an interior `input' puncture and a boundary `output' puncture, which satisfy a perturbed Floer equation for which solutions have asymptotics given by a Hamiltonian orbit of $X_H$ at the interior puncture and a Hamiltonian chord at the output.  The same name is used for a map from quantum cohomology $QH^*(X)$ to Floer cohomology $HF^*(L,L)$, the latter defined in a `pearl' model, and defined by counting discs with interior and boundary marked points constrained to lie on appropriate cycles, and no Hamiltonian term in the Floer equation.  Equation \eqref{eqn:closed-open-inequality} is derived, for smooth $H$, from positivity of the energy of solutions to the closed-open map, the commutativity of
\[
\xymatrix{
QH^*(X) \ar[rr]^{\mathcal{CO}} \ar[d]^{\Phi_H'} & &  HF^*(L,L)  \ar[d]^{\Phi_H'}\\
HF^*(X,X_H) \ar[rr]^{\mathcal{CO}} &&  HF^*(L,X_H),
}
\]
and  the unitality of the algebra map $\mathcal{CO}: QH^*(X) \to HF^*(L,L)$ for any monotone Lagrangian $L\subset X$.   The inequality extends to all (non-smooth) $H$ by the Hofer Lipschitz property of spectral invariants.

\begin{lemma}\label{l:HamSame}
Let $H^0,H^1 \in C^{\infty}([0,1] \times X)$ be such that $\phi^t_{H^0}(L)=\phi^t_{H^1}(L)$ for all $t$ and 
$H^0=H^1$ in an open neighborhood containing $\cup_{t \in [0,1]}\phi^t_{H^i}(L)$.
Then $\ell(H^0, \omega)=\ell(H^1, \omega)$.
\end{lemma}

\begin{proof}
Let $H^s = (1-s) H^0 + s H^1$. We have $\phi^t_{H^s}(L)=\phi^t_{H^1}(L)$ for all $t,s \in [0,1]$ so $\spec(H^s: L)$ does not depend on $s$.
Since $\ell(H^s, \omega)$ depends continuously on $s$ and $\spec(H^s: L)=\spec(H^1: L)$ is nowhere dense, we conclude that $\ell(H^s, \omega)$ is independent of $s$.
\end{proof}

Via Lemma \ref{rmk:monotone_Lagrangian} and Remark \ref{rmk:floer_is_usual_floer}, we can use this theory to study the spectral invariant $c_{\ul{L}}$ defined by a Lagrangian link when $\eta=0$ or $g=0$.

\subsection{Link spectral invariants are monotone spectral invariants}\label{sec:link=monotone}

\emph{Throughout this section, we assume that $\eta=0$ or $g=0$.}
\medskip

Fix an open neighbourhood $V\supset \Delta$ in $\Sym^k(\Sigma)$, and (an inflation of) a Perutz-type form $\omega_{V,\eta}$.  
The manifold $(\Sym^k(\Sigma), \omega_{V,\eta})$ is spherically monotone, so there is a Hamiltonian spectral invariant
\[
c(\bullet; \omega_{V,\eta}): C^{0}(S^1 \times \Sym^k(\Sigma)) \to \R.
\]
Via the canonical embedding
\begin{equation} \label{eqn:take_symmetric_product}
C^0(\Sigma) \to C^0(\Sym^k(\Sigma)), \ H \mapsto \Sym(H)
\end{equation}
this defines a spectral invariant 
\[
c(\bullet,\omega_{V,\eta}): C^0(S^1 \times\Sigma) \to \R.
\]
Fix a $k$-component  $\eta$-monotone Lagrangian link $\ul{L} \subset \Sigma$ such that $\Sym(\ul{L}) \cap V=\emptyset$. As $\Sym(\ul{L}) \subset (\Sym^k(\Sigma), \omega_{V,\eta})$ is a monotone Lagrangian submanifold, there is a Lagrangian spectral invariant, which one can again restrict via \eqref{eqn:take_symmetric_product} to give
\[
\ell(\bullet, \omega_{V,\eta}): C^0([0,1]\times \Sigma) \to \R.
\]

  Consider the spectral invariant $c_{\Sym(\ul{L})}$ associated to the trivial local system over $\Sym(\ul{L})$, and its normalized cousin 
\[
c_{\ul{L}} = \frac1k\ c_{\Sym(\ul{L})}: C^0([0,1] \times \Sigma) \to \R
\]
from \eqref{eqn:normalise}. Let $H\in C^{\infty}([0,1]\times \Sigma)$ and $\phi_H^t$ denote the associated  Hamiltonian flow.

\begin{lemma}\label{lem:spectral_stabilises}
  Assume that $\Sym(\phi_H^t(\ul{L}))$ is disjoint from the closure of $V$ for $0\leq t\leq 1$. Then $c_{\Sym(\ul{L})}(H) = \ell(H,\omega_{V,\eta})$.
\end{lemma}

Before we proceed to the proof, it is helpful to recall from Remark \ref{r:diffComplex} the differences of the chain complexes.
The spectral invariant $c_{\ul{L}}(H)$ is computed using the chain complex $CF(\cE,\Sym(H))$, while $\ell(H,\omega_{V,\eta})$ is computed using a sequence of chain complexes $CF(\cE,X_{K_n}, \omega_{V,\eta})$ such that $\lim_n K_n=\Sym(H)$.


\begin{proof}
By Hofer-Lipschitz continuity, we can assume that $\Sym(\phi_H^1(\ul{L})) \pitchfork \Sym(\ul{L})$.

Let $K$ be a function compatible with $H$ and equal to a constant inside $V$. We have action-filtration-preserving isomorphisms of complexes (see \eqref{eq:Hchi})
\begin{equation} \label{eqn:two_filtered_models}
CF(\cE,\Sym(H)) \simeq CF(\cE,X_K)
\end{equation}
for any 
local system, and in particular for the trivial local system. 
Remark \ref{rmk:floer_is_usual_floer} identifies the complex on the LHS with $CF(\cE,\Sym(H),\omega_{V,\eta})$.

The invariant $c_{\Sym(\ul{L})}$ is defined by the PSS map (see \eqref{eqn:PSS})
\begin{equation} \label{eqn:PSS_for_SymH}
\Phi_K: CF(\cE,\cE) \to CF(\cE,X_K). 
\end{equation}
On the other hand, we have the classical PSS map with respect to the symplectic form $\omega_{V,\eta}$
\begin{equation} \label{eqn:PSS_for_SymH2}
\Phi_K': CF(\cE,\cE) \to CF(\cE,X_K, \omega_{V,\eta}). 
\end{equation}
and $CF(\cE,X_K, \omega_{V,\eta})$ is action-preserving isomorphic to $CF(\cE;\Sym(H),\omega_{V,\eta})$ as in \eqref{eq:Hchi}.

Since $\phi_K^t$ is supported away from $V$ and $\omega_{V,\eta}=\omega_X$ outside $V$,
the two PSS maps commute with the isomorphism $CF(\cE,X_K) \simeq CF(\cE,\Sym(H))  \simeq CF(\cE,\Sym(H),\omega_{V,\eta}) \simeq CF(\cE,X_K, \omega_{V,\eta})$.
Therefore, $c_{\Sym(\ul{L})}(H)$ is exactly computing $\ell(K,\omega_{V,\eta})$. 

The invariant $\ell(H,\omega_{V,\eta})$ is defined by choosing a sequence of smooth functions $K_n \in C^{\infty}([0,1]\times \Sym^k(\Sigma))$ which $C^0$-approximate the Lipschitz function $\Sym(H) \in C^0([0,1]\times \Sym^k(\Sigma))$, and taking the limit of the $\ell(K_n,\omega_{V,\eta})$.
We can take all the $K_n$ to co-incide with $K$ in a fixed open set containing the Lagrangian isotopy $\phi_H^t(\Sym(\ul{L}))$. 
By Lemma \ref{l:HamSame}, $\ell(K_n,\omega_{V,\eta})=\ell(K,\omega_{V,\eta})$ so the result follows.
\end{proof}

The inequality in the corollary below is crucial to the arguments of the following section.  It follows immediately from Lemma \ref{lem:spectral_stabilises} combined with inequality \eqref{eqn:closed-open-inequality}, which, as we explained, holds for not necessarily time-periodic $H$.

 Fix a sequence of open neighborhoods $\cdots \supset V_n \supset V_{n+1} \supset \cdots$ which shrink to $\Delta$.

\begin{corol} \label{cor:requested_estimate}
For any $H \in C^\infty([0,1]\times \Sigma)$, there is $N(H) > 0$ for which 
\[
c_{\ul{L}}(H) \leq \frac1k c(H,\omega_{V_n,\eta})
\]
 for all $n > N(H)$.
\end{corol}

\begin{remark} \label{rmk:cohomology_class_after_inflation}
We remark that $\omega_{V_n,\eta}([\PP^1])=(k+1)\lambda$, where $[\PP^1]$ is the positive generator of $H_2(X,\Z)$ and $\lambda$ is the monotonicity constant (see Definition \ref{def:monotone_link}). In particular, when $\eta=0$, we have $\omega_{V_n,\eta}([ \PP^1])=1$, assuming $\omega$ gives $\PP^1$ total area $1$. 
\end{remark}

If $\Sigma = \PP^1$,  the forms $\omega_{V,\eta}$ can be scaled to be isotopic, so the quantum cohomology $QH^*(\Sym^k(\Sigma),\omega_{V,\eta})$ is independent (up to $R$-algebra isomorphism) of the choice of $V$ and $\eta$. 
Recalling that we are working over $R = \C[[T]][T^{-1}]$, one has $QH^*(\PP^k,\omega_{V,\eta}) = R[x]/(x^{k+1}-T^{k+1})$.

\begin{lemma}\label{rem:upperbound-quasimorphisms}
The spectral invariant $c(\bullet,\omega_{V,\eta}): C^0(S^1\times\PP^k) \to \R$ satisfies the following inequality for any (continuous) Hamiltonian $H$:

\begin{equation}\label{eqn:EP-inequality}
c(H,\omega_{V,\eta}; R) + c (\bar H,\omega_{V,\eta}; R) \leq \omega_{V_n,\eta}([\PP^1])=(k+1)\lambda,
\end{equation}
where $\bar H(t,x):= -H(1-t,x).$  Here, we are writing $c(\bullet,\omega_{V,\eta}; R)$ instead of $c(\bullet ,\omega_{V,\eta})$ to emphasize the choice of the Novikov field $R=\C[[T]][T^{-1}]$, which is important for what follows.
\end{lemma}
\begin{proof}
We explain how the above can be deduced from a similar inequality proven in \cite{Entov-Polterovich}.  Let $\hat R$ denote the Novikov field

$$\hat R=\C[[S]][S^{-1}]=\left \{ \sum_{i=0}^{\infty} a_i S^{b_i} | a_i \in \C, b_i \in \Z, b_0 < b_1< \dots\right \},$$ 
where the variable $S$ has degree $2(k+1)=2c_1(\PP^k)[\PP^1]$.

Denote by $c(\bullet,\omega_{V,\eta}; \hat{R}): C^0(S^1\times\PP^k) \to \R$ the Hamiltonian Floer spectral invariant constructed with the field $\hat R$ as the choice of Novikov coefficients.  It follows from \cite[Section 3.3]{Entov-Polterovich} (see also \cite[Example 12.6.3]{Polterovich-Rosen}) that there exists a constant $D >0$ such  that 

\begin{equation}\label{eqn:EP-inequality-original}
 c(H,\omega_{V,\eta}; \hat R) + c (\bar H,\omega_{V,\eta}; \hat R) \leq D.  
\end{equation}
The proof of this inequality relies on the fact that, with $\hat R$ coefficients, we have $$QH^*(\PP^k,\omega_{V, \eta}; \hat R) =  \hat{R}[x]/(x^{k+1} - S)$$
which is a field; see \cite[Example 12.1.3]{Polterovich-Rosen}\footnote{We use the cohomological convention while \cite{Entov-Polterovich} use the homological convention. Therefore, even though our $S$ corresponds to their $s^{-1}$,
the degree of $S$ is $2(k+1)$ in our convention but the degree of $s^{-1}$ is $-2(k+1)$ in their convention.}.  Although not explicitly stated in \cite{Entov-Polterovich}, the arguments therein imply that 

\begin{equation}\label{eqn:D-upperbound}
D \leq \omega_{V_n,\eta}([\PP^1])=(k+1)\lambda.
\end{equation}
This upper bound on $D$ is not essential to our main results and is only used below in the proof of Proposition \ref{prop:quasicalabi}.

\medskip

Now, there exists an embedding of fields $\hat R \hookrightarrow R$, induced by $S \mapsto T^{k+1}$,  which  in turn induces the injective maps
$QH^*(\PP^k,\omega_{V, \eta}; \hat R) \hookrightarrow QH^*(\PP^k,\omega_{V, \eta}; R)$ and (when $H \in C^{\infty}(S^1 \times X)$ is non-degenerate) $  HF^*(X, H; \hat R) \hookrightarrow  HF^*(X, H;  R)$, fitting into the commutative diagram 
\[
\xymatrix{
QH^*(\PP^k,\omega_{V, \eta}; \hat R) \ar[rr]^{\Phi'_H} \ar@{^{(}->}[d] & &  HF^*(X, H; \hat R)  \ar@{^{(}->}[d] \\
QH^*(\PP^k,\omega_{V, \eta}; R) \ar[rr]^{\Phi'_H} &&  HF^*(X, H; R),
}
\]
where the horizontal arrows denote the corresponding PSS maps.  
Since the map $  HF^*(X, H; \hat R) \hookrightarrow  HF^*(X, H;  R)$ respects the action filtration, we have  $$c(H,\omega_{V,\eta};  R) \leq c(H,\omega_{V,\eta}; \hat R),$$  which proves \eqref{eqn:EP-inequality}.    

In fact, since the vertical arrows in the diagram are injective and are given by $- \otimes_{\hat R} R$, we can further conclude that it preserves the action (not only the action filtration) and hence $c(H,\omega_{V,\eta};  R) = c(H,\omega_{V,\eta}; \hat R).$
\end{proof}

\subsection{Quasimorphisms on $S^2$}\label{sec:quasimorphisms}
We will now use the contents of the previous section to prove our results on quasimorphisms, namely  Theorems \ref{corol:quasi_homeo} and Theorem \ref{theo:quasimorphisms}.  These will be immediate consequences of Theorems \ref{theo:quasimorphisms2} and \ref{theo:quasi_independence}; see Remark \ref{rem:proof_quasimorphism_results} below.  It will be convenient for the remainder of the paper to fix $S^2 = \lbrace x^2 + y^2 + z^2 = 1 \rbrace \subset \R^3$, with its  standard area form scaled to have area $1$.

\medskip

Recall that $c_{\ul{L}} : \widetilde{\Ham}(S^2, \omega) \rightarrow \R$ is defined by $c_{\ul{L}}(\tilde\varphi):= c_{\ul{L}}(H)$, where $H$ is any mean-normalized Hamiltonian whose flow represents $\tilde \varphi $.  For $\varphi \in \Ham(S^2, \omega)$ we define the homogenization 
\begin{equation}\label{eqn:homogenization}
\mu_{\ul{L}}(\varphi) := \lim_{n \to \infty} \frac{c_{\ul{L}}(\tilde\varphi^n)}{n},
\end{equation}
where $\tilde \varphi \in \widetilde{\Ham}(S^2, \omega)$ is any lift of $\varphi$.  The limit (\ref{eqn:homogenization}) exists in $\{-\infty\}\cup\R$ since the sequence $(c_{\ul{L}}(\tilde\varphi^n))$ is subadditive. Now, Hofer continuity implies that the sequence $(\frac{c_{\ul{L}}(\tilde\varphi^n)}{n})$ is bounded and so we see that the limit exists.  Moreover, the limit depends only on $\varphi$, and not the lift $\tilde{\varphi}$ because the fundamental group of $ \Ham(S^2, \omega)$ has finite order; see \cite[Prop. 3.4]{Entov-Polterovich}.

\begin{theo}\label{theo:quasimorphisms2}
For any monotone Lagrangian link $\ul{L}$, the map  $$\mu_{\ul{L}}: \Ham(S^2, \omega) \rightarrow \R$$
is a homogeneous quasimorphism with the following properties:

\begin{enumerate}
\item (Hofer Lipschitz)  $\vert \mu_{\ul{L}}(\varphi) - \mu_{\ul{L}}(\psi) \vert \leq d_{H}(\varphi, \psi)$;

\item (Lagrangian control) Suppose $H$ is mean-normalized.  If $H_t|_{L_i} = s_i(t)$ for each $i$, then $$\mu_{\ul{L}}(H) = \frac{1}{k} \, \sum_{i=1}^k \int_0^1 s_i(t) dt.$$
Moreover, $$ \frac{1}{k} \sum_{i=1}^{k} \int_0^1 \min_{L_i} H_t \, dt \leq \mu_{\ul{L}}(H) \leq \frac{1}{k} \sum_{i=1}^{k} \int_0^1 \max_{L_i} H_t \, dt.$$

\item (Support control) If  $\mathrm{supp}(\varphi) \subset S^2 \backslash \cup_j L_j$, then $\mu_{\ul{L}}(\varphi) = - \Cal(\varphi)$.
\end{enumerate}
\end{theo}

The next theorem tells us how the quasimorphisms $\mu_{\ul{L}}$ are related to each other.

\begin{theo}\label{theo:quasi_independence}
\begin{enumerate}[(i)]
\item Suppose that $\ul{L}, \ul{L}'$ are $\eta$-monotone links in $S^2$ which have the same number of components $k$.  Then, the quasimorphisms $ \mu_{\ul{L}}$ and $\mu_{\ul{L}'}$ coincide and we denote by $\mu_{k,\eta}$ their common value.
\item Let $\{ \mu_i\}$, $i\in I$ be a family of quasimorphisms of the form $\mu_{k,\eta}$ associated to links whose respective monotonicity constants $\lambda_i$ (See Def. \ref{def:monotone_link}) are pairwise distinct. Then the family $\{ \mu_i\}$ is linearly independent.
\item The difference $\mu_{k, \eta} - \mu_{k', \eta'}$ is $C^0$ continuous and extends continuously to $\Homeo_0(S^2, \omega)$.
\end{enumerate}
\end{theo}

By Remark~\ref{rmk:eta_values}, the set of all values of $(k,\eta)$ in Theorem~\ref{theo:quasi_independence} is exactly $$ \left\lbrace (k,\eta) : k\in\N_{\ge 2}, \eta\in [0,\frac14)\right\rbrace \cup \{(1,0)\}.$$


\begin{remark}\label{rem:proof_quasimorphism_results} The family of quasimorphisms $\{\mu_{k,0}\}$ satisfies the condition of item (ii), hence is linearly independent. It follows that the family of quasimorphisms $\{\mu_{k, 0} - \mu_{k', 0} \}$ satisfies the conclusions of Theorems \ref{corol:quasi_homeo} and  \ref{theo:quasimorphisms}.    
\end{remark}

We also remark that by combining these results with our Theorem~\ref{theo:recovering_Calabi}, on $S^2$, we can extend the Calabi property from Theorem \ref{theo:recovering_Calabi} to more general links, for example equally spaced horizontal links on $S^2$, as studied in \cite{CGHS20, Pol-Shel21}.   The precise statement is as follows. 

\begin{prop} 
\label{prop:quasicalabi}

Let $\ul{L}_k$ be any sequence of $k$-component monotone links in $S^2$ with $\eta_k < \frac1{2k(k-1)}$. Then, for any $H$ we have
$$c_{\ul{L}_k}(H) \to \int_0^1 \int_{S^2} H_t \omega \; dt$$
and for any $\phi$ we have
\[ \mu_{k,\eta_k}(\phi) \to 0.\] 

\end{prop}

We now prove the results stated above. 

\medskip

Recall that we denote by $\lambda$ the monotonicity constant of the link $\ul{L}$; see Definition \ref{def:monotone_link}.  The following lemma will be useful.

\begin{lemma}\label{lemma:monotonicity-constant} Let $\ul{L}$ be an $\eta$-monotone link on $S^2$ with $k$ components. Then, the value of $\lambda$ is given by 
  \begin{equation}
  \label{eqn:moncons}
    \lambda=\frac{1+2\eta(k-1)}{k+1}.
  \end{equation}
  \end{lemma}
  \begin{proof} First note that by induction on $k$, the number of components of $S^2\setminus\ul{L}$ is $k+1$. 
Recall also the $B_j, \tau_j, A_j$ with $j\in\{1,\dots, k+1\}$ from Theorem~\ref{t:spectral}. 

Now, by the definition of the monotonicity constant, we have $\lambda=A_j+2\eta(\tau_j-1)$ for each $j$. 
Summing over $j\in\{1,\dots, k+1\}$ and using the fact that $\sum A_j=\mathrm{area}(S^2)=1$, we get  
\[ (k+1)\lambda=1+2\eta(2k-(k+1)),\] 
hence $\lambda(k+1)=1+2\eta(k-1)$ as claimed. 
\end{proof}

\begin{proof}[Proof of Theorem \ref{theo:quasimorphisms2}]
The Hofer Lipschitz, Lagrangian control and Support control properties are inherited from Theorem~\ref{t:spectral}, so it remains to prove the quasimorphism property.

Let $k$ denote the number of components in the monotone link $\ul{L}$ and denote by $\lambda$ the monotonicity constant of the link; see \eqref{eqn:independent_of_j}.
We will prove that  
\begin{equation}\label{eqn:quasi_bound}
c_{\ul{L}}(H) + c_{\ul{L}}(\bar{H}) \leq  \frac{k+1}{k} \lambda,
\end{equation}
where $\bar{H}$ is the Hamiltonian 
\[ \bar{H} = - H(1-t,x).\]
By \cite[Theorem 1.4]{usher11}, this implies\footnote{In \cite{usher11}, the author uses a different definition of $\bar{H}$, namely that $\bar{H} = - H(t,\phi^t_H(x)).$ However, since both definitions for $\bar{H}$ determine the same element in $\widetilde{\Ham}(S^2, \omega)$, one can still cite \cite{usher11}.} that $\mu_{\ul{L}}$ is a homogeneous quasimorphism  with defect $2 \frac{k+1}{k}\lambda $.

By Corollary \ref{cor:requested_estimate}, for every Hamiltonian $H$ on $S^2$, there exists a family of spectral invariants $c(H,\omega_{V,\eta}) $ with the property that
\begin{equation}
\label{eq:openclosedupper} 
c_{\ul{L}}(H) \leq \frac
{1}{k} c(H,\omega_{V,\eta}).
\end{equation} 
Recall that  $\omega_{V,\eta}$ is a K\"ahler form on $\PP^k$  symplectomorphic to the standard Fubini-Study form $\omega_{FS}$, where that form is normalized so that the symplectic area of $[\PP^1]$ is $(k+1)\lambda$, cf. Remark \ref{rmk:keep_eta}.

According to Lemma \ref{rem:upperbound-quasimorphisms}, for any  $F\in C^0([0,1] \times \PP^k$ we have the inequality
\begin{equation}
\label{eqn:input}
c(F,\omega_{V, \eta}) + c(\bar{F},\omega_{V, \eta}) \leq  D = (k+1) \lambda.
\end{equation}

Taking $F = \Sym^k(H)$ and noting that  $\bar{F} = \Sym^k(\bar H)$ we obtain 
\begin{equation}
  c(H,\omega_{V,\eta}) + c(\bar H,\omega_{V,\eta}) \le (k+1) \lambda.\label{eq:input2}
  \end{equation}
Equation \eqref{eqn:quasi_bound} follows by applying \eqref{eq:openclosedupper}.

\end{proof}

We next prove Theorem \ref{theo:quasi_independence}.

\begin{proof}[Proof of Theorem \ref{theo:quasi_independence}.] We begin with the proof of part (i).  Since the links $\ul{L}$ and $\ul{L}'$ are both $\eta$-monotone for the same  $\eta$, by Corollary \ref{cor:requested_estimate}, we have for any Hamiltonian $H$ a Perutz-type form $\omega_{V,\eta}$ such that
  \[c_{\ul{L}}(H)\leq \frac1k c(H,\omega_{V,\eta}), \quad c_{\ul{L}'}(\bar H) \leq \frac1k c(\bar H;  \omega_{V,\eta}).\]
 We can apply these inequalities in combination with (\ref{eq:input2}) to obtain for every Hamiltonian $H$:
  \begin{align*}
    c_{\ul L}(H)+ c_{\ul L'}(\bar H)\leq \frac1k \left(c(H ,\omega_{V,\eta})+ c(\bar H,\omega_{V,\eta})\right)\leq \frac{k+1}{k} \lambda.
  \end{align*}
  In view of this $H$-independent upper bound, whose precise value is not relevant here, it follows that after homogenization we have 
  \[\mu_{\ul L}(\varphi)-\mu_{\ul L'}(\varphi)\leq 0. \]
  Switching the roles of $\ul L$ and $\ul L'$, we deduce that $\mu_{\ul L}=\mu_{\ul L'}$.

  We now turn to the proof of part (ii) of the theorem. 
  For this purpose, let $\lambda_1< \dots< \lambda_n$ and $\mu_1, \dots, \mu_n$ be quasimorphisms obtained from Lagrangian links $\ul L_1, \dots, \ul L_n$ whose monotonicity constant are respectively $\lambda_1, \dots, \lambda_n$. 
  Assume we have
  \begin{equation}
\sum_{i=1}^na_i\mu_i=0,\label{eq:linear-indep}
\end{equation}
for some real numbers $a_1, \dots, a_n$.  We will show by induction on $n$ that all $a_i$ vanish; 
the base case $n=1$ follows from the Lagrangian control property from Theorem~\ref{t:spectral}.
For the inductive step, by part (i) of this theorem, we may assume without loss of generality that each $\ul L_i$ consists of parallel horizontal circles. Then the bottom circle $C_i$ of $\ul{L_i}$ bounds a disc of area $\lambda_i$.  Consequently, $C_1$ is disjoint from all the components of the remaining links $\ul L_2, \ldots, \ul L_n$.  Now
let $\varphi$ be a Hamiltonian diffeomorphism generated by a mean normalized Hamiltonian $H$ which is supported in a small neighborhood of $C_1$ and such that the restriction of $H$ to $C_1$ is constant and equal to $1$. 
We choose this neighborhood small enough so that the support of $H$ does not intersect any of the components of the $L_i$ for $i\geq2$. As a consequence, by the Support control property from Theorem~\ref{theo:quasimorphisms2}, we have $\mu_i(\varphi)=0$, for $i\geq 2$, and by the Lagrangian control property of the same theorem we have  $\mu_1(\varphi)=1$. 
 Thus, (\ref{eq:linear-indep}) yields $a_1=0$, and then by induction we deduce  
that all $a_i$ vanish, hence item (ii).


As for the third item, its proof is very similar to that of Proposition \ref{prop:C0-cont1} and so we will not present it; it can also be proven via the arguments given in \cite{EPP}.  
\end{proof}

We conclude with the promised proof of our result concerning recovering Calabi for more general links.

\begin{proof}[Proof of Proposition~\ref{prop:quasicalabi}]

By the Shift property from Theorem~\ref{t:spectral}, it suffices to assume that $H$ is mean-normalized and then show that both limits are zero.  So, assume this.  Write $\varphi = \phi^1_H.$

As in the proof of Theorem~\ref{theo:quasimorphisms2}, each $c_{\ul{L}_k}$ is a quasimorphism with defect given by  $$D_k = \frac{k+1}{k}\lambda_k,$$ where  $\lambda_k$ denotes the monotonicity constant of the link $\ul{L}_k$.  Hence
\begin{equation}
\label{eqn:defect}
|  c_{\ul{L}_k}(H) - \mu_{\ul{L}_k}(\varphi) | =  | c_{\ul{L}_k}(H) - \mu_{k, \eta}(\varphi) | \le D_k,
\end{equation}
since the first equality here holds by the first part of Theorem~\ref{theo:quasi_independence} above.  By \eqref{eqn:moncons} and the assumption on $\eta_k$, we have that $D_k$ tends to $0$ with $k$.  Assume first that the $\ul{L}_k$ are equidistributed; we can find such a sequence via Example \ref{ex:equidistrib-on-S2}.  Then by Theorem~\ref{theo:recovering_Calabi},  $c_{\ul{L}_k}(H)$ converges to $0$, hence by \eqref{eqn:defect} the sequence $\mu_{k, \eta}(\varphi)$ does as well.   It now follows in addition, again applying \eqref{eqn:defect}, that $c_{\ul{L}_k}(H)$ converges to $0$ without the assumption that the links are equidistributed.  
\end{proof}

\subsection{The commutator and fragmentation lengths}
\label{sec:comm}

We collect here some final applications of our new quasimorphisms.

To start, as illustrated in Example~\ref{ex:comm}, our quasimorphisms can be used to deduce a result about the commutator length on $\Homeo_0(S^2,\omega)$ that contrasts the situation for $\Homeo_0(S^2)$.  Here is a result in a similar vein.  It has recently been shown in \cite[Thm.\ 5.5]{Bowden} that for the group $\Homeo_0(\Sigma_g)$ of homeomorphisms of a closed 
surface in the component of the identity, 
the stable commutator length is $C^0$ continuous.

\begin{prop}
\label{prop:notcont}
The stable commutator length on $\Homeo_0(S^2,\omega)$ is unbounded in any $C^0$ neighborhood of the identity.  In particular, it is not $C^0$ continuous on $\Homeo_0(S^2,\omega).$ 
\end{prop}

In a different direction, recall the {\bf quantitative fragmentation norm} $|| \cdot ||_A$ on $\Homeo_0(S^2,\omega)$  associated to a positive real number $A$: $|| \psi ||_A$ is the minimum $N$ such that $\psi = f_1 \ldots f_N$, where the $f_i$ are supported in open discs of area no more than $A$.  In applications of fragmentation, one often assumes in addition that the discs are displaceable, in other words that $A < 1/2$.  For more about fragmentation norms, we refer the reader to (for example) \cite{Entov-Polterovich,bip}.  

In contrast to Proposition~\ref{prop:notcont}, one expects that the quantitative fragmentation norm is bounded in a $C^0$ neighborhood of the identity.  Indeed, this fact is known for diffeomorphisms by combining \cite[Prop. 3.1]{Sey13} with \cite[Lem. 4.7]{LSV} and one should be able to adapt these proofs without difficulty for homeomorphisms; it should also be noted that we actually use this boundedness, for diffeomorphisms, in our proof of Proposition~\ref{prop:C0-cont1} because \cite[Prop. 3.1]{Sey13} and \cite[Lem. 4.7]{LSV} are used in \cite[Lem. 3.11]{CGHS21}.   Nevertheless, we can prove that, just as with the stable commutator length, the fragmentation norm is unbounded (for $A< \frac{1}{2}$). Clearly, elements of $\Homeo_0(S^2,\omega)$ with large fragmentation norm must be $C^0$-far from the identity.

\begin{prop}
\label{prop:frag}
When $A < \frac{1}{2}, || \cdot ||_A$ is unbounded.
\end{prop}

We remark that in \cite[Ex. 1.24]{bip}, the authors show that the quantitative fragmentation norm is unbounded on displaceable subsets of tori and raise the question of what happens on a complex projective space.  Proposition~\ref{prop:frag} gives a partial answer to this: the quasimorphism we construct in the course of proving Proposition~\ref{prop:frag} shows that $|| \cdot ||_A$ is unbounded on $\Ham(S^2,\omega)$ for $A < 1/2$, since this is a subgroup of  $\Homeo_0(S^2,\omega)$.

\begin{proof}[Proof of Proposition~\ref{prop:notcont}]

Choose Hamiltonians $H_n: S^2 \to \R$ for $n \ge 2$, depending only on $z$, such that:
\begin{itemize}
\item $H_n|_{z = -1 + 1/n} =  n$,
\item $\text{supp}(H_n) \subset \lbrace -1 \le z \le -1 + \frac{1.5}{n} \rbrace,$
\item  $\int_{S^2} H_n \omega = 0.$
\end{itemize}
Then $\varphi^1_{H_n}$ is $C^0$ converging to the identity.  Moreover, since $\Ham(S^2,\omega)$ is perfect, each $\varphi^1_{H_n}$ is in the commutator subgroup of $\Homeo_0(S^2,\omega)$;  however, we will show that the stable commutator length in $\Homeo_0(S^2,\omega)$ of $\varphi^1_{H_n}$ is diverging. 

 To see this, we consider the family of quasimorphisms $f_n : = \mu_{\ul{L}_1} - \mu_{\ul{L}_{n}}$, where $\ul{L}_1$ is the link $\lbrace z = 0 \rbrace$ as above, and $\ul{L}_{n}$ is the link consisting of the circles  $\lbrace z = -1 + k/n \rbrace$ for $1 \le k \le 2n-1.$ Since the $f_n$ are homogeneous quasimorphisms, we have   
 \[scl(\varphi^1_{H_n}) \ge \frac{\vert f_n(\varphi^1_{H_n}) \vert}{D(f_n)},\]
 where $D(f_n)$ denotes the defect of $f_n$.
Now,  
 by the Lagrangian control property of Theorem~\ref{t:spectral}, we have that $f_n(\varphi^1_{H_n}) = n$;  on the other hand, as in the proof of Theorem~\ref{theo:quasimorphisms2}, the quasimorphism associated to an $\eta$-monotone link with $k$ components has defect $2\frac{1 + 2 \eta(k-1)}{k},$ our links are $\eta$-monotone with $\eta = 0$, and so it follows that $D(f_n)$ remains bounded,  as $n \to \infty$. 
 We therefore conclude that $scl(f_n) \to \infty$ although $f_n \xrightarrow{C^0} \id$.  
\end{proof}

\begin{proof}[Proof of Proposition~\ref{prop:frag}]
The proposition is an immediate consequence of the fact that we can construct a non-trivial homogeneous quasimorphism that vanishes on any map supported on a disc of area $A$.  To construct such a quasimorphism, let $\ul{L}_2$ denote the monotone link consisting of two horizontal circles so close to the equator $\lbrace z = 0 \rbrace$ that they are disjoint from the disc of area $A$ bounded by a horizontal circle in the southern hemisphere and let $\ul{L}_1$ denote the one-component link consisting of the equator itself.  Then, by  the Lagrangian control property of Theorem~\ref{t:spectral} and Theorem~\ref{theo:quasi_independence},  $\mu_{\ul{L}_2} - \mu_{\ul{L}_1}$ vanishes on any map supported in a disc of area $A$.  
\end{proof}

%
%
 
\bibliographystyle{abbrv}
\bibliography{biblio}

{\small

\medskip
 \noindent Dan Cristofaro-Gardiner\\
\noindent Department of Mathematics, University of Maryland, 4176 Campus Dr.\, College Park, MD 20742-4015, USA. \\
{\it e-mail}: dcristof@umd.edu
\medskip

\medskip
\noindent Vincent Humili\`ere \\
\noindent Sorbonne Universit\'e and Universit\'e de Paris, CNRS, IMJ-PRG, F-75006 Paris, France\\
\noindent \& Institut Universitaire de France.\\
{\it e-mail:} vincent.humiliere@imj-prg.fr
\medskip

\medskip
\noindent Cheuk Yu Mak\\
\noindent School of Mathematics, University of Edinburgh, James Clerk Maxwell Building, Edinburgh, EH9 3FD, U.K.\\
{\it e-mail:} cheukyu.mak@ed.ac.uk

\medskip
 \noindent Sobhan Seyfaddini\\
\noindent Sorbonne Universit\'e and Universit\'e de Paris, CNRS, IMJ-PRG, F-75006 Paris, France.\\
 {\it e-mail:} sobhan.seyfaddini@imj-prg.fr

 \medskip
 \noindent Ivan Smith\\
\noindent Centre for Mathematical Sciences, University of Cambridge, Wilberforce Road, CB3 0WB, U.K.\\
{\it e-mail:} is200@cam.ac.uk

}
 
\end{document}